\newif\ifdraftmode
\draftmodefalse
\documentclass[10pt]{article}

\ifdraftmode
\fi

\usepackage{arxiv}

\usepackage[utf8]{inputenc} %
\usepackage[T1]{fontenc}    %
\usepackage{hyperref}       %
\usepackage{url}            %
\usepackage{amsfonts}       %
\usepackage{nicefrac}       %
\usepackage{microtype}      %
\usepackage{dsfont}         %
\usepackage{graphicx}
\usepackage[numbers]{natbib}         %
\usepackage{doi}

\usepackage{amsmath}             %
\usepackage{amssymb}             %
\usepackage{amsthm}              %
\usepackage{mathtools}           %
\usepackage[dvipsnames]{xcolor}  %
\usepackage{enumitem}            %
\usepackage{newtxtext,newtxmath} %
\usepackage[ruled]{algorithm2e}  %
\usepackage{caption}             %
\usepackage{subcaption}          %

\usepackage{pgf,tikz}
\usetikzlibrary{positioning}
\usetikzlibrary{calc}
\usetikzlibrary{intersections}
\usetikzlibrary{decorations.pathreplacing}  %

\usepackage[capitalise,noabbrev]{cleveref} %
\usepackage{autonum}                       %

\theoremstyle{plain}
\newtheorem{theorem}{Theorem}[section]
\newtheorem{proposition}[theorem]{Proposition}
\newtheorem{corollary}[theorem]{Corollary}
\newtheorem{lemma}[theorem]{Lemma}

\theoremstyle{definition}
\newtheorem{definition}[theorem]{Definition}
\newtheorem{remark}[theorem]{Remark}
\newtheorem{example}[theorem]{Example}
\newtheorem{counterexample}[theorem]{Counterexample}
\newcommand{\proofappendix}[1]{\hyperref[#1]{\hspace{\stretch{1}}{\raisebox{0.3ex}{\setlength{\fboxsep}{0.2ex}\fbox{\hspace{0.1ex}\tiny Proof in \cref*{#1}.\hspace{0.2ex}}}}}}
\newcommand{\qedappendix}{\hfill$\blacksquare$}

\newlist{thmenum}{enumerate}{1}  %
\setlist[thmenum]{label=\thethmenumi., ref=\thetheorem.\thethmenumi}  %
\crefalias{thmenumi}{theorem}

\newlist{corenum}{enumerate}{1}
\setlist[corenum]{label=\thecorenumi., ref=\thecorollary.\thecorenumi}
\crefalias{corenumi}{corollary}

\theoremstyle{definition}
\newlist{dpenum}{enumerate}{1}
\setlist[dpenum]{label=(DP\arabic*), ref=(DP\arabic*)} %
\crefname{dpenumi}{design principle}{design principles}
\Crefname{dpenumi}{Design principle}{Design principles}

\newtheorem{precondition}[theorem]{Precondition}

\ifdraftmode
    \definecolor{amaranth}{rgb}{0.9, 0.17, 0.31}%
    
    \newcommand{\todo}[1]{\marginpar{\tiny\color{BurntOrange}#1}}
    \newcommand{\note}[1]{\marginpar{\tiny\color{MidnightBlue}#1}}
    
    \usepackage[notref,notcite]{showkeys}
    
\else
    \newcommand{\todo}[1]{}
    \newcommand{\note}[1]{}
\fi

\tikzset{core/.style={inner sep=0pt}}
\tikzset{contraction/.style={line width=0.75}}
\tikzset{contractionDots/.style={contraction, dotted}}

\def\Core#1{%
    \fill[black] #1 circle [radius=0.1];
}
\def\Orth#1#2{%
    \fill[black] #1 -- ++(#2:0.1) arc (#2:#2+180:0.1) -- cycle;
    \draw[black, fill=white] #1 -- ++(#2+180:0.1) arc (#2+180:#2+360:0.1) -- cycle;
}

\newcommand{\indep}{\perp\kern-0.6em\perp}

\newcommand*{\mbb}[1]{\mathbb{#1}}
\newcommand*{\mcal}[1]{\mathcal{#1}}
\newcommand*{\mfrak}[1]{\mathfrak{#1}}
\newcommand*{\dd}{\ensuremath{\mathrm{d}}}
\newcommand*{\dx}[1][x]{\ensuremath{\,\dd{#1}}}

\mathchardef\mhyphen="2D

\newcommand{\bbone}{\mathds{1}}

\DeclareMathOperator{\supp}{supp}

\DeclareMathOperator{\rch}{rch}
\DeclareMathOperator{\cl}{cl}

\let\inf\relax  %
\DeclareMathOperator*{\inf}{inf\vphantom{\sup}}
\DeclareMathOperator*{\argmin}{arg\,min}

\DeclareMathOperator*{\esssup}{ess\,sup}

\DeclarePairedDelimiter{\pars}{\ensuremath{(}}{\ensuremath{)}}
\DeclarePairedDelimiter{\bracs}{\ensuremath{[}}{\ensuremath{]}}
\DeclarePairedDelimiter{\braces}{\ensuremath{\{}}{\ensuremath{\}}}

\DeclarePairedDelimiter{\inner}{\langle}{\rangle}
\DeclarePairedDelimiter{\norm}{\|}{\|}
\DeclarePairedDelimiter{\abs}{\lvert}{\rvert}

\let\oldbullet\bullet
\newlength{\raisebulletlen}
\setbox1=\hbox{$\bullet$}\setbox2=\hbox{\tiny$\bullet$}
\renewcommand\bullet{\raisebox{\raisebulletlen}{\,\tiny$\oldbullet$}\,}

\newcommand*{\xhat}[2][0.3em]{#2\kern-#1\hat{\vphantom{#2}}\kern#1}

\def\keywordname{{\bfseries \emph{Key words.}}}%

\def\mscclassname{{\bfseries \emph{AMS subject classifications.}}}%
\def\mscclasses#1{\par\addvspace\medskipamount{\rightskip=0pt plus1cm
\def\and{\ifhmode\unskip\nobreak\fi\ $\cdot$
}\noindent\mscclassname\enspace\ignorespaces#1\par}}

\title{Convergence bounds for nonlinear least squares and applications to tensor recovery}
\renewcommand{\undertitle}{}
\renewcommand{\shorttitle}{Convergence bounds for tensor recovery}
\hypersetup{
	breaklinks,
	colorlinks,
	linkcolor=gray,
	citecolor=gray,
	urlcolor=gray,
    pdftitle={Convergence bounds for nonlinear least squares and applications to tensor recovery},
    pdfsubject={math.NA, cs.LG},
    pdfauthor={Philipp Trunschke},
    pdfkeywords={empirical L2 approximation, sample efficiency, sparse tensor networks, alternating least squares},
}

\author{\href{https://orcid.org/0000-0002-2995-126X}{\includegraphics[scale=0.06]{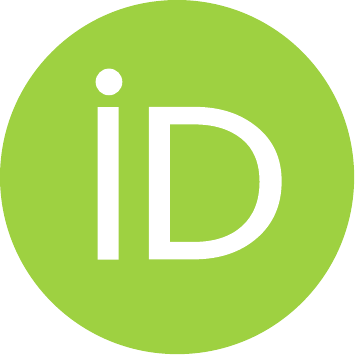}\hspace{1mm}\textcolor{black}{Philipp Trunschke}} \\
	Department of Mathematics\\
	Technische Universität Berlin\\
	Berlin, Germany \\
	\texttt{ptrunschke@mail.tu-berlin.de} \\
}

\begin{document}
\maketitle

\begin{abstract}
    We consider the problem of approximating a function in general nonlinear subsets of $L^2$ when only a weighted Monte Carlo estimate of the $L^2$-norm can be computed.
    Of particular interest in this setting is the concept of sample complexity, the number of samples that are necessary to recover the best approximation.
    Bounds for this quantity have been derived in a previous work and depend primarily on the model class and are not influenced positively by the regularity of the sought function.
    This result however is only a worst-case bound and is not able to explain the remarkable performance of iterative hard thresholding algorithms that is observed in practice.
    We reexamine the results of the previous paper and derive a new bound that is able to utilize the regularity of the sought function.
    A critical analysis of our results allows us to derive a sample efficient algorithm for the model set of low-rank tensors.
    The viability of this algorithm is demonstrated by recovering quantities of interest for a classical high-dimensional random partial differential equation.
\end{abstract}

\vspace{-1em}
\mscclasses{15A69 \and 41A30 \and 62J02 \and 65Y20 \and 68Q25}

\section{Introduction}
\note{Often only the introduction is read. Therefore it should state the main results as clearly as possible.}

In this paper we consider the task of estimating an unknown function from noiseless observations.
For this problem to be well-posed, some prior information about the function has to be imposed.
This often takes the form of regularity assumptions, like the ability to be well approximated in some model class.
Regularization is another popular method to encode regularity assumptions but, assuming Lagrange duality, can also be interpreted as a restriction to a model class.
Given such a model class, it is of particular interest how well a sample-based estimator can approximate the sought function.
In investigating this question, many papers rely on a \emph{restricted isometry property} (RIP) or a RIP-like condition.
The RIP asserts, that the sample-based estimate of the approximation error is equivalent to the approximation error for all elements of the model class.
This is an important property with respect to generalization.
Without the equivalence, it is easy to conceive circumstances under which a minimizer of the empirical approximation error is arbitrarily far away from the real best approximation.

In this setting, the quality of the estimator depends on the number of samples that are required for the RIP to hold with a prescribed probability.
This has been studied extensively for linear spaces~\cite{cohen_2017_least-squares} and sparse-grid spaces~\cite{bohn_2018_sparse_grid}, for sparse vectors~\cite{candes06recovery,Rauhut2016weighted_l1}, low-rank matrices and tensors~\cite{recht_2010_nuclear_norm_minimization,rauhut2017iht}, as well as for neural networks~\cite{goessmann2020restricted} and, only recently, for generic, non-linear model classes~\cite{EST20}.\footnote{As machine learning and statistics are huge and highly-active research areas, this list raises no claim to completeness.}

This work continues the line of thought, started in~\cite{EST20}, by
studying the dependence of the RIP on the model class and by utilizing the gained insights to develop a new algorithm that drastically outperforms existing state-of-the-art algorithms~\cite{eigel2018vmc0,Grasedyck2019SALSA,Kraemer2020thesis} in the sample-scarce regime.

Although applicable to a wide range of model classes, our deliberations focus on model classes of tensor networks.
These have been applied successfully in uncertainty quantification~\cite{haberstich_2020_thesis} and dynamical systems recovery~\cite{goette_2020_dynamical_laws}.
In~\cite{ESTW19} tensor networks were used for the sample-based, non-intrusive computation of parametric solutions and quantities of interest of high-dimensional, random partial differential equations.
It was shown that, compared to (quasi) Monte-Carlo methods, tensor networks allow for a drastic reduction in the number of samples.

\paragraph{Setting.}
Consider the space $\mcal{V} = L^2\pars{Y,\rho}$ for some probably measure $\rho$ and define the norms $\norm{\bullet} := \norm{\bullet}_{L^2\pars{Y,\rho}}$ and $\norm{\bullet}_\infty = \norm{\bullet}_{L^\infty\pars{Y,\rho}}$.
Given point-evaluations $\braces{\pars{y^i, u\pars{y^i}}}_{i=1}^n$ of an unknown function $u\in\mcal{V}$ we want to find a (not necessarily unique) best approximation
\begin{equation}
    u_{\mcal{M}} \in \argmin_{v\in \mcal{M}}\ \norm{u - v} \label{eq:min}
\end{equation}
of $u$ in the \emph{model class} $\mcal{M}\subseteq\mcal{V}$.
In general however, $u_{\mcal{M}}$ is not computable and a popular remedy is to estimate
\begin{equation} \label{eq:min_emp}
    \norm{v} \approx \norm{v}_{y} := \pars*{\frac{1}{n}\sum_{i=1}^n w\pars{y^i} \abs{v\pars{y^i}}^2}^{1/2}
    \quad\text{and}\quad
    u_{\mcal{M}} \approx u_{\mcal{M},y} \in \argmin_{v\in\mcal{M}}\ \norm{u - v}_{y} ,
\end{equation}
where $w$ is a fixed \emph{weight function}, satisfying $w \ge 0$ and $\int_Y w^{-1} \dx[\rho] = 1$, and where $y_i\sim w^{-1}\rho$ for all $y=1,\ldots, n$.

This problem occurs in many applications like system identification~\cite{goette_2020_dynamical_laws}, the computation of surrogate models for high-dimensional partial differential equations~\cite{eigel2018vmc0,ESTW19}
and the computation of conditional expectations in computational finance~\cite{BEST21}.
Three particularly illustrative application to which we will repeatedly refer to in this work are polynomial regression, sparse regression and matrix completion.
\begin{example} \label{ex:polynomial_regression}
    Let $Y = \bracs{-1,1}$, $\rho = \frac{1}{2}\dx$ be the uniform distribution on $Y$ and define $\mcal{V} = L^2\pars{Y,\rho}$.
    In polynomial regression the model class $\mcal{M}$ is given by a finite dimensional subspace of polynomials.
    The goal is to obtain a fit of $u$ in $\mcal{M}$ using just point-wise evaluations of $u$.
\end{example}
\begin{example} \label{ex:sparse_regression}
    Let $Y$, $\rho$ and $\mcal{V}$ be defined as in \cref{ex:polynomial_regression}.
    The objective of sparse regression is the same as for least-squares regression, the model class, however, is more restrictive.
    For any fixed orthonormal basis $\braces{B_j}_{j\in\mbb{N}}\subseteq\mcal{V}$ and any sequence $\braces{\omega_j}_{j\in\mbb{N}}$, that satisfies $\omega_j \ge \norm{B_j}_\infty$ for all $j\in\mbb{N}$, the $\omega$-weighted $\ell^p$ is defined as
    \begin{equation}
        \norm{v}_{\omega,p} := \pars*{\sum_{j\in\mbb{N}} \abs{v_j^{\vphantom{p}}}_{\vphantom{j}}^p\omega_j^{2-p}}^{1/p} 
        \quad\text{for $p\in(0,2]$ and}\quad
        \norm{v}_{\omega,0} := \sum_{j\in\mbb{N}} \abs{v_j^{\vphantom{2}}}_{\vphantom{j}}^0\omega_j^{2},
    \end{equation}
    with the convention that $0^0 = 0$.
    Here $v_j$ denotes the coefficient of $v$ with respect to the basis function $B_j$.
    The model class of $\omega$-weighted $s$-sparse functions can then be defined as
    \begin{equation}
        \mcal{M} := \braces*{v\in\mcal{V} \,:\, \norm{v}_{\omega,0} \le s}.
    \end{equation}
\end{example}
\begin{example} \label{ex:matrix_completion}
    Let $Y = \bracs{m}\times\bracs{n}$ and $\rho$ be a uniform distribution on $Y$.
    Observe that this means $\mcal{V} = \mbb{R}^{m\times n}$ and $\norm{\bullet} = \frac{1}{\sqrt{d_1d_2}}\norm{\bullet}_{\mathrm{Fro}}$.
    In matrix completion we assume that $\mcal{M}$ is the set of matrices with rank bounded by $r$ and want to find a matrix $u:Y\to \mbb{R}$ from just a few of its entries.
\end{example}

In a previous work~\cite{EST20} the authors show that the empirical best approximation error $\norm{u - u_{\mcal{M},n}}$ is equivalent to the best approximation error $\norm{u - u_{\mcal{M}}}$ if the \emph{restricted isometry property} (RIP)
\begin{equation}\label{eq:rip}
    \operatorname{RIP}_A\pars{\delta} :\Leftrightarrow \pars{1-\delta}\norm{u}^2 \le \norm{u}_y^2 \le \pars{1+\delta}\norm{u}^2 \qquad \forall u\in A
\end{equation}
is satisfied for the set $A=\braces{u_{\mcal{M}}} - \pars{\mcal{M}\cup\braces{u}}$ and any $\delta \in \pars{0,1}$.
A worst-case estimate for the probability of this RIP is derived and it is shown that the RIP holds with high probability for many model classes.
To the knowledge of the authors these are the first bounds in this general nonlinear setting.
Although these bounds are far from optimal they allow  us to consider the empirical best approximation problem for arbitrary nonlinear model classes.

In this work we consider model classes of tensor networks and show that the worst-case estimate for their sample complexity, i.e.\ the number of samples that are necessary to achieve the RIP with a prescribed probability, behaves asymptotically the same way as the sample complexity estimate for the full tensor space in which they are contained.
Although the covering number of a tensor network is typically exponentially smaller compared to that of its ambient tensor spaces, this agrees with observations from matrix and tensor completion where a low-rank matrix or tensor has to satisfy an additional \emph{incoherence condition} to guarantee a reduced sample complexity~\cite{Candes2010convexRelaxationMatrixCompletion,yuan2014tensor}.
This means that not every tensor can be recovered with a reasonable number of samples.

From the numerical experiments in~\cite{ESTW19} we know that a rank-adaptive iterative hard thresholding algorithm~\cite{eigel2018vmc0} is capable of recovering solutions of high-dimensional parametric partial differential equations with surprisingly few samples.
This indicates that the regularity of the sought function $u$ has a beneficial effect on the sample complexity which is not captured by the theory in \cite{EST20}.

In this article we expand upon the basic results from~\cite{EST20} and show that a generalized incoherence condition can be derived for general nonlinear model classes.
From this result we derive design principles for algorithms for empirical norm minimization
and present an adapted version of the \emph{alternating least squares} (ALS) algorithm~\cite{oseledets_2011_tensor_trains,holtz_alternating_2012} for low-rank tensor recovery.
Finally, we perform numerical experiments to illustrate the remarkable performance of the derived algorithm compared to other state of the art methods.

Our deliberation is motivated by the application to the model class of tensor networks.
A short introduction to this topic and its applications is provided in Appendix~\ref{sec:tensor_networks} and understanding the graphical notation introduced therein is presumably necessary to understand \cref{sec:restricted_als}.
For a comprehensive discussion we refer the reader to~\cite{Hackbusch_2012_book,Grasedyck_2013_survey}.

\begin{remark}
    For the sake of clarity, we consider the special case $\norm{\bullet} = \norm{\bullet}_{L^2\pars{Y,\rho}}$.
    However, we conjecture that our results still hold, in a similar form, in the more general setting of~\cite{EST20}. %
    Moreover, although the present discussion is motivated by empirical norm minimization, the RIP also guarantees the convergence of $\ell^1$-minimization~\cite{candes06recovery}, nuclear norm minimization~\cite{mohan2010new} and iterative hard thresholding~\cite{rauhut2017iht}.
    Finally, note that the theory is not restricted to the minimization of errors $\norm{u-v}$ but is also applicable to the minimization of residuals $\norm{u-Lv}$ as done, for example for the approximation of the stationary Bellman equation in~\cite{oster2021approximating}.
\end{remark}

\paragraph{Structure.}
The remainder of the paper is organized as follows.
In \cref{sec:convergence_bounds} we recall and expand upon the basic results from~\cite{EST20} and derive calculus rules for the computation of the variation constant.
This section culminates in a proof that shows the equality of the worst-case probability estimates for tensor networks and their ambient linear spaces.
\cref{sec:local_variation_constant} starts with an example that illustrates this result.
We observe that the worst-case bounds are reachable only when the elements in the model class can be arbitrarily far away from the best approximation $u_{\mcal{M}}$.
Building on this insight we define a local version of the model class $\mcal{M}^{\mathrm{loc}}_{u,r} = \braces{v\in\mcal{M} : \norm{u-v}\le r}$ and prove the main theorem of this work:
When the model class $\mcal{M}$ is locally linearizable in the neighborhood of $u_{\mcal{M}}$, then the sample complexity of $\mcal{M}^{\mathrm{loc}}_{u,r}$ can be estimated by the sample complexity of the tangent space for sufficiently small $r$.
We conclude with an illustration of this theorem in the setting of low-rank matrices.
These results are discussed further in \cref{sec:restricted_als}.
We argue that the requirement of local linearizability is too stringent for most practical applications and formulate design principles that ought to be fulfilled by recovery algorithms.
From these principles we derive a new optimization algorithm for tensor networks.
We conclude this paper with promising experimental results in \cref{sec:experiments}.

\paragraph{Notation.}
Denote the set of integers from $1$ through $d$ by $\bracs{d}$.
For $d\in\mbb{N}$ and any $i\in\bracs{d}$ we denote by $e_i\in\mbb{R}^d$ the $i$\textsuperscript{th} standard basis vector and define $\boldsymbol{1} := \pars{1 \ \cdots \ 1}^\intercal\in\mbb{R}^d$.
For any $v\in\mbb{R}^d$, define $\operatorname{supp}\pars{v} := \braces{j\in\bracs{d} : v_j\ne0}$.

For any set $X$, the notation $\inner{X}$ denotes the linear span of $X$, $\cl\pars{X}$ denotes its closure and $\mfrak{P}\pars{X}$ denotes the set of all subsets of $X$.
If $\pars{X,p}$ is a metric space, then $\mfrak{C}\pars{X}$ denotes the set of non-empty, compact subsets of $X$ and $S\pars{x,r}\subseteq X$ and $B\pars{x,r}\subseteq X$ denote the sphere and ball of radius $r>0$ and with center $x\in X$, respectively.
Since the metric space $X$ should always be clear from context, we do not include it in the notation for $S\pars{x,r}$ and $B\pars{x,r}$.

In \cref{thm:K_properties} we require the concept of a continuous function that operates on sets.
The relevant topologies are induced by the following two metrics.
\begin{definition}[Hausdorff distance]
    Let $\pars{M, d}$ be a metric space.
    The function $d_{\mathrm{H}} : \mfrak{C}\pars{M}^2 \to [0,\infty)$
    \begin{equation}
        d_{\mathrm{H}}\pars{X,Y} := \max\braces{\sup_{x\in X}\inf_{y\in Y} d\pars{x,y}, \sup_{y\in Y}\inf_{x\in X} d\pars{x,y}}
    \end{equation}
    defines a metric on $\mfrak{C}\pars{M}$ and a pseudometric on $\mfrak{P}\pars{M}$.
\end{definition}

Note that this is not an appropriate metric when the sets $X$ and $Y$ are cones, since in this case
\begin{equation}
    d_{\mathrm{H}}\pars{X,Y} = \begin{cases}
                                   0 & \cl\pars{X} = \cl\pars{Y} \\
                                   \infty & \text{otherwise}
                               \end{cases} .
\end{equation}
In the following we define $\operatorname{Cone}\pars{X} := \braces{\lambda x : \lambda >0, x\in X}$ and denote by $\operatorname{Cone}\pars{\mfrak{P}\pars{M}}$ the set of all cones in $M$.
Since conic set are uniquely defined by their intersection with the unit sphere we can define a more suitable (pseudo-)metric by considering the Hausdorff distance between these intersections.
This metric is called the truncated Hausdorff distance.
\cite{seeger_2010_Hausdorff_distance}
\begin{definition}[truncated Hausdorff distance]
    The function $d_{\mathrm{tH}} : \operatorname{Cone}\pars{\mfrak{P}\pars{M}}^2 \to [0,\infty)$, defined by
    \begin{equation}
        d_{\mathrm{tH}}\pars{X,Y} := d_{\mathrm{H}}\pars{S\pars{0,1}\cap X, S\pars{0,1}\cap Y},
    \end{equation}
    is a pseudometric on $\operatorname{Cone}\pars{\mfrak{P}\pars{M}}$.
\end{definition}

\section{Convergence bounds}\label{sec:convergence_bounds}

The restricted isometry property can be used to show the following equivalence. %

\begin{theorem} \label{thm:empirical_projection_error}
    If $\operatorname{RIP}_{\braces{u_{\mcal{M}}}-\pars{\mcal{M}\cup\braces{u}}}\pars{\delta}$ holds then
    \begin{equation}
        \norm{u-u_{\mcal{M}}} \le \norm{u - u_{\mcal{M},\boldsymbol{y}}} \le \pars*{1+2\sqrt{\frac{1+\delta}{1-\delta}}}\norm{u - u_{\mcal{M}}} .
    \end{equation}
\end{theorem}
\begin{proof}
    Observe that $\operatorname{RIP}_{\braces{u_{\mcal{M}}}-\pars{\mcal{M}\cup\braces{u}}}\pars{\delta}$ holds if and only if $\operatorname{RIP}_{\braces{u_{\mcal{M}}}-\mcal{M}}$ and $\operatorname{RIP}_{\braces{u_{\mcal{M}}-u}}$ hold.
    The theorem then follows from Theorem~2.12 in~\cite{EST20}.
\end{proof}

\Cref{thm:empirical_projection_error} holds for any choice of $\boldsymbol{y}\in Y^n$, but we assume that the $\boldsymbol{y}_i$ are i.i.d.\ random variables.
This means that $\operatorname{RIP}_A\pars{d}$ is a random variable as well and its probability can be bounded by a standard concentration of measure argument.
To do this, we define the normed space
\begin{equation}
    \mcal{V}_{w,\infty} := \braces{v\in\mcal{V} : \norm{v}_{w,\infty} < \infty}
    \qquad\text{where}\qquad
    \norm{v}_{w,\infty} := \esssup_{y\in Y} \sqrt{w\pars{y}} \abs{v\pars{y}} .
\end{equation}
The \emph{variation function} of  a model class $A$ is then given by 
\begin{equation}
    \mfrak{K}_A\pars{y} := \sup_{a\in U\pars{A}} \abs{a\pars{y}}^2
    \quad\text{where}\quad
    U\pars{A} := \braces*{\tfrac{u}{\norm{u}} : u\in A\!\setminus\!\braces{0}}
\end{equation}
and provides a point-wise bound for the relative oscillation\footnote{The oscillation of $a\in A$ is defined by $\operatorname{osc}\pars{a} := \sup_{y\in Y}a\pars{y} - \inf_{y\in Y} a\pars{y}$. The relative oscillation is bounded by $\operatorname{osc}\pars{a}/\norm{a} \le 2 \norm{a}_\infty/\norm{a}$.} of the functions in $A$.
\begin{remark}
    The variation function can be seen as the inverse of a generalized Christoffel function~\cite{nevai_1986_Christoffel_functions}.
\end{remark}

With this definition we can state the following bound on the probability of $\operatorname{RIP}_A\pars{\delta}$.
\begin{theorem}[Theorem~2.7 and Corollary~2.10 in~\cite{EST20}] \label{thm:P_RIP}
    For any $A\subseteq \mcal{V}$ and $\delta>0$ there exists $C$ such that
    \begin{equation}
        \mbb{P}\bracs{\operatorname{RIP}_A\pars{\delta}} \ge 1 - C \exp\pars*{-\tfrac{n}{2}\pars{\tfrac{\delta}{\norm{\mfrak{K}_A}_{w,\infty}}}^2} .
    \end{equation}
    The constant $C$ is independent of $n$ and depends only polynomially on $\delta$ and $\norm{\mfrak{K}_A}_{w,\infty}^{-1}$ if $\dim\pars{\inner{A}}<\infty$.
    \note{This must be the case since we know that $\nu\pars{U\pars{A}, r} < \nu\pars{U\pars{\inner{A}}, r} \lesssim r^{-M}$ where $M=\dim\pars{\inner{A}}$.}
\end{theorem}

\begin{remark}
    Note that \cref{thm:P_RIP} also provides worst-case bounds for deterministic algorithms.
    If $\mbb{P}\bracs{\operatorname{RIP}_{\mcal{M}-\mcal{M}}\pars{\delta}} > 0$, we can find $\boldsymbol{y}\in Y^n$ such that $\operatorname{RIP}_{\mcal{M}-\mcal{M}}\pars{\delta}$ is satisfied.
    Thus, the conditions for \cref{thm:empirical_projection_error} are satisfied for any $u\in\mcal{M}$ and hence there exists a deterministic algorithm to exactly recover any $u\in\mcal{M}$ using $n$ function evaluations.
\end{remark}

From the bound in \cref{thm:P_RIP} we can see that a low value of $\norm{\mfrak{K}_A}_{w,\infty}$ is necessary to obtain a large probability for $\operatorname{RIP}_A\pars{\delta}$.
Together, \cref{thm:empirical_projection_error} and \cref{thm:P_RIP} allow us to compute the probability with which the best approximation $u_{\mcal{M}}$ of a function $u$ may be recovered exactly in a given model class $\mcal{M}$.
The conditions of this theorem are satisfied by many model classes, such as finite dimensional vector spaces, sets of sparse vectors or sets of low-rank tensors.

The variation function allow us to compute the optimal sampling density of a set $A$ as stated in the subsequent theorem.
\begin{theorem}[Theorem~3.1 in~\cite{EST20}] \label{thm:optimal_weight_function}
    $\mfrak{K}_A$ is $\rho$-measurable and
    \note{$\mcal{V}=L^2$ is separable}
    \begin{equation}
        \norm{w\mfrak{K}_A}_{L^\infty\pars{Y,\rho}} \ge \norm{\mfrak{K}_A}_{L^1\pars{Y,\rho}}
    \end{equation}
    for any weight function $w$.
    The lower bound is attained by the weight function $w=\frac{\norm{\mfrak{K}_A}_{L^1\pars{Y,\rho}}}{\mfrak{K}_A}$.
\end{theorem}

The subsequent theorem provide us with calculus rules for the computation of $\mfrak{K}$ which we will frequently use in the remainder of this work.
\begin{theorem}[Basic properties of $\mfrak{K}$] \label{thm:K_properties}
    Let $A,B\subseteq\mcal{V}_{w,\infty}$ and $\mcal{A}\subseteq\mfrak{P}\pars{\mcal{V}_{w,\infty}}$.
    Then the following statements hold.
    \begin{thmenum}
        \item %
        $\mfrak{K}_{\bigcup\!\mcal{A}} = \sup_{A\in\mcal{A}} \mfrak{K}_{A}$, where $\bigcup\!\mcal{A} := \bigcup_{A\in\mcal{A}}A$.
        \label{thm:K_properties:union}
        \item $\mfrak{K}_{A} = \mfrak{K}_{\operatorname{cl}\pars{A}}$. \label{thm:K_properties:closure}
        \item $\mfrak{K}_{\bullet} : \mfrak{C}\pars{\mcal{V}_{w,\infty}\setminus \braces{0}} \to \mcal{V}_{w^2,\infty}$ is continuous with respect to the Hausdorff metric. \label{thm:K_properties:continuity:compact_sets}
        \note{5.\ $\not\Rightarrow$ 4., since 5.\ considers sequences of compact sets which vonverge to $\braces{0}$. This, however, is not the case in 4.}\item $\mfrak{K}_{\bullet} : \mfrak{P}\pars{\mcal{V}_{w,\infty}\setminus B\pars{0,r}} \to \mcal{V}_{w^2,\infty}$ is continuous with respect to the Hausdorff pseudometric for all $r>0$. \label{thm:K_properties:continuity:all_sets}
        \item $\mfrak{K}_{\bullet} : \operatorname{Cone}\pars{\mfrak{P}\pars{\mcal{V}_{w,\infty}}} \to \mcal{V}_{w^2,\infty}$ is continuous with respect to the truncated Hausdorff pseudometric. \label{thm:K_properties:continuity:cones}
        \item If $A\perp B$ then $\mfrak{K}_{A + B} \le \mfrak{K}_{A} + \mfrak{K}_{B}$. \label{thm:K_properties:sum}
        \note{In genereal, equality can not hold. To see this, consider $\mfrak{K}_{\braces{b_1+\alpha b_2}}$ for $b_1 \equiv 1$ and $b_2 = 2\chi_{\bracs{-\frac{1}{2},\frac{1}{2}}}-1$ in $L^2\pars{\bracs{-1,1}, \frac{1}{2}\dx}$. Then $\mfrak{K}_{\braces{b_1+\alpha b_2}} = \frac{\pars{1+\alpha}^2}{1+\alpha^2}$ but $\mfrak{K}_{\braces{b_1}} + \mfrak{K}_{\braces{\alpha b_2}} = 2$.}
        \item $\mfrak{K}_{A\oplus B} = \mfrak{K}_{A} + \mfrak{K}_{B}$. \label{thm:K_properties:direct_sum}
        \item If $A\indep B$ then $\mfrak{K}_{A\cdot B} = \mfrak{K}_{A} \cdot \mfrak{K}_{B}$. \label{thm:K_properties:product}
        \item $\mfrak{K}_{A\otimes B} = \mfrak{K}_{A} \cdot \mfrak{K}_{B}$. \label{thm:K_properties:tensor_product}
    \end{thmenum}
    Where the sum ($+$), product ($\cdot$), the stochastic independence ($\indep$) and the orthogonality ($\perp$) of sets have to be understood element-wise. \proofappendix{proof:thm:K_properties}
\end{theorem}
As a consequence of \cref{thm:K_properties:union} it follows that $A\subseteq B$ implies $\mfrak{K}_A \le \mfrak{K}_B$.
In combination with \cref{thm:K_properties:product} and \ref{thm:K_properties:sum} this allows for the interpretation of the function $\mfrak{K}_{\bullet}$ as a monotonic and (uniformly) continuous (partially defined) morphism of algebras.
The continuity of $\mfrak{K}$ can, for example, be used to estimate the variation constant numerically, as is done in \cref{sec:algorithm_KUbbone}.
Moreover, by virtue of \cref{thm:optimal_weight_function}, the properties in \cref{thm:K_properties} induce analogous properties of the norm $\norm{\mfrak{K}_A}_{w,\infty}$.
\cref{thm:K_properties:sum} for example, implies that for any linear space $A$ that is spanned by an orthonormal basis $\braces{B_k}_{k=1}^{\dim\pars{A}}$
\begin{equation} \label{KU_properties:dim}
    \norm{\mfrak{K}_{A}}_{w,\infty}
    = \big\Vert\sum_{k=1}^{\dim\pars{A}}\mfrak{K}_{\inner{B_k}}\big\Vert_{w,\infty}
    \le \sum_{k=1}^{\operatorname{dim}\pars{A}} \norm{B_k}_{w,\infty}^2 .
\end{equation}
Finally note, that \cref{thm:optimal_weight_function} and \cref{thm:K_properties} provides calculus rules for the computation of optimal weight functions.
\note{The fact that $\mfrak{K}_{A\otimes B} = \mfrak{K}_{A\cdot B}$ is not a coincidence.
The maximizer in $\mfrak{K}_{A\otimes B}$ is indeed an element of $A\cdot B$.}

\begin{remark}
    A common misconception is that the probability bound in \cref{thm:P_RIP} relies primarily on the metric entropy~\cite{cockreham_2017_reach} of the model class.
    This however is not true, since $\norm{\mfrak{K}_A}_{\infty}$ is independent of the metric entropy of $A$.
    To see this, consider any set $A$ for which $U\pars{A}$ is compact.
    By continuity, there exists $a^*\in A$ such that $\norm{\mfrak{K}_{\braces{a^*}}}_\infty \ge \norm{\mfrak{K}_{\braces{a}}}_\infty$ for all $a\in A$.
    Thus, $\norm{\mfrak{K}_{\braces{a^*}}}_\infty \ge \norm{\mfrak{K}_{\mcal{N}}}_\infty$ for any subclass $\mcal{N}\subseteq\mcal{M}$, independent of its metric entropy.
\end{remark}

We use the remainder of this section to compute the variation function for a generic model class of tensor networks $\mcal{M}$ (cf.~\cref{sec:tensor_networks}).
We do this by proving the sequence of inequalities
\begin{equation} \label{eq:chain}
    \mfrak{K}_{\braces{u_{\inner{\mcal{M}}}} - \inner{\mcal{M}}}
    = \mfrak{K}_{\braces{u_{\mcal{M}}} - \inner{\mcal{M}}}
    \ge \mfrak{K}_{\braces{u_{\mcal{M}}} - \mcal{M}}
    \ge \mfrak{K}_{\mcal{M}}
    = \mfrak{K}_{\inner{\mcal{M}}}
    = \mfrak{K}_{\braces{u_{\inner{\mcal{M}}}} - \inner{\mcal{M}}} .
\end{equation}
The first and the last equality hold, since $\braces{u_{\inner{\mcal{M}}}} - \inner{\mcal{M}} = \inner{\mcal{M}} = \braces{u_{\mcal{M}}} - \inner{\mcal{M}}$.
The remaining relations follow from \cref{thm:K_properties:union}, \cref{prop:KuM=KspanuM} and \cref{prop:KM=KspanM}.

Since the probability of $\operatorname{RIP}_{\braces{u_{\mcal{M}}}-\mcal{M}\cup\braces{u}}$ can not exceed that of $\operatorname{RIP}_{\braces{u_{\mcal{M}}}-\mcal{M}}$, this shows that recovery in any model class of tensor networks $\mcal{M}$ requires roughly the same number of samples as recovery in the ambient space $\inner{\mcal{M}}$.
Since $\norm{\mfrak{K}_{\inner{\mcal{M}}}}_{w,\infty} \ge \operatorname{dim}\pars{\inner{\mcal{M}}}$ grows exponentially with the order of the tensors, this model class may be infeasible for the recovery of certain tensors.
This is not surprising.
In the setting of low-rank matrix and tensor recovery it is well known, that the sought tensor has to satisfy an additional incoherence condition to be recoverable with few samples (cf.~\cite{Candes2010convexRelaxationMatrixCompletion,yuan2014tensor}).
To illustrate this, we provide phase diagrams for the recovery of two different functions in \cref{fig:recovery_phase_diagram}.

\begin{proposition} \label{prop:KuM=KspanuM}
    Let $\mcal{M}$ be conic and symmetric and let $v\in\mcal{V}$.
    Then $\mfrak{K}_{\braces{v}-\mcal{M}} = \mfrak{K}_{\inner{v}+\cl\pars{\mcal{M}}}$.
\end{proposition}
\note{Let $u=\Delta_1 + \Delta_2$ and $\mcal{M}$ be the set of matrices of rank at most $2$.
Then $u_{\mcal{M}} = u$ and the maximum $\sup_{v\in\mcal{M}} \frac{\norm{u-v}_\infty^2}{\norm{u-v}^2} = \frac{\norm{\Delta_1 + \Delta_2 - v}_\infty^2}{\norm{\Delta_1 + \Delta_2 - v}^2} = \frac{\norm{\Delta_1}_\infty^2}{\norm{\Delta_1}^2} = d_1d_2$ is attained for $v=\Delta_2$.
This shows that the inequality can not be an equality!}

To prove \cref{prop:KuM=KspanuM} we need the following lemma.
\begin{lemma} \label{lem:closure_of_sum}
    $\cl\pars{A+B} \supseteq \cl\pars{A}+\cl\pars{B}$ for all sets $A$ and $B$.
\end{lemma}
\begin{proof}
    Let $a\in \cl\pars{A}$ and $b\in \cl\pars{B}$.
    Then there exist sequences $\braces{a_k}\in A$ and $\braces{b_k}\in B$ such that $a_k\to a$ and $b_k \to b$.
    Since $a_k + b_k\in A+B$ we have $a+b = \lim_k a_k+b_k \in \cl\pars{A+B}$.
\end{proof}

\begin{proof}[Proof of \cref{prop:KuM=KspanuM}]
    Since $\mcal{M}$ is symmetric it holds that $\braces{\pm v} - \mcal{M} = \pm\pars{\braces{v} - \mcal{M}}$.
    From \cref{thm:K_properties:union} and \ref{thm:K_properties:product} we can thus conclude
    \begin{equation}
        \mfrak{K}_{\braces{-v,v} - \mcal{M}}
        = \max\braces{\mfrak{K}_{-\pars{\braces{v} - \mcal{M}}}, \mfrak{K}_{\braces{v} - \mcal{M}}}
        = \max\braces{\mfrak{K}_{\braces{v} - \mcal{M}}, \mfrak{K}_{\braces{v} - \mcal{M}}}
        = \mfrak{K}_{\braces{v} - \mcal{M}} .
    \end{equation}
    Moreover, since $\mcal{M}$ is conic, it holds for any set $A$ that $\operatorname{Cone}\pars{A-\mcal{M}} = \operatorname{Cone}\pars{A}-\mcal{M}$.
    \Cref{thm:K_properties:product} implies $\mfrak{K}_{\braces{-v, v} - \mcal{M}} = \mfrak{K}_{\inner{v}\!\setminus\!\braces{0}-\mcal{M}}$ and consequently $\mfrak{K}_{\braces{v} - \mcal{M}} = \mfrak{K}_{\inner{v}\!\setminus\!\braces{0}-\mcal{M}}$.
    Finally, using \cref{lem:closure_of_sum} and \cref{thm:K_properties:closure} yields $\mfrak{K}_{\braces{v}-\mcal{M}} = \mfrak{K}_{\inner{v}-\cl\pars{\mcal{M}}}$.
\end{proof}

\begin{remark}
    Since $\inner{u_{\mcal{M}}}-\cl\pars{\mcal{M}}\supseteq\mcal{M}$, \cref{thm:K_properties:union} and 
    \cref{prop:KuM=KspanuM} show that $\mfrak{K}_{\braces{u_{\mcal{M}}}-\mcal{M}} \ge \mfrak{K}_{\mcal{M}}$.
    This means that the variation function $\mfrak{K}_{\braces{u_{\mcal{M}}}-\mcal{M}}$ is not favorably influenced by the regularity of $u_{\mcal{M}}$.
\end{remark}

\begin{proposition} \label{prop:KM=KspanM}
    For any model class of tensor networks $\mcal{M}$ of fixed order it holds that  $\mfrak{K}_{\mcal{M}} = \mfrak{K}_{\inner{\mcal{M}}}$.
\end{proposition}
\begin{proof}
    Let $\mcal{M} \subseteq L^2\pars{Y_1,\rho_1}\otimes\cdots\otimes L^2\pars{Y_M,\rho_M}$ be a set of tensor networks of order $M$ with arbitrary but fixed architecture and rank constraints.
    
    Define the marginal vector spaces $\mcal{V}_m\subseteq L^2\pars{Y_m, \rho_m}$ such that
    \begin{equation}
        \inner{\mcal{M}} = \bigotimes_{m=1}^M \mcal{V}_m .
    \end{equation}
    and the set of \emph{rank--$1$ tensors} (cf.~\cite{hitchcock1927cp_format}) as $\mcal{T}_{1} := \braces{v_1\otimes\cdots\otimes v_M : v_m\in \mcal{V}_m\text{ for all }m} = \mcal{V}_1\cdots\mcal{V}_M$.
    Since every every element in $\mcal{T}_1$ can be approximated arbitrarily well in $\mcal{M}$, \cref{thm:K_properties:closure} and \cref{thm:K_properties:union} imply
    \begin{equation}
        \mfrak{K}_{\mcal{M}} = \mfrak{K}_{\operatorname{cl}\pars{\mcal{M}}}\ge \mfrak{K}_{\mcal{T}_1} \label{eq:KM=KT1}
    \end{equation}
    and by \cref{thm:K_properties:product} and \ref{thm:K_properties:tensor_product}, it holds that 
    \begin{equation}
        \mfrak{K}_{\mcal{T}_1} = \mfrak{K}_{\mcal{V}_1\cdots\mcal{V}_M} = \mfrak{K}_{\mcal{V}_1} \cdots \mfrak{K}_{\mcal{V}_M} = \mfrak{K}_{\mcal{V}_1\otimes\cdots\otimes\mcal{V}_M} = \mfrak{K}_{\inner{\mcal{M}}} . \label{eq:KT1=KV}
    \end{equation}
    Employing \cref{thm:K_properties:union} a final time and combining \eqref{eq:KM=KT1} and \eqref{eq:KT1=KV} yields the chain of inequalities
    $\mfrak{K}_{\inner{\mcal{M}}} \ge \mfrak{K}_{\mcal{M}} \ge \mfrak{K}_{\mcal{T}_1} = \mfrak{K}_{\inner{\mcal{M}}}$,
    which concludes the proof.
\end{proof}

\begin{figure}
    \centering
    \begin{subfigure}[t]{0.5\textwidth}
        \centering
        \includegraphics[width=\textwidth]{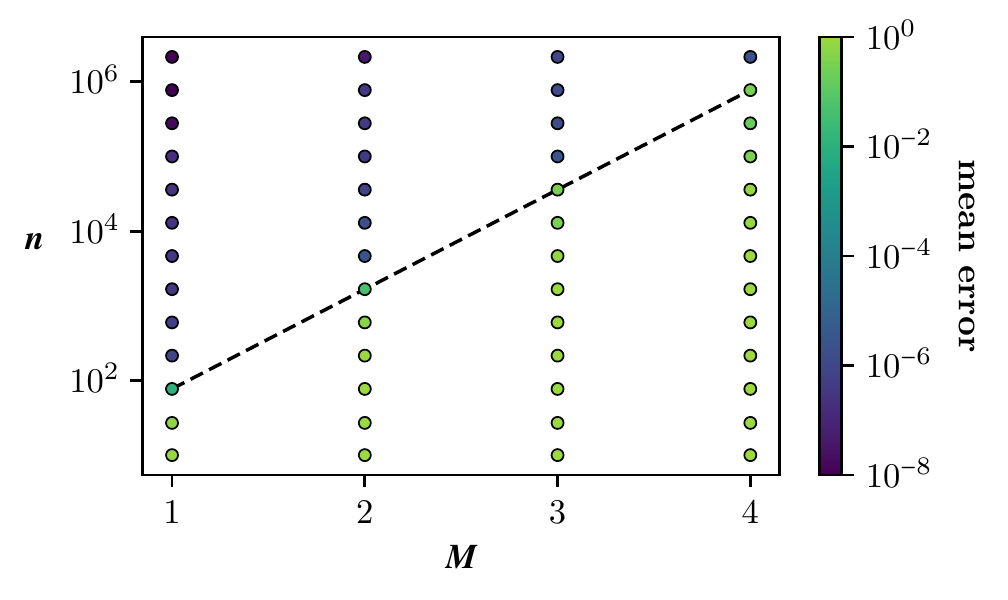}
        \caption{The sought function is defined by $C_{k_1,...,k_M} = 1$.}
        \label{fig:recovery_phase_diagram:worst_case}
    \end{subfigure}%
    \begin{subfigure}[t]{0.5\textwidth}
        \centering
        \includegraphics[width=\textwidth]{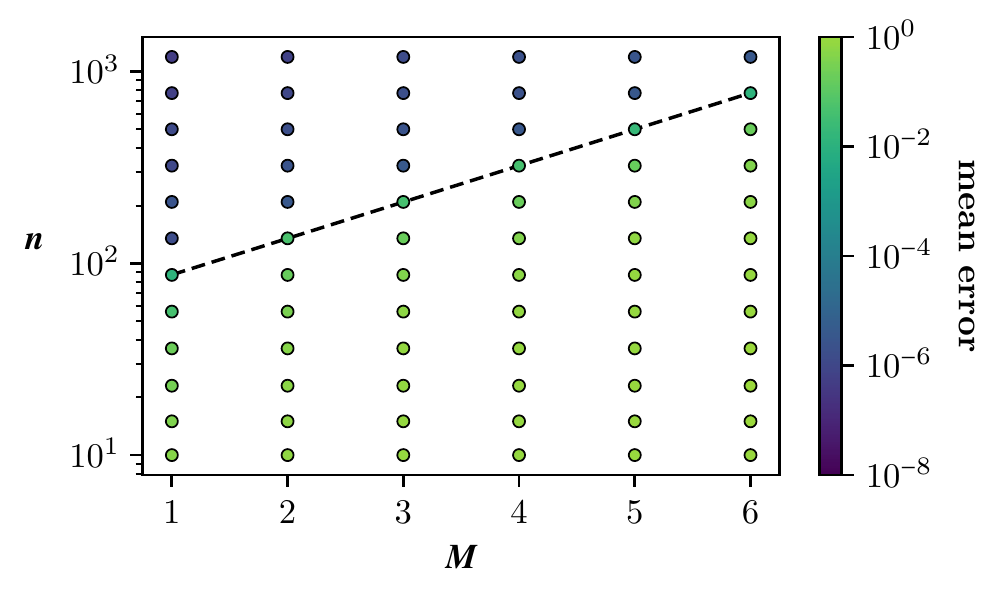}
        \caption{The sought function is $\exp\pars{y_1+\cdots+y_M}$.}
        \label{fig:recovery_phase_diagram:exp}
    \end{subfigure}
    \caption{Two phase diagram for the recovery of multivariate polynomials in the tensor product basis of Legendre polynomials.
    For every order $M$ and number of samples $n$, the mean error is computed as the relative $L^2$-error of the approximation, averaged over $20$ independent realizations.
    A hard-thresholding algorithm (cf.~\cite{eigel2018vmc0}) was used for recovery.
    Note that the optimal coefficient tensor $C\in\pars{\mbb{R}^{15}}^{\otimes M}$ is always of rank $1$.}
    \label{fig:recovery_phase_diagram}
\end{figure}

\begin{remark} \label{rmk:restriction}
    \leavevmode\setlength{\parskip}{0px}
    In light of \cref{thm:K_properties:union}, it stands to reason that the problem arising from equation~\eqref{eq:chain} can be tackled by restricting the model class $\mcal{M}$ to a subclass $\mcal{N}\subseteq\mcal{M}$ that still contains $u_{\mcal{M}}$.
    This is presumably the reason for the practical success of many algorithms for low-rank approximation, which remain in a small neighborhood of the best approximation and the initial guess during their execution.
    
    This gives a heuristic argument as to why the block alternating steepest descent algorithm in~\cite{eigel2018vmc0,ESTW19} and the stabilized ALS algorithm in~\cite{Grasedyck2019SALSA,Kraemer2020thesis} are so successful in practice.
    Both algorithms generates the sequence of iterates $\braces{v_l}_{l\in\mbb{N}}$ by refining the initial guess $v_0$.
    We can thus define the corresponding sequence of approximation errors $r_l := \norm{u_{\mcal{M}} - v_l}$ and the subclasses $\mcal{M}_l := \mcal{M}\cap B\pars{u,r_l}$.
    If $\mfrak{K}_{\braces{u}-\mcal{M}_l}$ is small enough, then arguably $r_{l+1} \le r_l$ and consequently $\mfrak{K}_{\braces{u}-\mcal{M}_{l+1}} \le \mfrak{K}_{\braces{u} - \mcal{M}_l}$.
    In \cref{sec:local_variation_constant} we show that it is important that the rank of $u$ is not overestimated.
    Both algorithms ensure this by starting with a rank of $1$ and successively increasing the rank while testing for divergence on a validation set.
    The majority of the problems in~\cite{ESTW19} possess highly regular solutions and allow for the computation of a descent initial guesses, resulting in a relatively small initial approximation error $r_0$.
    
    Note that the model classes $\mcal{M}_l$ are chosen implicitly by the algorithm and do not enter the implementation.
\end{remark}
This remark is illustrated by the following example.
\begin{example} \label{ex:KU_rank-1} \leavevmode
    Recall the definition of $\mcal{V} = \mbb{R}^{d_1\times d_2}$ and $\norm{\bullet} = \frac{1}{\sqrt{d_1d_2}}\norm{\bullet}_{\mathrm{Fro}}$ from \cref{ex:matrix_completion} and let $\mcal{M}\subseteq\mcal{V}$ be the set of rank-$1$ matrices.
    For any pair $\pars{i,j}\in Y = \bracs{d_1}\times\bracs{d_2}$ define the two matrices $\Delta = e^{\vphantom{\intercal}}_ie_j^\intercal$ and $\bbone = \boldsymbol{1}\boldsymbol{1}^\intercal$.
    Then
    \begin{align}
        \mfrak{K}_{\braces{\bbone}-\mcal{V}}\pars{i,j}
        &= \mfrak{K}_{\mcal{V}}\pars{i,j}
        = d_1d_2 \sup_{v\in\mcal{V}} \frac{\abs{v_{ij}}^2}{\norm{v}_{\mathrm{Fro}}^2}
        = d_1d_2 \frac{1}{\norm{\Delta}_{\mathrm{Fro}}^2}
        = d_1d_2 \\
        \mfrak{K}_{\braces{\bbone}-\mcal{M}}\pars{i,j}
        &\ge \mfrak{K}_{\braces{\bbone}-\inner{\Delta}}\pars{i,j}
        = d_1d_2 \sup_{r\in\mbb{R}} \frac{\abs{1-r}^2}{\norm{\bbone-r\Delta}_{\mathrm{Fro}}^2}
        \ge d_1d_2 \lim_{r\to\infty} \frac{\abs{1-r}^2}{\norm{\bbone-r\Delta}_{\mathrm{Fro}}^2}
        = d_1d_2
    \end{align}
    Thus, $\mfrak{K}_{\braces{\bbone} - \mcal{M}} = \mfrak{K}_{\braces{\bbone} - \mcal{V}}$ as stated in equation~\eqref{eq:chain}.
    Note however, that this only works because $\Delta\in\mcal{M}$ can be scaled such that $\bbone+r\Delta\approx r\Delta$.
    This can be prevented, if $\mcal{M}$ is restricted to the model class $\mcal{N}:=\mcal{M}\cap B\pars{\bbone,R}$ for $R>0$.
\end{example}

\section{Restriction to local model classes} \label{sec:local_variation_constant}

Even with a very good initial guess, the idea from \cref{rmk:restriction} can only work when the neighborhood of the best approximation $u_{\mcal{M}}$ exhibits a sufficiently small variation function.
In this section we derive a lower bound for this variation function for a wide range of model classes and, in doing so, discover three preconditions that any iterative approximation algorithm \textbf{must} satisfy to be successful.
For any $v\in\mcal{M}$ and $r>0$ we consider a local version of the model class $\mcal{M}$, namely
\begin{equation}
    \mcal{M} \cap B\pars{v,r} = \braces{w \in \mcal{M} : \norm{v - w} \le r} .
\end{equation}

In the following we show that, under certain conditions, the \emph{local model class} $\mcal{M}\cap B\pars{v,r}$ can be well approximated by a ball of radius $r$ in a low-dimensional, affine subspace of $\inner{\mcal{M}}$.
We use this fact to estimate the corresponding \emph{local variation function}
\begin{equation}
    \mfrak{K}^{\mathrm{loc}}_{\mcal{M},v} := \lim_{r\to0} \mfrak{K}_{\mcal{M}\cap B\pars{v,r}} .
\end{equation}
Due to the monotonicty of $\mfrak{K}_{\bullet}$
(cf.~\cref{thm:K_properties:union}), this limit provides a lower bound for the variation function in any neighborhood of $v$.
\note{%
For any neighborhood $\mcal{N}$ of $u$ and any sequence $r_j\to0$ there exists $J$ such that $\mcal{N}\supseteq\mcal{M}^{\mathrm{loc}}_{u,r_j}$ for all $j\ge J$.
Thus
\begin{equation}
    K\pars{U\pars{u - \mcal{N}}} \ge K^{\mathrm{loc}}_{u,r_K}\pars{\mcal{M}} \ge K^{\mathrm{loc}}_{u,0}\pars{\mcal{M}} .
\end{equation}%
}%
Moreover, the continuity of $\mfrak{K}_{\bullet}$ (cf.~\crefrange{thm:K_properties:continuity:compact_sets}{thm:K_properties:continuity:cones}) implies that the variation function approaches this limit if the neighborhood is sufficiently small.

This definition allows us to formalize the first precondition for a successful recovery.
\begin{precondition} \label{imp:regular_solution}
    $u_{\mcal{M}}$ has to be sufficiently regular in the sense that $\norm{\mfrak{K}^{\mathrm{loc}}_{\mcal{M},u_{\mcal{M}}}}_{w,\infty}$ must be small.
\end{precondition}

The local variation function can be computed explicitly, if the model class can be linearly approximated in a neighborhood of $v$.
We therefore define the concept of local linearizability.
\begin{definition} \label{def:loc_mfld}
    We call a set $\mcal{M}$ \emph{locally linearizable} in $v\in\mcal{M}$ if, for sufficiently small $r$, the set $\mcal{M}\cap B\pars{v,r}$ is an embedded, differentiable submanifold of a Euclidean space with positive reach.
    \note{Any $C^3$-manifold can be embedded into a finite-dimensional Euclidean space due to Nash's embedding theorem. So this requirement become trivial if the regularity is sufficiently large.}
    The reach $\rch\pars{\mcal{M}}$ of a manifold $\mcal{M}$ is the largest number such that any point at distance less than $\rch\pars{\mcal{M}}$ from $\mcal{M}$ has a unique nearest point on $\mcal{M}$.
\end{definition}
\begin{example}
    Linear spaces are classical examples of $C^\infty$-manifolds with infinite reach.
\end{example}
\begin{example} \label{ex:loc_mfld:sparse}
    The model class of $s$-sparse vectors $\mcal{M} := \braces{w\in\mbb{R}^d : \abs{\operatorname{supp}\pars{w}}\le s}$ is locally linearizable in any $v\in\mcal{M}$ with $\abs{\operatorname{supp}\pars{v}}=s$.
    To see this, let $r < \min \braces{\abs{v_j} \,:\, j\in\supp\pars{v}}$ and observe that
    \begin{equation}
        \mcal{M}\cap B\pars{v,r}
        = \braces{w\in\mbb{R}^d : \supp\pars{w}=\supp\pars{v}} \cap B\pars{v,r}
    \end{equation}
    is a $C^\infty$-manifold with infinite reach.
\end{example}
\begin{example} \label{ex:loc_mfld:low_rank}
    \leavevmode\setlength{\parskip}{0px}
    Consider the set $\mcal{M} := \braces{w\in\mbb{R}^{d_1\times d_2} : \operatorname{rank}\pars{w}\le R}$, of matrices with a rank that is bounded by $R$.
    $\mcal{M}$ is locally linearizable in all $v\in\mcal{M}$ with $\operatorname{rank}\pars{v}=R$.
    To see this, let $\sigma_{j}\pars{v}$ denote the $j$\textsuperscript{th} largest singular value of $v$ and let $r < \sigma_{R}\pars{v}$.
    Then $\mcal{M}\cap B\pars{v,r}$ is a $C^\infty$-submanifold of the manifold of rank-$R$ matrices.
    $\operatorname{rch}\pars{\mcal{M}} \ge \frac{r}{2}$, since for any matrix $w$ with $\norm{v-w}_{\mathrm{Fro}} \le \frac{r}{2}$,
    \begin{equation}
        \sigma_{R}\pars{w}
        \ge \sigma_{R}\pars{v} - \frac{r}{2}
        > 0
        \qquad\text{and}\qquad
        \sigma_{R}\pars{w} - \sigma_{R+1}\pars{w}
        \ge \sigma_{R}\pars{v} - r
        > 0 .
    \end{equation}
    This means, that $\operatorname{rank}\pars{w} \ge R$ and that its best rank-$R$ approximation, given by the truncated singular value decomposition, is uniquely defined.
\end{example}

\begin{proposition}[Lemma 4.3 in \cite{rataj_2002_reach}, Theorem 3.10 in \cite{colesanti_2010_reach}, or Proposition 4 in \cite{cockreham_2017_reach}] \label{prop:reach_increases}
    If $R = \rch\pars{\mcal{M}\cap B\pars{v,r_0}}$ and $r\le \min\braces{r_0, R}$, then $\rch\pars{\mcal{M}\cap B\pars{v,r}} \ge R$.
\end{proposition}

A common intuition for a differentiable manifold is the the interpretation as a hypersurfaces that can be locally approximated by a Euclidean space which is called the tangent space.
This intuition is formalized in the following theorem.

\begin{theorem} \label{thm:convergence_M}
    Let $\mcal{M}$ be locally linearizable in $v$ and $R=\rch\pars{\mcal{M}\cap B\pars{v,r_0}} > 0$.
    Then $d_{\mathrm{H}}\pars{\mcal{M}\cap B\pars{v,r}, \pars{v+\mbb{T}_v\mcal{M}}\cap B\pars{v,r}} \le \frac{r^2}{2R}$ for any $r\le \min\braces{r_0, R}$.
    \proofappendix{proof:thm:convergence_M}
\end{theorem}

\begin{remark} \label{rmk:dH_ball}
    Note that $d_{\mathrm{H}}\pars{X\cap B\pars{v,r},Y\cap B\pars{v,r}} \le 2r$ because $X\cap B\pars{v,r}\subseteq B\pars{v,r}$ and $Y\cap B\pars{v,r}\subseteq B\pars{v,r}$ for any two sets $X$ and $Y$.
\end{remark}
Looking at \cref{rmk:dH_ball} it is clear that $X\cap B\pars{0,r} \to Y\cap B\pars{0,r}$ does not imply $U\pars{X\cap B\pars{0,r}} \to U\pars{Y\cap B\pars{0,r}}$ for general sets $X$ and $Y$.
For locally linearizable sets, however, this is indeed the case as is shown in the subsequent theorem.

\begin{theorem} \label{thm:convergence_UM}
    Let $\mcal{M}$ be locally linearizable in $v$ and $R=\rch\pars{\mcal{M}\cap B\pars{v,r_0}} > 0$.
    Then $d_{\mathrm{H}}\pars{U\pars{\mcal{M}\cap B\pars{v,r}-v}, U\pars{\mbb{T}_v\mcal{M}}} \le \frac{r}{R}$ for any $r\le \min\braces{r_0, R}$.
    \proofappendix{proof:thm:convergence_UM}
\end{theorem}

This motivates the following corollary.

\begin{corollary} \label{cor:KU_manifold_limit}
    Assume that $\mcal{M}$ is locally linearizable in $v$.
    Then $\mfrak{K}^{\mathrm{loc}}_{\mcal{M},v} = \mfrak{K}_{\mbb{T}_v\mcal{M}}$.
\end{corollary}
\begin{proof}
    Recall from the proof of \crefrange{thm:K_properties:continuity:compact_sets}{thm:K_properties:continuity:cones} that $\mfrak{K}_{\bullet} = F\circ\,U$ where $F = \operatorname{sqr}\circ \sup\circ \operatorname{abs}$ and $\operatorname{sqr}$, $\sup$, and $\operatorname{abs}$ are defined in \crefrange{eq:sqr}{eq:sup}\vphantom{\eqref{eq:sqr}\eqref{eq:abs}\eqref{eq:sup}}.\note{Maybe use $\operatorname{SUP}$ instead of $\sup$.}
    By \crefrange{lem:abs_continuous}{lem:sqr_continuous} $F : \mfrak{P}\pars{\mcal{V}_{w,\infty}} \to \mcal{V}_{w^2,\infty}$ is continuous with respect to the Hausdorff pseudometric.
    The continuity of $F$ and \cref{thm:convergence_UM} then imply the first assertion,
    \begin{equation}
        \lim_{r\to 0} F\pars{U\pars{\mcal{M}\cap B\pars{v,r}-v}}
        = F\pars{\lim_{r\to 0} U\pars{\mcal{M}\cap B\pars{v,r}-v}}
        = F\pars{U\pars{\mbb{T}_v\mcal{M}}} .
    \end{equation}
    The continuity of $\norm{\bullet}_{w,\infty}$ implies the second assertion.
\end{proof}

We conclude this section with two examples in which we use the preceding theorems to derive bounds for the local variation function $\mfrak{K}^{\mathrm{loc}}_{\mcal{M}, v}$ of low-rank matrices.
The following proposition will be a useful tool for this.
\begin{proposition}
    Let $\mcal{M}$ be conic and locally linearizable in $v$. Then $v\in\mbb{T}_v\mcal{M}$.
\end{proposition}
\begin{proof}
    Fix $r>0$ such that $\mcal{M}\cap B\pars{v,r}$ is an embedded, differentiable submanifold and consider the path $\gamma : \pars{-r,r} \to \mcal{M}\cap B\pars{v,r}$, defined by $\gamma\pars{x} := \pars{1+x}v$.
    Since $\gamma\pars{0} = v$ and $\gamma'\pars{0} = v$, it represents the tangent vector $v\in\mbb{T}_v\mcal{M}$.
\end{proof}

\begin{example} \label{ex:KU_loc_rank-1}
    \leavevmode\setlength{\parskip}{0px}
    As in \cref{ex:matrix_completion}, let $\mcal{V} = \mbb{R}^{d_1\times d_2}$ and $\norm{\bullet} = \frac{1}{\sqrt{d_1d_2}}\norm{\bullet}_{\mathrm{Fro}}$, and let $\mcal{M}\subseteq\mcal{V}$ be the set of rank-$1$ matrices.
    We now compute the local variation function $\mfrak{K}^{\mathrm{loc}}_{\mcal{M}, v}$ for $v = w_{\mathrm{L}}\otimes w_{\mathrm{R}}\in\mcal{M}$ with $w_{\mathrm{L}}\in\mbb{R}^{d_1}$ and $w_{\mathrm{R}}\in\mbb{R}^{d_2}$.

    Since $\mbb{T}_{w_{\mathrm{L}}\otimes w_{\mathrm{R}}}\mcal{M} = \mcal{W}_{\mathrm{L}} \oplus \mcal{W}_{\mathrm{R}}$ with $\mcal{W}_{\mathrm{L}} := \inner{w_{\mathrm{L}}}\otimes\mbb{R}^{d_2}$ and $\mcal{W}_{\mathrm{R}} := \inner{w_{\mathrm{L}}}^{\perp} \otimes \inner{w_{\mathrm{R}}}$,  \cref{cor:KU_manifold_limit} yields $\mfrak{K}^{\mathrm{loc}}_{\mcal{M}, v} = \mfrak{K}_{\mbb{T}_{v}\mcal{M}} = \mfrak{K}_{\mcal{W}_{\mathrm{L}} \oplus \mcal{W}_{\mathrm{R}}}$.
    Using \cref{thm:K_properties} we can bound this by
    \begin{align}
        \mfrak{K}_{\mcal{W}_{\mathrm{L}}\oplus\mcal{W}_{\mathrm{R}}}
        &= \mfrak{K}_{\mcal{W}_{\mathrm{L}}} + \mfrak{K}_{\mcal{W}_{\mathrm{R}}} \\
        \mfrak{K}_{\mcal{W}_{\mathrm{L}}} &= d_2\mfrak{K}_{\braces{w_{\mathrm{L}}}} \\
        \mfrak{K}_{\mcal{W}_{\mathrm{R}}} &\le d_1\mfrak{K}_{\braces{w_{\mathrm{R}}}} .
    \end{align}
    Moreover, since $\mcal{M}$ is conic, $v\in\mbb{T}_v\mcal{M}$.
    Hence, $\mfrak{K}_{\braces{v}} \le \mfrak{K}^{\mathrm{loc}}_{\mcal{M},v} \le d_2\mfrak{K}_{\braces{w_{\mathrm{L}}}} + d_1\mfrak{K}_{\braces{w_{\mathrm{R}}}}$.
     
    Finally, we apply this bound to the rank-$1$ matrices $\bbone = \boldsymbol{1}\boldsymbol{1}^\intercal$ and $\Delta = e_1e_1^\intercal$ from \cref{ex:KU_rank-1} and arrive at the estimates
    \begin{equation}
        1 \le 
        \mfrak{K}^{\mathrm{loc}}_{\mcal{M},\bbone}
        \le d_1 + d_2
        \qquad\text{and}\qquad
        d_1d_2 \le
        \mfrak{K}^{\mathrm{loc}}_{\mcal{M},\Delta}
        \le d_1d_2 .
    \end{equation}
    
    Concrete values of $\mfrak{K}_{\braces{\bbone}-\mcal{M}\cap B\pars{\bbone, r}}$ for different dimensions $d=d_1=d_2$ and different values of $r$ are estimated numerically in \cref{fig:local_variation_constant}.
    It can be seen, that indeed $\mfrak{K}_{\braces{\bbone}-\mcal{M}\cap B\pars{\bbone, r}}\in\mcal{O}\pars{d^2}$ for $r\to\infty$ and $\mfrak{K}_{\braces{\bbone}-\mcal{M}\cap B\pars{\bbone, r}}\in\mcal{O}\pars{d}$ for $r\to 0$.
    The algorithm used to generate this plot is derived in~\cref{sec:algorithm_KUbbone}.
\end{example}

\begin{figure}[ht]
    \centering
    \includegraphics[width=\textwidth]{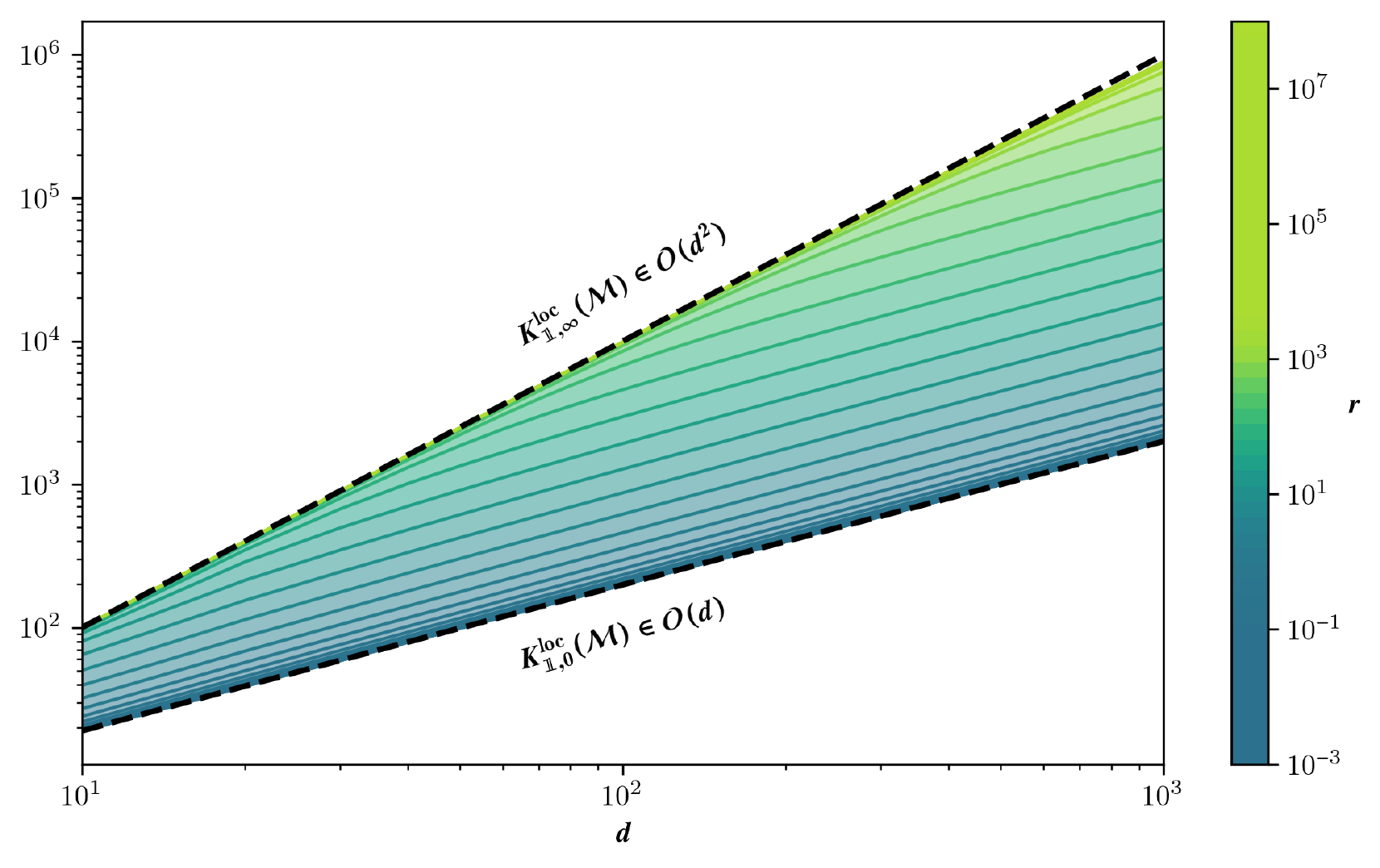}
    \caption{The local variation constant $K^{\mathrm{loc}}_{\bbone,r}\pars{\mcal{M}}$ of the set of rank-$1$ matrices $\mcal{M}\subseteq\mbb{R}^{d\times d}$ for different $d$ and $r$.}
    \label{fig:local_variation_constant}
\end{figure}

From the previous example we conclude the second prerequisite.
\begin{precondition} \label{imp:initial_guess}
    An initial guess $u_{\mathrm{init}}\in\mcal{M}\cap B\pars{u_{\mcal{M}},r}$ is required, for which $r$ is sufficiently small.
\end{precondition}

\begin{remark} \label{rmk:KU_loc_rank-r}
    \leavevmode\setlength{\parskip}{0px}
    The arguments from \cref{ex:KU_loc_rank-1} can also be used to derive bounds for the variation function of the set $\mcal{M}$, of matrices with rank bounded by $R$, from \cref{ex:loc_mfld:low_rank}.
    If $v\in\mcal{M}$ satisfies $\operatorname{rank}\pars{v} = R$, then 
    \begin{equation}
        \mfrak{K}_{\braces{v}}
        \le \mfrak{K}^{\mathrm{loc}}_{\mcal{M},v}
        \le d_2\mfrak{K}_{\inner{w_{\mathrm{L},1}^{\vphantom{\intercal}}, \ldots, w_{\mathrm{L},R}^{\vphantom{\intercal}}}} + d_1\mfrak{K}_{\inner{w_{\mathrm{R},1}^{\vphantom{\intercal}}, \ldots, w_{\mathrm{R},R}^{\vphantom{\intercal}}}},
    \end{equation}
    where $v = \sigma_1 w_{\mathrm{L},1}^{\vphantom{\intercal}} w_{\mathrm{R},1}^\intercal + \cdots + \sigma_R w_{\mathrm{L},R}^{\vphantom{\intercal}} w_{\mathrm{R},R}^\intercal$
    is the singular value decomposition of $v$.

    Since $\mfrak{K}_{\inner{w_{\mathrm{L},1}^{\vphantom{\intercal}}, \ldots, w_{\mathrm{L},R}^{\vphantom{\intercal}}}}$ and $\mfrak{K}_{\inner{w_{\mathrm{R},1}^{\vphantom{\intercal}}, \ldots, w_{\mathrm{R},R}^{\vphantom{\intercal}}}}$ measure how ``spread out'' the singular vectors of $v$ are, the local variation constant can be interpreted as an analogue of the incoherence of the matrix $v$, as known from classical matrix completion.
    
    Indeed, in~\cite[Section~1.5]{Candes2010convexRelaxationMatrixCompletion}, the authors discuss a class of rank-$R$ matrices in $\mbb{R}^{d\times d}$, where the incoherence condition is satisfied with high probability.
    The matrices in this class satisfy 
    \begin{equation}
        \max\braces{\norm{w_{\mathrm{L},k}}_\infty, \norm{w_{\mathrm{R},k}}_\infty} \le \sqrt{\mu/d}
    \end{equation}
    for all $k=1,\ldots,R$, which implies that $\mfrak{K}^{\mathrm{loc}}_{\mcal{M},v} \le 2Rd\mu$.
    This means that the local variation function of every matrix in this class is bounded.

    The bounds from \cref{ex:KU_loc_rank-1} can also be extended to hierarchical tensor formats but the relation to the corresponding incohrence conditions is not as straight-forward.
    This may be due to the fact that there is no canonical definition of a tensor rank.
    In~\cite{yuan2014tensor}, for example, the rank
    \begin{equation}
        \overline{r}\pars{x} := \sqrt{\pars{r_1r_2d_3 + r_1r_3d_2 + r_2r_3d_1}/\pars{d_1+d_2+d_3}},
    \end{equation}
    is used, which does not correspond to any class of tensor networks.
    
    It would be quite interesting to see if the discussed relation of the incoherence condition to the variation function can be strengthened and even extended to the tensor case.
\end{remark}

We conclude this section with an example that highlights the limitations of the result.

\begin{example} \label{ex:KU_loc_rank-2}
    Recall the definition of $\mcal{V} = \mbb{R}^{d_1\times d_2}$ and $\norm{\bullet} = \frac{1}{\sqrt{d_1d_2}}\norm{\bullet}_{\mathrm{Fro}}$ from \cref{ex:matrix_completion},
    let $\mcal{M}\subseteq\mcal{V}$ be the set of \textbf{rank-$\boldsymbol{2}$} matrices and let $v$ be any rank-$1$ matrix.
    To compute the local variation function, let $\pars{i,j}\in Y = \bracs{d_1}\times\bracs{d_2}$ and define the matrix $\Delta = e_ie_j^\intercal$ as in \cref{ex:KU_rank-1}.
    Observe that $v - r\Delta \in \mcal{M}\cap B\pars{v,r}$ and therefore
    \begin{equation}
        \mfrak{K}_{\braces{v} - \mcal{M}\cap B\pars{v,r}}\pars{i,j}
        \ge \mfrak{K}_{\braces{v-\pars{v-r\Delta}}}\pars{i,j}
        = \mfrak{K}_{\braces{\Delta}}\pars{i,j}
        = d_1d_2
        = \mfrak{K}_{\mcal{V}}\pars{i,j}
    \end{equation}
    for any $r>0$ .
    This implies, that $\mfrak{K}^{\mathrm{loc}}_{\mcal{M},v} = d_1d_2 = \mfrak{K}_{\mcal{V}}$ and shows, that overestimating the rank blows up the variation function.
\end{example}

This example provides us with the final prerequisit.
\begin{precondition} \label{imp:overestimation_blowup}
    $\mcal{M}$ must be locally linearizable in $u_{\mcal{M}}$.
    If $\mcal{M}$ is a model class of tensor networks, then the corresponding rank of $u$ must not be overestimated.
\end{precondition} 

\section{A modified ALS} \label{sec:restricted_als}

Although \cref{rmk:restriction} provides a heuristic argument for why state-of-the-art algorithms work so well, {\bf guaranteeing} the Preconditions \ref{imp:regular_solution}, \ref{imp:initial_guess} and \ref{imp:overestimation_blowup}, that are necessary for a good recovery, is unrealistic in most practical applications.
Even if the best approximation $u_{\mcal{M}}$ is known to have high regularity (in the sense of \cref{imp:regular_solution}), finding an appropriate initial guess can be a challenging task and a tight bound for the rank of $u_{\mcal{M}}$ is rarely known.
This means that we can not rely on the algorithm to stay in a sufficiently regular subclass of $\mcal{M}$.
To guarantee a successful recovery we propose to design specialized algorithms that explicitly enforce a small variation function.
This idea can be found in many well-known algorithms.

\begin{example} \label{ex:dp:regression}
    \leavevmode\setlength{\parskip}{0px}
    Consider the setting of polynomial regression from \cref{ex:polynomial_regression}.
    Denote by $L_m$ the $m$\textsuperscript{th} normalized Legendre polynomial and define the linear model space $\mcal{M} := \inner{L_m : \pars{m+1}^2 \le r}$.
    In equation~\eqref{KU_properties:dim} it is shown that $\norm{\mfrak{K}_{\braces{u_{\mcal{M}}} - \mcal{M}}}_\infty = \norm{\mfrak{K}_{\mcal{M}}}_\infty \le r$.
    This allows us to bound the variation function by restricting the maximal degree of the polynomials.
\end{example}

\begin{example} \label{ex:dp:sparse_regression}
    \leavevmode\setlength{\parskip}{0px}
    Consider the problem of sparse regression from \cref{ex:sparse_regression}.
    The method of $\omega$-weighted $\ell^1$-minimization~\cite{Rauhut2016weighted_l1,bouchot_2015_CSPG} works by solving the optimization problem
    \begin{equation}
        \text{minimize } \norm{v}_{\omega,1} \quad\text{subject to}\quad \norm{u-v}_n = 0 ,
    \end{equation}
    which is a convexified version of the problem
    \begin{equation}
        \text{minimize } \norm{v}_{\omega,0} \quad\text{subject to}\quad \norm{u-v}_n = 0 .
    \end{equation}
    
    Now observe that, by triangle and Cauchy--Schwarz inequality, %
    \begin{equation}
         \norm{v}_\infty
         \le \sum_{i=1}^d \abs{v_i}\norm{B_i}_\infty
         \le \norm{v} \norm{v}_{0,\omega},
    \end{equation}
    which implies $\mfrak{K}_{\braces{v}} \le \norm{v}_{\omega,0}^2$.
    This means that $\omega$-weighted $\ell^1$-minimization restricts the solutions to a model class in which the variation function is small.
\end{example}

This restriction to a model class with small variation function is, however, not the case for nuclear norm minimization, as is demonstrated in the subsequent example.

\begin{counterexample} \label{ex:dp:nuclear_norm_minimization}
    \leavevmode\setlength{\parskip}{0px}
    Consider the problem of matrix recovery from \cref{ex:matrix_completion} and consider the model class of
    rank-$R$ matrices.
    Nuclear norm minimization~\cite{Candes2010convexRelaxationMatrixCompletion,recht_2010_nuclear_norm_minimization,yuan2014tensor} works by solving the optimization problem
    \begin{equation}
        \text{minimize } \norm{v}_{*} \quad\text{subject to}\quad \norm{u-v}_n = 0 ,
    \end{equation}
    which is a convexified version of the problem
    \begin{equation}
        \text{minimize } \operatorname{rank}\pars{v} \quad\text{subject to}\quad \norm{u-v}_n = 0 .
    \end{equation}
    We have seen in \cref{sec:convergence_bounds}, that the rank of a model-class of matrices does not have any influence on the variation function and it is easy to conceive matrices $v\in\mcal{M}$ with small nuclear norm but with large variation function.
    This means that nuclear norm minimization does not minimize a bound for the variation function.

    Indeed, the application of the triangle and Cauchy--Schwarz inequality, as done in \cref{ex:dp:sparse_regression}, yields for any matrix $v\in\mcal{M}$ that
    \begin{align}
        \norm{v}_{\infty}
        \le \sum_{j=1}^R \sigma_j \norm{w_{\mathrm{L},j}}_\infty \norm{w_{\mathrm{R},j}}_\infty
        \le \sqrt{d_1d_2} \norm{v} \pars*{\sum_{j=1}^{R} \norm{w_{\mathrm{L},j}}_\infty^2 \norm{w_{\mathrm{R},j}}_\infty^2}^{1/2},
    \end{align}
    where $v = \sigma_1 w_{\mathrm{L},1} w_{\mathrm{R},1}^\intercal + \cdots + \sigma_R w_{\mathrm{L},R} w_{\mathrm{R},R}^\intercal$ is the singular value decomposition of $v$.
    This expression provides an explicit bound for the variation function which, however, is not commonly minimized.
\end{counterexample}

\note{Im Microstep können wir die RIP testen indem wir die Gram-Matrix für die lokale Basis SVD'n.}

The remainder of this section showcases the idea of explicitly restricting the variation function in the optimization algorithm.
This is done by modifying the \emph{alternating least squares} (ALS) algorithm for the empirical best-approximation in the model class of low-rank tensor networks.

\subsection{The standard ALS algorithm}

This section provides a brief overview of the \emph{alternating least squares} (ALS) algorithm introduced in~\cite{oseledets_2011_tensor_trains,holtz_alternating_2012}.

Let the space $\mcal{V}_{d}\subseteq L^2\pars{Y,\rho}$ be spanned by the $d\in\mbb{N}$ orthonormal basis functions $\braces{b_k}_{k=1,\ldots,d}$ and recall from \cref{sec:tensor_networks} that every function $v\in\mcal{V}_{d}^{\otimes M} \subseteq L^2\pars{Y^M, \rho^{\otimes M}}$ can be represented graphically as
\begin{equation} \label{eq:VTT}
\begin{tikzpicture}[baseline=(current bounding box.center)]
    \node[core, label={[anchor=east, label distance=3]left:$v\pars{y_1,\ldots,y_M}$}] (V) {};
    \node[right=0.4 of V] (eqVC) {$=$};
    \node[right=4.6 of eqVC, core, label={[anchor=east, label distance=3]left:$\pars{\boldsymbol{V}, \boldsymbol{b}\pars{y_1}\otimes\cdots\otimes\boldsymbol{b}\pars{y_M}}_{\mathrm{Fro}}$}] (VC) {};
    \node[right=0.4 of VC] (eqV) {$=$};
    \node[core, right=0.4 of eqV,, label={[anchor=south, label distance=3]above:$\boldsymbol{V}_1$}] (V1) {};
    \node[core, below=0.5 of V1, label={[anchor=north, label distance=-2]below:$\boldsymbol{b}\pars{y_1}$}] (d1) {};
    \node[core, right=1.25 of V1, label={[anchor=south, label distance=3]above:$\boldsymbol{V}_2$}] (V2) {};
    \node[core, below=0.5 of V2, label={[anchor=north, label distance=-2]below:$\boldsymbol{b}\pars{y_2}$}] (d2) {};
    \node[core, right=1.75 of V2, label={[anchor=south, label distance=3]above:$\boldsymbol{V}_{M-1}$}] (V3) {};
    \node[core, below=0.5 of V3, label={[anchor=north, label distance=-2]below:$\boldsymbol{b}\pars{y_{M-1}}$}] (d3) {};
    \node[core, right=1.25 of V3, label={[anchor=south, label distance=3]above:$\boldsymbol{V}_M$}] (V4) {};
    \node[core, below=0.5 of V4, label={[anchor=north, label distance=-2]below:$\boldsymbol{b}\pars{y_M}$}] (d4) {};

    \draw[contraction] (V1) edge (d1) -- (V2) edge (d2)
                        (V2) -- ($(V2)+(0.3,0)$)
                        ($(V3)-(0.3,0)$) -- (V3)
                        (V3) edge (d3) -- (V4) edge (d4);
    \draw[contractionDots] ($(V2)+(0.3,0)$) -- ($(V3)-(0.3,0)$);

    \Core{(V)}
    \Core{(VC)}
    \Core{(V1)}
    \Core{(V2)}
    \Core{(V3)}
    \Core{(V4)}
\end{tikzpicture}
\end{equation}
where the $\boldsymbol{V}_k$'s are the components of the tensor train representation of the coefficient tensor $\boldsymbol{V}\in\pars{\mbb{R}^d}^M$ and where $\boldsymbol{b}\pars{x} = \bracs{b_1\pars{x}, \ldots, b_d\pars{x}}^\intercal$ denotes the vector of basis function, evaluated at $x$.

Now consider the model class 
\begin{equation}
    \mcal{M} := \{ v\in \mcal{V}_{d}^{\otimes M} \,:\, \text{TT-rank}(\boldsymbol{V})\le r\}
\end{equation}
of functions in $\mcal{V}_{d}^{\otimes M}$ with a coefficient tensor $\boldsymbol{V}$ that can be represented in the tensor train format with a rank of at most $r$.
The minimization problem \eqref{eq:min_emp} can then be reformulated as
\begin{equation}
    \min_{v\in\mcal{M}} \frac{1}{n} \sum_{i=1}^n w\pars{y^i} \abs{u\pars{y^i} - \pars{\boldsymbol{V}, \boldsymbol{b}\pars{y^i_1}\otimes\cdots\otimes\boldsymbol{b}\pars{y^i_M}}_{\mathrm{Fro}}}^2
\end{equation}
Defining $\boldsymbol{u}\in\mbb{R}^n$ by $\boldsymbol{u}_i = \sqrt{w\pars{y^i}} u\pars{y^i}$ and $\boldsymbol{M} \in\mcal{L}\pars{\mbb{R}^{d_1\times\cdots\times d_M}, \mbb{R}^n}$ by $\pars{\boldsymbol{MV}}_i = \sqrt{w\pars{y^i}}\pars{\boldsymbol{V}, \boldsymbol{b}\pars{y^i_1}\otimes\cdots\otimes\boldsymbol{b}\pars{y^i_M}}_{\mathrm{Fro}}$, this can be written as
\begin{equation} \label{eq:min_emp_coef}
    \underset{v\in\mcal{M}}{\text{minimize}}\ \norm{\boldsymbol{u} - \boldsymbol{M}\boldsymbol{V}}_{2}^2,
\end{equation}
where the tensor train representation of $\boldsymbol{V}$ allows for an efficient evaluation of the operator $\boldsymbol{M}$.

The ALS method solves \eqref{eq:min_emp_coef} by refining an initial guess in a sequence of \emph{microsteps} which optimize a single component tensor $\boldsymbol{V}_m$ at a time.
To formalize this, we define for every $m=1,\ldots,M$ the operator $\xhat{\boldsymbol{V}}_m : \mcal{L}\pars{\mbb{R}^{r_{m-1}\times d_m\times r_m}, \mbb{R}^{d_1\times\cdots\times d_M}}$ by 
\begin{equation}
\begin{tikzpicture}[baseline=(current bounding box.center)]
    \node (U) {$\xhat{\boldsymbol{V}}_m$};
    \node[right=0.2 of U] (eqU) {$=$};
    \node[core, right=0.4 of eqU] (U1) {};
    \node[core, below=0.5 of U1, label={[anchor=north, label distance=-2]below:$\scriptstyle d_1$}] (dU1) {};
    \node[core, right=0.75 of U1] (UCL) {};
    \node[core, below=0.5 of UCL, label={[anchor=north, label distance=-2]below:$\scriptstyle d_{m-1}$}] (dUCL) {};

    \node[core, right=0.75 of UCL] (UC) {};
    \node[core, below=0.5 of UC, label={[anchor=north, label distance=-2]below:$\scriptstyle d_m$}] (dUC) {};

    \node[core, right=0.75 of UC] (UCR) {};
    \node[core, below=0.5 of UCR, label={[anchor=north, label distance=-2]below:$\scriptstyle d_{m+1}$}] (dUCR) {};
    \node[core, right=0.75 of UCR] (UM) {};
    \node[core, below=0.5 of UM, label={[anchor=north, label distance=-2]below:$\scriptstyle d_M$}] (dUM) {};

    \draw[contraction] (U1) edge (dU1) -- ($(U1)+(0.2,0)$)
                        ($(UCL)-(0.2,0)$) -- (UCL) edge (dUCL) -- ($(UC)-(0.2,0)$)
                        ($(UC)+(0,-0.2)$) edge (dUC)
                        ($(UC)+(0.2,0)$) -- (UCR) edge (dUCR) -- ($(UCR)+(0.2,0)$)
                        (UM) edge (dUM) -- ($(UM)-(0.2,0)$);
    \draw[contractionDots] ($(U1)+(0.2,0)$) -- ($(UCL)-(0.2,0)$)
                            ($(UCR)+(0.2,0)$) -- ($(UM)-(0.2,0)$);

    \Orth{(U1)}{-45}
    \Orth{(UCL)}{-45}
    \Orth{(UCR)}{45}
    \Orth{(UM)}{45}
\end{tikzpicture} .
\end{equation}
The microstep that updates the $m$\textsuperscript{th} component tensor $\boldsymbol{V}_m$ of $\boldsymbol{V}$ can the be formalized as
\begin{equation} \label{eq:ALS_microstep}
    \underset{\boldsymbol{V}_m}{\text{minimize}}\ \norm{\boldsymbol{u} - \boldsymbol{M}\xhat{\boldsymbol{V}}_m \boldsymbol{V}_m}_2^2
\end{equation}

The ALS algorithm is not designed to restrict the variation function explicitly.
To show that this also does not happen implicitly, we define the linear subspace
\begin{equation}
    \mcal{V}_{\xhat{\boldsymbol{V}}_m} := \braces{\pars{\xhat{\boldsymbol{V}}_m\boldsymbol{V}_m, \boldsymbol{b}\otimes\cdots\otimes\boldsymbol{b}}_{\mathrm{Fro}} : \boldsymbol{V}_m\in\mbb{R}^{r_{m-1}\times d_m\times r_m}} .
\end{equation}
This is the space over which the microstep on the $m$\textsuperscript{th} component of the tensor $\boldsymbol{V}$ optimizes.
It is easy to see that: %

\begin{theorem} \label{thm:KUMicrostep}
    The bound $\mfrak{K}_{\mcal{V}_{\xhat{\boldsymbol{V}}_m}} \le \mfrak{K}_{\mcal{V}_{d}^{\otimes M}}$ is sharp.
    Moreover, if $b_0 \equiv 1$, then every microstep can increase the variation constant by a factor of up to $\mfrak{K}_{\mcal{V}_{d}}$.
\end{theorem}
\begin{proof}
    Let $\boldsymbol{B}\pars{y} := \boldsymbol{b}\pars{y_1}\otimes\cdots\otimes\boldsymbol{b}\pars{y_M}$ and recall that
    \begin{equation}
        \norm{\mfrak{K}_{\mcal{V}_{d}^{\otimes M}}}_\infty
        = \sup_{v\in\mcal{V}_{d}^{\otimes M}} \frac{\norm{v}_\infty^2}{\norm{v}^2}
        = \sup_{y\in Y} \sup_{\boldsymbol{V}\in\mbb{R}^{d_1\times\cdots\times d_M}} \frac{\boldsymbol{V}^\intercal\boldsymbol{B}\pars{y}\boldsymbol{B}\pars{y}^\intercal\boldsymbol{V}}{\norm{\boldsymbol{V}}_{\mathrm{Fro}}^2}
        = \sup_{y\in Y} \boldsymbol{B}\pars{y}^\intercal \boldsymbol{B}\pars{y} .
    \end{equation}
    Let $\braces{y_k}_{k\in\mbb{N}}$ be a maximizing sequence and define $\boldsymbol{V}^k := \boldsymbol{B}\pars{y_k}$ for every $k\in\mbb{N}$.
    Then 
    \begin{equation}
        \norm{\mfrak{K}_{\mcal{V}_{d}^{\otimes M}}}_\infty
        \ge \lim_{k\to\infty} \norm{\mfrak{K}_{\mcal{V}_{\xhat{\boldsymbol{V}}^k_m}}}_\infty
        \ge \lim_{k\to\infty} \boldsymbol{B}\pars{y_k}^{\intercal}\boldsymbol{B}\pars{y_k}
        = \norm{\mfrak{K}_{\mcal{V}_{d}^{\otimes M}}}_\infty .
    \end{equation}
    Moreover, if $b_0 \equiv 1$ and $\boldsymbol{V} = e_1^{\otimes M}$ then $\mfrak{K}_{\mcal{V}_{\xhat{\boldsymbol{V}}_m}} = \mfrak{K}_{\mcal{V}_{d}}$.
    Assume that the microstep results in a function with coefficient tensor $\boldsymbol{W} = e_1^{\otimes m-1} \otimes \boldsymbol{1}\otimes e_1^{\otimes M-m}$.
    Then $\mfrak{K}_{\mcal{V}_{\xhat{\boldsymbol{W}}_n}} = \mfrak{K}_{\mcal{V}_{d}}^2$ for every $n\ne m$.
\end{proof}

\subsection{A modified ALS algorithm}

The root problem in \cref{thm:KUMicrostep} is the microstep itself, which may result in local spaces $\mcal{V}_{\xhat{\boldsymbol{V}}_m}$ with ever increasing variation function, eventually approaching that of the ambient tensor space $\mcal{V}_{d}^{\otimes M}$.
The microstep is therefore the natural leverage point for a modification of the ALS algorithm.
To design an ALS microstep with bounded variation constant we thus modify the $m$\textsuperscript{th} microstep by restricting the admissible set from the linear space $\mcal{V}_{\xhat{\boldsymbol{V}}_m}$ to a reduced, \emph{nonlinear} set $\mcal{M}_{\xhat{\boldsymbol{V}}_m}$.
The resulting optimization problem reads
\begin{equation}
    \label{eq:micro_step}
    \underset{\boldsymbol{V}_m}{\text{minimize}}\ \norm{\boldsymbol{u} - \boldsymbol{M}\xhat{\boldsymbol{V}}_m\boldsymbol{V}_m}_2^2 \qquad \text{subject to} \qquad \boldsymbol{V}_m \in \mcal{M}_{\xhat{\boldsymbol{V}}_m}
\end{equation}
where, compared to \eqref{eq:ALS_microstep} only the linear space $\mcal{V}_{\xhat{\boldsymbol{V}}_m}$ has been replaced by the nonlinear set $\mcal{M}_{\xhat{\boldsymbol{V}}_m}$.
Inspired by \cite{Rauhut2016weighted_l1}, we choose $\mcal{M}_{\xhat{\boldsymbol{V}}_m}$ as a set of weighted sparsity
\begin{equation}
    \mcal{M}_{\xhat{\boldsymbol{V}}_m} := \braces{v\in\mcal{V}_{\xhat{\boldsymbol{V}}_m} \,:\, \norm{v}_{\omega,0}\le s},
\end{equation}
where $\norm{\bullet}_{\omega,0}$ is defined as in \cref{ex:sparse_regression}.
In \cite[Section~3.2]{EST20} it is shown that $\mfrak{K}_{\braces{u_{\mcal{M}_{\xhat{\boldsymbol{V}}_m}}} - \mcal{M}_{\xhat{\boldsymbol{V}}_m}} \le 2s$
when $\omega_j \ge \norm{B_{\xhat{\boldsymbol{V}}_m, j}}_\infty$ for all $j\in\bracs{r_{m-1}\times d_m\times r_m}$ where $B_{\xhat{V}_m} = \xhat{\boldsymbol{V}}_m^\intercal \boldsymbol{B}^{\otimes M}$ is an orthonormal basis for the local space $\mcal{V}_{\xhat{\boldsymbol{V}}_m}$.

\begin{remark}
    Note that the sparsity is only used to bound the variation function and is lost during the orthogonalization steps that are performed in a classical ALS implementation.\footnote{These orthogonalization steps are required to improve numerical stability.}
\end{remark}

A classical approach to handle the sparsity constraints in \eqref{eq:micro_step} is to promote the $\ell^0$-constraints via an $\ell^1$-regularization term.
The resulting problem then reads
\begin{equation} \label{eq:weighted_lasso}
    \underset{\boldsymbol{V}_m}{\text{minimize}}\ \norm{y - M\xhat{\boldsymbol{V}}_m \boldsymbol{V}_m}_2^2 + \lambda\norm{\omega\odot \boldsymbol{V}_m}_1 .
\end{equation}
The regularization parameter $\lambda$ controls the sparsity of $\boldsymbol{V}_m$ and is discussed at the end of this subsection.

To choose the weight sequence $\omega_i$ appropriately we have to compute the norms $\norm{B_{\xhat{\boldsymbol{V}}_m,j}}_{\infty}$ for all $j\in\bracs{r_{m-1}\times d_m\times r_m}$.
This is a difficult problem in general and has to be repeated in every microstep, since the local basis $B_{\xhat{\boldsymbol{V}}_m}$ depends on $\xhat{\boldsymbol{V}}_m$.
To obtain an estimate of $\norm{B_{\xhat{\boldsymbol{V}}_m,j}}_{\infty}$ in a numerically feasible fashion, we use the fact that every finite dimensional linear space is a \emph{reproducing kernel Hilbert space} (RKHS).
Since the norm of a RKHS $\mcal{H}$ satisfies the property that $\norm{\bullet}_\infty \le C \norm{\bullet}_{\mcal{H}}$ we can choose $\omega_i := C\norm{B_{\xhat{\boldsymbol{V}}_m,i}}_{\mcal{H}} \ge \norm{B_{\xhat{\boldsymbol{V}}_m,i}}_\infty$.
\begin{example}
    Let $\braces{B_j}_{j=1,\ldots,d}$ be an arbitrary basis and define $G = \operatorname{diag}\pars{\norm{B_1}_\infty^2, \ldots, \norm{B_d}_\infty^2}$.
    Using the triangle inequality and Jensen's inequality, we can estimate $\abs{v\pars{x}}^2 \le \pars{\sum_{j=1}^d \abs{\boldsymbol{v}_j}\norm{B_j}_\infty}^2 \le d {\boldsymbol v}^\intercal G {\boldsymbol v}$.
    A simple choice for an RKHS inner product is thus given by $\pars{u,v}_{\mcal{H}} := {\boldsymbol u}^\intercal G {\boldsymbol v}$.
\end{example}
\begin{example}
    The standard Sobolev space $H^s\pars{\mbb{R}^d}$, with arbitrary positive integers $d$ and $s > \frac{d}{2}$ is a RKHS with $C \le \frac{d+1}{2^{\pars{d+1}/2}\pi^{d/4}}$.
    For a proof of this claim we refer to~\cite{novak_2018_sobolev_rkhs}.
\end{example}

Recall that $\mcal{M}\subseteq\mcal{V}_{d}^{\otimes M}$, where the $d$-dimensional, uniform space $\mcal{V}_{d}$ is spanned by the basis $\braces{b_j}_{j=1,\ldots,d}$.
Given a RKHS inner product $\pars{\bullet,\bullet}_{\mcal{H}_{d}}$ for the univariate spaces $\mcal{H}_{d}=\mcal{V}_{d}$, we define the corresponding Gramian $g_{ij} = \pars{b_i, b_j}_{\mcal{H}_{d}}$.
This induces a RKHS  inner product on the global space $\mcal{H} = \mcal{H}_{d}^{\otimes M}$ and the corresponding Gramian is given by $G = g^{\otimes M}$.
The Gramian of the local model space $\mcal{V}_{\xhat{V}_m}$ is then given by
\begin{equation} \label{eq:local_gramian}
    H_{ij}
    = \pars{B_{\xhat{\boldsymbol{V}}_m,i}, B_{\xhat{\boldsymbol{V}}_m,j}}_{\mcal{H}}
    = \sum_{k,l\in\bracs{d}^M} \xhat{\boldsymbol{V}}_{m,ki} \xhat{\boldsymbol{V}}_{m,lj} \pars{B^{\otimes M}_k, B^{\otimes M}_l}_{\mcal{H}}
    = \pars{\xhat{\boldsymbol{V}}_{m}^\intercal G \xhat{\boldsymbol{V}}_{m}}_{ij}
\end{equation}
Due to the product structure of $G=g^{\otimes M}$, this quantity can be computed easily in the tensor train format via the contraction diagram
\begin{equation}
\begin{tikzpicture}[baseline=(current bounding box.center)]
    \node (H) {$H$};
    \node[right=0.2 of H] (eqH) {$=$};
    
    \node[core, right=1.0 of eqH, label={[anchor=east, label distance=5]left:$G$}] (H1) {};
    \node[core, right=0.75 of H1] (HCL) {};
    \node[core, right=0.75 of HCL] (HC) {};
    \node[core, right=0.75 of HC] (HCR) {};
    \node[core, right=0.75 of HCR] (HM) {};

    \node[core, above=0.5 of H1, label={[anchor=east, label distance=5]left:$\xhat{\boldsymbol{V}}_m$}] (U1) {};
    \node[core, above=0.5 of HCL] (UCL) {};
    \node[core, above=0.5 of HC] (UC) {};
    \node[core, above=0.5 of HCR] (UCR) {};
    \node[core, above=0.5 of HM] (UM) {};

    \node[core, below=0.5 of H1, label={[anchor=east, label distance=5]left:$\xhat{\boldsymbol{V}}_m$}] (V1) {};
    \node[core, below=0.5 of HCL] (VCL) {};
    \node[core, below=0.5 of HC] (VC) {};
    \node[core, below=0.5 of HCR] (VCR) {};
    \node[core, below=0.5 of HM] (VM) {};

    \draw[contraction] (U1) -- ($(U1)+(0.2,0)$)
                        ($(UCL)-(0.2,0)$) -- (UCL) -- ($(UC)-(0.2,0)$)
                        ($(UC)+(0.2,0)$) -- (UCR) -- ($(UCR)+(0.2,0)$)
                        (UM) -- ($(UM)-(0.2,0)$)
                        (U1) -- (V1)
                        (UCL) -- (VCL)
                        ($(UC)-(0,0.2)$) -- ($(VC)+(0,0.2)$)
                        (UCR) -- (VCR)
                        (UM) -- (VM)
                        (V1) -- ($(V1)+(0.2,0)$)
                        ($(VCL)-(0.2,0)$) -- (VCL) -- ($(VC)-(0.2,0)$)
                        ($(VC)+(0.2,0)$) -- (VCR) -- ($(VCR)+(0.2,0)$)
                        (VM) -- ($(VM)-(0.2,0)$)
                        ;
    \draw[contractionDots] ($(U1)+(0.2,0)$)  -- ($(UCL)-(0.2,0)$)
                            ($(UCR)+(0.2,0)$) -- ($(UM)-(0.2,0)$)
                            ($(V1)+(0.2,0)$)  -- ($(VCL)-(0.2,0)$)
                            ($(VCR)+(0.2,0)$) -- ($(VM)-(0.2,0)$);

    \Orth{(U1)}{-45}
    \Orth{(UCL)}{-45}
    \Orth{(UCR)}{45}
    \Orth{(UM)}{45}
    \Core{(H1)}
    \Core{(HCL)}
    \Core{(HC)}
    \Core{(HCR)}
    \Core{(HM)}
    \Orth{(V1)}{-135}
    \Orth{(VCL)}{-135}
    \Orth{(VCR)}{135}
    \Orth{(VM)}{135}
\end{tikzpicture}
\end{equation}
We can thus choose $\omega_i := C\sqrt{H_{ii}} = C\norm{B_{\xhat{\boldsymbol{V}}_m}}_{\mcal{H}} \ge \norm{B_{\xhat{\boldsymbol{V}}_m}}_\infty$.
Defining $D := \operatorname{diag}\pars{\operatorname{diag}\pars{H}}^{1/2}$ and substituting $U = D\boldsymbol{V}_m$ into \eqref{eq:weighted_lasso} we obtain the standard LASSO equation
\begin{equation} \label{eq:std_lasso}
    \underset{U}{\text{minimize}}\ \norm{y - M\xhat{\boldsymbol{V}}_mD^{-1}U}_2^2 + \lambda C\norm{U}_1 .
\end{equation}

For simplicity, we choose $\lambda$ by $10$-fold cross-validation.
This allows us to drop the factor $C$ and allows the algorithm to choose a different regularization parameter $\lambda$, i.e.\ a different sparsity level, for every component $\boldsymbol{V}_m$.

\subsection{Parametrization independent regularization}

Recall that the component $\boldsymbol{V}_m$ and the operator $\xhat{\boldsymbol{V}}_m$ are defined only up to orthogonal transformation since
\begin{equation}
    \xhat{\boldsymbol{V}}_m \boldsymbol{V}_m = \pars{\xhat{\boldsymbol{V}}_m Q} \pars{Q^\intercal \boldsymbol{V}_m} .
\end{equation}
where $Q = \pars{Q_{\mathrm{L}} \otimes \operatorname{Id}_{d_m} \otimes Q_{\mathrm{R}}}$ for any two orthogonal matrices $Q_{\mathrm{L}}\in \mbb{R}^{r_{m-1}\times r_{m-1}}$ and $Q_{\mathrm{R}}\in \mbb{R}^{r_{m}\times r_{m}}$.
This means that the regularization term in \eqref{eq:weighted_lasso} is not well-defined, since every orthogonal transformation $Q$ corresponds to a different basis $B_{\xhat{\boldsymbol{V}}_m Q}$.
This ambiguity can be resolved by selecting a specific orthonormal basis $\braces{\widetilde{B}_j}_{j=1,\ldots,D}$ with $D = r_{m-1} d_m r_m$.
We propose to do this iteratively, by defining
\begin{equation} \label{eq:optimal_spaces}
    \mcal{W}_{d} = \operatorname{span}\braces{\widetilde{B}_k}_{k=1}^d
    \qquad \widetilde{B}_{d+1} \in \argmin_{\substack{f\in\mcal{W}_{D}/\mcal{W}_d \\ \norm{f}=1}} \norm{f}_\infty
\end{equation}
where we use the convention, that $\inner{\emptyset} = \braces{0}$ and where $\mcal{W}_{D}/\mcal{W}_d$ denotes the orthogonal complement of $\mcal{W}_{d}$ in $\mcal{W}_D$.
Selecting the basis in this way ensures that $\mcal{W}_d$ has the variation function with the lowest $\norm{\bullet}_\infty$-norm under all subspaces $\widetilde{\mcal{W}}_d\subseteq \mcal{W}_{D}$ of dimension $d$.
The intuition for this is that, although we do not know $u_{\mcal{W}_{D}}$, we can assume that $\mfrak{K}_{u_{\mcal{W}_{D}}}$ is small and that $u_{\mcal{W}_{D}}$ can be approximated with high accuracy in the spaces $\mcal{W}_d$, even for low $d$.

To do this in a numerically feasible way, we replace every $\norm{\bullet}_\infty$ in \eqref{eq:optimal_spaces} by a $\norm{\bullet}_{\mcal{H}}$.
Using the spectral decomposition $H = QSQ^\intercal$, we can write $\widetilde{\boldsymbol{B}} = Q^\intercal \boldsymbol{B}_{\xhat{\boldsymbol{V}}_m} = \boldsymbol{B}_{\xhat{\boldsymbol{V}}_m Q}$ and obtain the diagonal weight matrix $D = \sqrt{S}$ for \eqref{eq:std_lasso}.

\begin{remark}
    Note that this basis is uniquely defined if the minimizer in \cref{eq:optimal_spaces} is unique and that it still satisfies the $L^2$-orthogonality condition that is required in~\cite[Section~3.2]{EST20}.
\end{remark}

We call the resulting algorithm \emph{restricted alternating least-squares} (RALS) since it modifies a standard ALS method by restricting the microsteps.
A listing of the complete algorithm, in pseudo-code, is provided in \cref{alg:rals}.
There it can be seen that the algorithm differs from a standard ALS only in two points.
The standard regeression in the microstep is replaced by a LASSO and an additional operator, namely the Gramian $H$, needs to be computed.
It is therefore straight-forward to implement.

\begin{algorithm}[t!]
\KwData{Data pairs $(x^{i},y^{i})\in\mathbb{R}^M\times \mathbb{R}$ for $i = 1,\ldots,n$, univariate basis functions $\braces{b_1,\ldots, b_d}$, univariate Gramians $G_m$ for $m=1,\ldots,M$.}
\KwResult{Coefficient tensor $\boldsymbol{V}$ of a function $v\in\mcal{M}$ that approximates the data.}
\BlankLine

Initialize the coefficient tensor $\boldsymbol{V}$\;
\While{not converged}{
    Right orthogonalize $\boldsymbol{V}$\;
    \For{$m=1,\ldots,M$}{
        Compute the optimal basis according to Equation~\eqref{eq:optimal_spaces}\;
        Compute the Gramian $H$ according to Equation~\eqref{eq:local_gramian}\;
        Compute $D := \operatorname{diag}\pars{\operatorname{diag}\pars{H}}^{1/2}$\;
        Update $\boldsymbol{V}_m$ by solving Equation~\eqref{eq:std_lasso} using cross-validation\;
        Left orthogonalize $\boldsymbol{V}_m$ and adapt the $m$\textsuperscript{th} rank\;
    }
}
\Return{ $\boldsymbol{V}$ }
\caption{Restricted Alternating Least-Squares (RALS)}
\label{alg:rals}
\end{algorithm}

The preceding algorithm provably satisfies the design principles.
But the necessity of an additional operator, the handling of potential floating point under- and overflows in its construction, and the numerical stability of the final orthogonalization procedure result in a computationally costly algorithm.
Taking a step back and reexamining \eqref{eq:std_lasso} reveals that this is not necessary.

Observe that $\widetilde{B}$ is $L^2$-orthonormal and $\mcal{H}$-orthogonal basis and that the substitution $U = D\boldsymbol{V}_m$ can be seen as a basis transform that produces a $L^2$-orthogonal and $\mcal{H}$-orthonormal basis.
\eqref{eq:std_lasso} then finds a sparse coefficient tensor in this basis.
A similar effect can be achieved by a transformation of the global basis $b^{\otimes M}$.
By $\mcal{H}$-orthonormalizing the global basis $b^{\otimes M}$ we obtain a $\mcal{H}$-orthonormal local basis $B_{\xhat{\boldsymbol{V}}_m}$.
This basis does not necessesarily constitute an $L^2$-orthogonal basis, but still is a Riesz-sequence for which Theorem~3.11 from~\cite{EST20} can be employed.
The resulting problem then reads
\begin{equation} \label{eq:unweighted_lasso}
    \underset{\boldsymbol{V}_m}{\text{minimize}}\ \norm{y - M\xhat{\boldsymbol{V}}_m\boldsymbol{V}_m}_2^2 + \lambda C\norm{\boldsymbol{V}_m}_1 .
\end{equation}
In this case we do not have to perform a resubstitution and directly obtain the solution.
The sparsity that is promoted by the algorithm in the component tensor $\boldsymbol{V}_m$ can then be interpreted as a gauge condition.

\begin{remark}
    Note that $\mcal{H}$-orthonormalizing the global basis $b^{\otimes M}$ can be done very efficiently by $\mcal{H}_{d}$-othonormalizing the univariate basis $b$.
\end{remark}
\begin{remark}
    \leavevmode\setlength{\parskip}{0px}
    For many choices of $\mcal{H}$, there exists a unique $L^2$-orthonormal and $\mcal{H}$-orthogonal basis.
    
    To see this, let $B$ be any $L^2$-orthonormal basis and define the Gramian $G_{ij} = \pars{B_i, B_j}_{\mcal{H}}$ as well as its spectral decomposition $G = QSQ^\intercal$.
    Then the basis $Q^\intercal B$ is $L^2$-orthonormal and $\mcal{H}$-orthogonal. 
    It is unique, since $Q$ is uniquely defined.
\end{remark}

It is easy to imagine that the quality of the resulting algorithm immensely depends on the choice of the RKHS $\mcal{H}_{d}$.
Since all norms in finite dimensional vector spaces are equivalent every space is a RKHS but the quality of \eqref{eq:weighted_lasso} depends on the tightness of the bound, i.e.\ on the size of $C$.
Moreover, for the second algorithm we have an additional requirement: the tightness of the Riesz sequence $B_{\xhat{V}_m}$, which depends on $\mcal{H}$ as well as $\xhat{V}_m$.
This means that it depends on the sought function $u$ itself.

The resulting algorithm is called \emph{Riesz-sequence restricted ALS} (R\textsuperscript{2}\!ALS) and is listed in \cref{alg:rals_heuristic}.
It is significantly easier to implement than \cref{alg:rals}, since it differs from a standard ALS merely by a preceding orthogonalization of the basis and by the restriction of the microstep.
The LASSO that is employed in both algorithms is a standard LASSO for which highly optimized implementations are available.

\begin{algorithm}[t!]
\KwData{Data pairs $(x^{i},y^{i})\in\mathbb{R}^M\times \mathbb{R}$ for $i = 1,\ldots,n$, univariate basis functions $\braces{b_1,\ldots,b_d}$, univariate Gramians $G_m$ for $m=1,\ldots,M$.}
\KwResult{Coefficient tensor $\boldsymbol{V}$ of a function $v\in\mcal{M}$ that approximates the data.}
\BlankLine

Initialize the coefficient tensor $\boldsymbol{V}$\;
\For{$m=1,\ldots,M$}{
    Orthonormalize the univariate basis w.r.t.\ $G_m$\;
}
\While{not converged}{
    Right orthogonalize $\boldsymbol{V}$\;
    \For{$m=1,\ldots,M$}{
        Update $\boldsymbol{V}_m$ by solving Equation~\eqref{eq:unweighted_lasso} using cross-validation\;
        Left orthogonalize $\boldsymbol{V}_m$ and adapt the $m$\textsuperscript{th} rank\;
    }
}
\Return{ $\boldsymbol{V}$ }
\caption{Riesz-sequence Restricted Alternating Least-Squares (R\textsuperscript{2}\!ALS)}
\label{alg:rals_heuristic}
\end{algorithm}

\subsection{Rank adaptivity and numerical stability}

Both, RALS and R\textsuperscript{2}\!ALS, allow for a straight-forward integration of rank-adaptivity.
The heuristic in \cref{alg:rals_heuristic} penalizes the $\ell^1$-norm of the core tensor $V_m$ which, by the following theorem, provides a \emph{tight} upper bound for the Schatten-$1$ norm.
\begin{theorem}[\citet{Kong2019nuclearBound}]
    Let $\norm{\bullet}_*$ denote the nuclear norm and $\norm{\bullet}_\sigma$ the spectral norm of a matrix.
    Moreover, let $\norm{\bullet}_1$ be the sum of moduli of its entries.
    Then
    $$
        \norm{T}_* \le \norm{T}_1
    $$
    where equality holds for diagonal matrices.
\end{theorem}
Since the Schatten-$1$ norm provides a convex surrogate for the rank it seems natural to use this regularization to adapt the rank of our iterates in \cref{alg:rals_heuristic}.

This argument can however not be applied directly to \cref{alg:rals} which uses a \emph{weighted} $\ell^1$ regularization term.
To investigate the influence of the weighting, we plot the singular values of multiple realizations of a random matrix $X$ and its weighted version $\omega\odot X$ in \cref{fig:spectrum}.
There we see that the spectrum of the weighted matrix decays faster.
Since a regularization by a Schatten-$1$ norm can be implemented by soft-thresholding of the singular values, this observation encourages us to use the weighted $\ell^1$-norm as a substitute for the nuclear norm in \cref{alg:rals} as well.

To implement rank-adaptivity practically, we use the approach of stable\slash unstable singular values that was pioneered in \cite{Grasedyck2019SALSA}.
This approach splits the sequence of singular values of a singular value decomposition into two groups.
The first group contains all singular values that exceed a certain threshold.
These are deemed stable and unlikely to change drastically in future iterations.
The second group contains all remaining singular values.
By fixing the size of the second group, dropping the smallest singular values or adding small, random singular values, if necessary, adaptivity is achieved.
Since the $\ell^1$-regularization term promotes a low rank, it promotes the stability of large singular values in favour of smaller ones.

\begin{remark}
    Note that the rank adaptivity is not required to satisfy \cref{imp:overestimation_blowup} but reduces the best approximation error.
\end{remark}

\begin{figure}[ht]
    \centering
    \includegraphics[width=\textwidth]{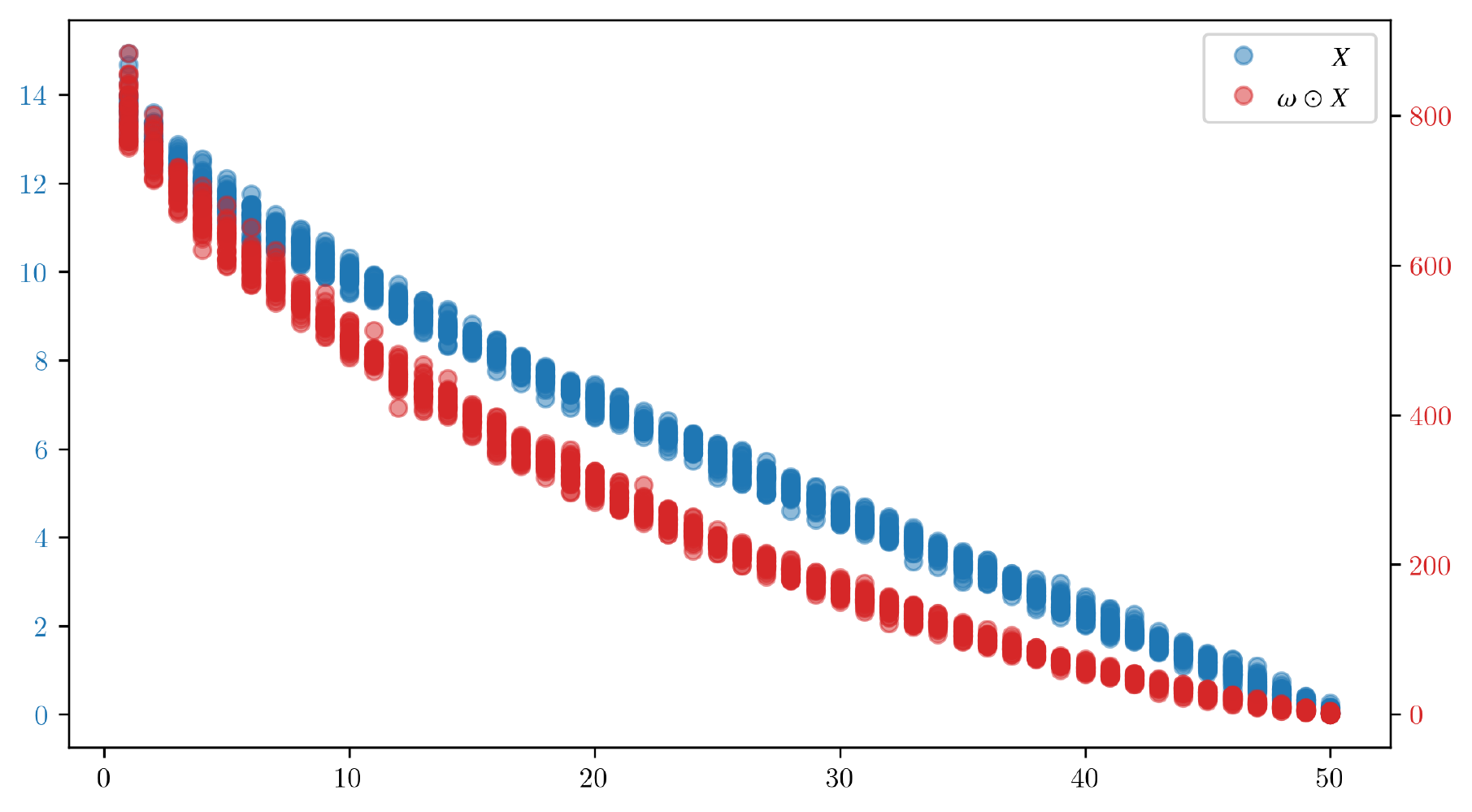}
    \caption{The spectra of matrices with standard normally distributed entries (blue) and their Hadamard product with a weighting matrix (red).
    The weight matrix $\omega$ is chosen as $\omega_{ij}=\sqrt{2i+1}\sqrt{2j+1}$ which corresponds to a basis $B$ of Legendre polynomials on $L^2\pars{\bracs{-1,1}^2}$.}
    \label{fig:spectrum}
\end{figure}

When considering a rank-adaptive augmentation of a given algorithm on tensor networks, the numerical stability of this algorithm is of particular interest.
The importance of the numerical stability, or the insensitivity to small perturbation, comes from the fact that the calibration of the rank requires a small perturbation.
Although the adaptation itself is numerically stable,
it is shown in~\cite{Grasedyck2019SALSA} that the result of an ALS microstep does not depend continuously on the tensor.
This implies that tiny changes in any iteration, such as those that are introduced during rank calibration, may have arbitrarily large influence on the further reconstruction.
To restore numerical stability they derive a regularization term that ensures stability.
It can be easily seen that the presented algorithm is not numerically stable as well.
This however does not result from our adaptation per se, but from the fact that we did not take the stability of the algorithm into account during its development.
We conjecture that, with a suitably modified microstep, our algorithm can be made stable as well.

\section{Experiments} \label{sec:experiments}

For the empirical validation of the R\textsuperscript{2}\!ALS algorithm, we consider a quanity of interest, derived from the stationary, random diffusion problem
\begin{equation}
    \label{eq:experiments:darcy}
    \begin{aligned}
        -\operatorname{div}( a(x,y)\operatorname{grad} w(x,y)) &= f(x), &\mbox{in }D,\\
        w(x,y) &= 0, &\mbox{on }\partial D
    \end{aligned}
\end{equation}
on the unit square $D=[0,1]^2$ and for $y\in Y$, where $Y$ depends on the specific parametrization of $a$.

For the sake of a clear presentation, the source term $f\in L^2(D)$ and the boundary conditions are assumed to be deterministic.
Pointwise solvability of~\eqref{eq:experiments:darcy} for almost all $y\in\mathbb{R}^M$ is guaranteed by a Lax--Milgram argument in~\cite{galvis-sarkis,SG11}.
Well-posedness of the variational parametric problem is way more intricate and we refer to~\cite{SG11} for a detailed discussion.

The solution $u$ often measures the concentration of some substance in the domain $D$ and one may be interested in the total amount of this substance in the entire domain
\begin{equation}
    U(y) := \int_{D} u(x,y) \,\mathrm{d}x .
\end{equation}
This quantity of interest was already considered in~\cite{bouchot_2015_CSPG} where a sparse approximation strategy was proposed.
The feasibility of low-rank approximation is ensured, since the coefficient tensor of $U$ can be sparsely approximated (cf.~\cite{bouchot_2015_CSPG} and~\cite{Hansen_2012}) and since sparse tensors can be represented efficiently in a low-rank format~\cite{li_2020_sparse_tt,bachmayr_2017_sparse_vs_low-rank}.
In the following we aim to approximate this quantity of interest for two different models of the diffusion coefficient $a$.

In the first numerical example we consider the affine-parametric diffusion equation with $Y = \bracs{-1,1}^{20}$ and
\begin{equation}
    a(x,y) := 1 + \frac{6}{\pi^2} \sum_{m=1}^{20} k^{-2} \sin(\hat{\varpi}_{m} x_1) \sin(\check{\varpi}_{m} x_2) y_m,
\end{equation}
where $\hat{\varpi}_m = \pi\lfloor\frac{m}{2}\rfloor$ and $\check{\varpi}_m = \pi\lceil\frac{m}{2}\rceil$.
We assume that $\rho = \frac{1}{2}\dx[y]$ and $w\equiv 1$ and search for the best approximation with respect to $\norm{\bullet}_{L^2\pars{Y,\rho}}$.
A comparison of R\textsuperscript{2}\!ALS to other state-of-the-art algorithms for the empirical best-approximation of $U$ is provided in \cref{tbl:darcy_uniform}.
It can be seen that R\textsuperscript{2}\!ALS clearly outperforms the other algorithms in the sample-sparse regime and that this edge vanishes when the number of samples increases.
This is to be expected, since the probability of the restricted isometry property increases with the number of samples.

The second example considers the log-normal diffusion equation with $Y = \mbb{R}^{20}$ and
\begin{equation}
    a(x,y) := \exp\pars*{\frac{1}{H_{20}} \sum_{m=1}^{20} m^{-1} \sin(\hat{\varpi}_{m} x_1) \sin(\check{\varpi}_{m} x_2) y_m},
\end{equation}
where again $\hat{\varpi}_m = \pi\lfloor\frac{m}{2}\rfloor$ and $\check{\varpi}_m = \pi\lceil\frac{m}{2}\rceil$ and $H_{20} := \sum_{m=1}^{20} \frac{1}{m}$ is the $20$\textsuperscript{th} harmonic number.
We assume that $\rho$ is a multivariate standard normal distribution and search for the best approximation with respect to $\norm{\bullet}_{L^2\pars{Y,\rho}}$.
Although the theory demands the use of an adapted sampling density, we observe that the choice $w\equiv 1$ seems to work well in practice.
The results of this experiment are provided in \cref{tbl:darcy_lognormal} and provide the same conclusion as for the previous example.

\begin{table}[ht]
    \centering
    \includegraphics[width=\textwidth]{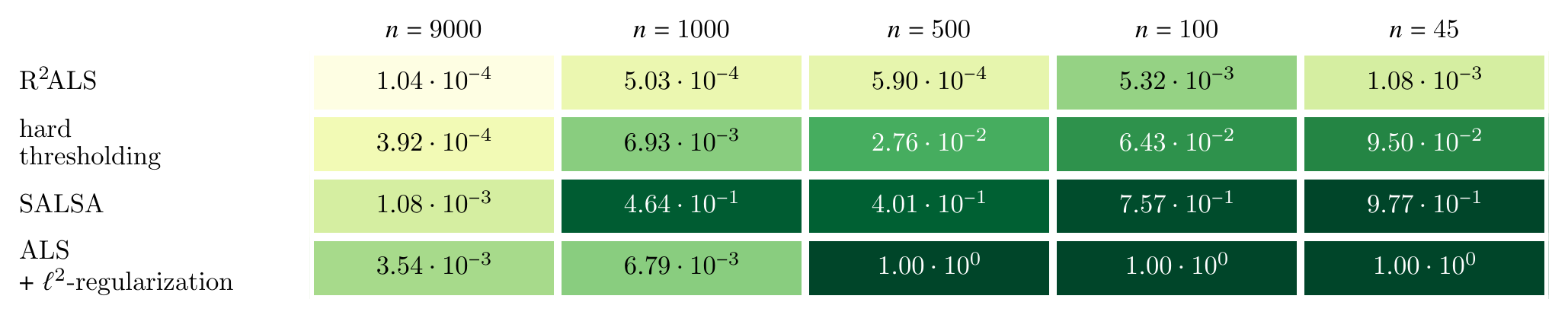}
    \caption{%
    Relative $L^2$-approximation error for the quantity of interest $U$ corresponding to the affine parametrization.
    R\textsuperscript{2}\!ALS is compared to bASD~\cite{eigel2018vmc0}, SALSA~\cite{Grasedyck2019SALSA} and to an $\ell^2$-regularized ALS, as used in~\cite{fackeldey2020approximativePolicyIteration,Kraemer2020thesis}.
    The regularization parameter for the $\ell^2$-regularized ALS is chosen by $10$-fold cross-validation.
    The relative error in the $L^2$-norm is estimated on a test set of $1000$ independent samples.
    All algorithms use the same samples to compute the empirical approximation (in each column) and the errors are always computed on the same test set.}
    \label{tbl:darcy_uniform}
\end{table}

\begin{table}[ht]
    \centering
    \includegraphics[width=\textwidth]{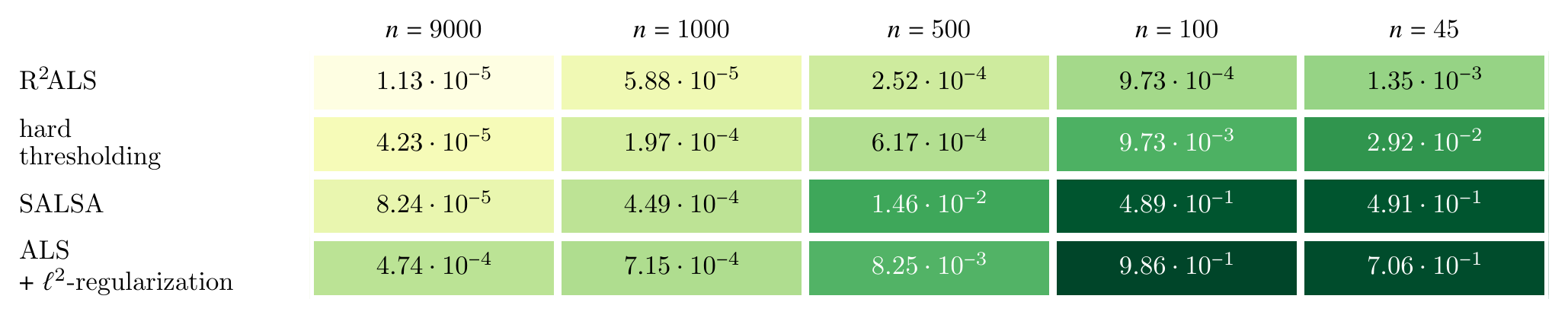}
    \caption{%
    Relative $L^2$-approximation error for the quantity of interest $U$ corresponding to the lognormal parametrization.
    R\textsuperscript{2}\!ALS is compared to bASD~\cite{eigel2018vmc0}, SALSA~\cite{Grasedyck2019SALSA} and to an $\ell^2$-regularized ALS, as used in~\cite{fackeldey2020approximativePolicyIteration,Kraemer2020thesis}.
    The regularization parameter for the $\ell^2$-regularized ALS is chosen by $10$-fold cross-validation.
    The relative error in the $L^2$-norm is estimated on a test set of $1000$ independent samples.
    All algorithms use the same samples to compute the empirical approximation (in each column) and the errors are always computed on the same test set.}
    \label{tbl:darcy_lognormal}
\end{table}

\section{Discussion}

This work extends the theory developed in~\cite{EST20}, where it was conjectured, that the worst-case sample complexity for any model class of tensor networks is of the same order of magnitude as for the ambient tensor space.
This hypothesis is confirmed and we argue, that current algorithms do not commonly display this behaviour, because they implicitly restrict the problem to a subclass on which fewer samples are required.
We investigate the validity of this heuristic argument for a wide range of model classes and discover, that it requires several assumptions, which may be hard to verify in practice.
In the context of matrix completion, we observe that one of these preconditions is related to the well-known incoherence condition.
To avoid this restriction, we propose to modify existing algorithms in such a way as to ensure a low sample complexity.
We demonstrate this by presenting two possible modifications of the alternating least squares algorithm for tensor approximation.
Both algorithms are rank-adaptive but not stable in the sense of~\cite{Grasedyck2019SALSA}, which can be attributed to the use of a cross-validated LASSO in the microsteps.
These microsteps result in a non-monotonic behaviour of the validation-set and training-set errors, which can indeed be observed during the minimization.
As of yet, there exists no proof of convergence for these algorithms.

We compare \cref{alg:rals_heuristic} to other state-of-the-art algorithms on two common benchmark problems from uncertainty quantification and observe that it drastically outperforms the others in the sample-scarce regime.
Although we expect \cref{alg:rals} to perform even better, we did not implement it due to numerical challenges and leave this as an interesting problem for a future work.
The experiments that are performed are inspired by those in~\cite{bouchot_2015_CSPG} and only consider the approximation of a quantity of interest.
However, we see no reason, why the same algorithm could not be extended to approximate the entire parametric solution.

This is not the first work that proposes the utilization of sparsity in the component tensors of a tensor network.
In~\cite{michel2021learning}, the authors consider the abstract setting of empirical risk minimization on bounded model classes of, potentially, sparse tensor networks.
They present a model selection strategy for the network topology and sparsity pattern and they derive error bounds.
In contrast to Theorems~\ref{thm:empirical_projection_error} and~\ref{thm:P_RIP}, the risk bound presented in~\cite{michel2021learning} works for arbitrary risk functions but does not guarantee an equivalence of errors.
It requires the model class to be bounded and it is not straight-forward to relate the sample complexity to a single quantity of the model class, like it is done in \cref{thm:P_RIP}.

In~\cite{Chevreuil2015sparseRank1} the authors propose an algorithm that computes the best approximation in the model class of sparse rank-$1$ tensors.
This algorithm is, conceptually, very similar to~\cref{alg:rals_heuristic}.
But since the authors delegate the choice of an appropriate basis, they can not exploit the advantages of weighted sparsity.
This means that, in the worst case, vastly more samples may be required than are actually necessary.
Contrary to our work, the authors in~\cite{Chevreuil2015sparseRank1} do not adapt the rank by adding small perturbations but by computing sparse rank-1 updated.
Although it is known, that such a sum of best approximations can lead to a suboptimal rank (cf.~\cite{Stegeman_2010} and the references therein), convergence is guaranteed by~\cite{ehrlacher_2021_greedy_tt}\todo{Check!}.
The success of multi-level methods in medical image reconstruction (cf.~\cite[Example~4.3]{EST20},~\cite{adcock_hansen_poon_roman_2017}, and~\cite{Wang20142693NS}) and parametric PDEs (cf.~\cite{ballani_2017_ML_tensor_approximation}) indicates that this may be an interesting application of our theory.
In contrast to Algorithms~\ref{alg:rals} and~\ref{alg:rals_heuristic}, greedy algorithms do not require explicit rank adaptation, which simplifies the implementation and alleviates any concerns about stability.
Moreover, since the representation of a rank-$1$ tensor is unique up to scaling factors of the coefficient tensors, both algorithms coincide.
This holds the promise to combine the conjectured performance benefits of \cref{alg:rals} with the numerical efficiency of \cref{alg:rals_heuristic} in this special case.
Finally, note that the $\norm{\bullet}_\infty$-norms of a rank-$1$ function can be estimated more easily, which may result in sharper bounds and in an improved convergence.

Block-sparse tensor networks are a well-known tool in the numerics of quantum mechanics~\cite{singh_2010_block_sparse_dmrg} and were recently introduced to the mathematics community by~\cite{bachmayr_2021_blocksparse}.
This theory is used in~\cite{goette_2021_blocksparse} to restrict the model class of tensor train networks to the subspace of homogeneous polynomials of fixed degree.
This guarantees a more moderate bound for the sample complexity. %
In contrast with this approach, where the sparsity structure has to be known in advance,
the two algorithms in \cref{sec:restricted_als} choose the sparsity implicitly and are agnostic to the chosen basis.
The downside of this is that the sparsity structure can not be interpreted as a restriction to a linear subspace of the ambient tensor space, which reduces the interpretability and increases the degrees of freedom.
It is also observed, that the introduction of an additional, virtual mode (cf.~\cite[Equation~(28)]{goette_2021_blocksparse}) is necessary to achieve block sparsity for arbitrary polynomials.
It would be interesting to investigate the effect of this construction on the theoretical bounds, developed in the present paper, and on the experimental performance of Algorithms~\ref{alg:rals} and~\ref{alg:rals_heuristic}.

During the completion of this article we came across the recent work~\cite{sancarlos2021pgdbased}, where a similar method is proposed and additional empirical evidence for its viability is provided.

\section*{Acknowledgements}

P.\ Trunschke acknowledges support by the Berlin International Graduate School in Model and Simulation based Research (BIMoS).

Our code made use of the following Python packages: numpy, scipy, and matplotlib \cite{numpy,scipy,matplotlib}.\todo{Check!}

\bibliographystyle{plainnat}
\bibliography{main}

\appendix
\section{Tensor Networks} \label{sec:tensor_networks}

This section introduces the concept of tensor networks and a graphical notation for the involved contractions related to tensor networks.
This notation drastically simplifies the expressions and makes the whole setup more approachable.
\todo{Also discuss why low-rank tensors are needed ($\dim(\mcal{V}) = d^M$...)}

\subsection{Tensors and indices}
\todo{$d$ and $k$ vs $M$ and $m$...}
\begin{definition}
    Let $d\in\mathbb{N}_{>0}$. Then $\boldsymbol{n} = (n_1,\cdots,n_d)\in\mathbb{N}^d$ is called a \emph{dimension tuple of order $d$} and $x\in\mathbb{R}^{n_1\times\cdots\times n_d} =: \mathbb{R}^{\boldsymbol{n}}$ is called a \emph{tensor of order $d$ and dimension $\boldsymbol{n}$}.
    Let $\mathbb{N}_{n} = \{1,\ldots,n\}$ then a tuple $(l_1,\ldots,l_d) \in \mathbb{N}_{n_1}\times\cdots\times\mathbb{N}_{n_d}=:\mathbb{N}_{\boldsymbol{n}}$ is called a \emph{multi-index} and the corresponding entry of $x$ is denoted by $x(l_1,\ldots,l_d)$.
    The positions $1,\ldots,d$ of the indices $l_1,\ldots,l_d$ in the expression $x(l_1,\ldots,l_d)$ are called \emph{modes of $x$}.
\end{definition} 

To define further operations on tensors it is often useful to associate each mode with a symbolic index.
\begin{definition}
	A \emph{symbolic index} $\mathrm{i}$ of dimension $n$ is a placeholder for an arbitrary but fixed natural number between $1$ and $n$.
	For a dimension tuple $\boldsymbol{n}$ of order $d$ and a tensor $x\in\mathbb{R}^{\boldsymbol{n}}$ we may write $x(\mathrm{i}_1,\ldots,\mathrm{i}_d)$ and tacitly assume that $\mathrm{i}_k$ are indices of dimension $n_k$ for each $k=1,\ldots,d$.
    When standing for itself this notation means $x(\mathrm{i}_1,\ldots,\mathrm{i}_d) = x\in\mathbb{R}^{\boldsymbol{n}}$ and may be used to \emph{slice} the tensor
    \begin{equation}
        x(\mathrm{i}_1,l_2,\ldots,l_d) \in \mathbb{R}^{n_1}
    \end{equation}
    where $l_k\in\mathbb{N}_{n_k}$ are fixed indices for all $k=2,\ldots,d$.
    For any dimension tuple $\boldsymbol{n}$ of order $d$ we define the symbolic multi-index $\mathrm{i}^{\boldsymbol{n}} = (\mathrm{i}_1, \ldots, \mathrm{i}_d)$ of dimension $\boldsymbol{n}$ where $\mathrm{i}_k$ is a symbolic index of dimension $n_k$ for all $k=1,\ldots,d$.
\end{definition}

\begin{remark}
	We use roman font letters (with appropriate subscripts) for symbolic indices while reserving standard letters for ordinary indices.
\end{remark}

\begin{example}
	Let $x$ be an order $2$ tensor with mode dimensions $n_1$ and $n_2$, i.e.\ an $n_1$-by-$n_2$ matrix.
	Then $x(l_1,\mathrm{j})$ denotes the $l_1$-th row of $x$ and $x(\mathrm{i},l_2)$ denotes the $l_2$-th column of $x$.
\end{example}

Inspired by Einstein notation we use the concept of symbolic indices to define different operations on tensors.
\begin{definition}\label{def:indexproduct}
	Let $i_1$ and $i_2$ be (symbolic) indices of dimension $n_1$ and $n_2$, respectively and let $\varphi$ be a bijection
	\begin{equation}
	     \varphi : \mathbb{N}_{n_1}\times \mathbb{N}_{n_2} \rightarrow \mathbb{N}_{n_1n_2}.
	\end{equation}
	We then define \textit{the product of indices} with respect to $\varphi$ as $\mathrm{j}=\varphi(\mathrm{i}_1,\mathrm{i}_2)$ where $\mathrm{j}$ is a (symbolic) index of dimension $n_1n_2$.
	In most cases the choice of bijection is not important and we will write $\mathrm{i}_1\cdot \mathrm{i}_2 := \varphi(\mathrm{i}_1,\mathrm{i}_2)$ for an arbitrary but fixed bijection $\varphi$.
	For a tensor $x$ of dimension $(n_1,n_2)$ the expression
	\begin{equation}
	    y(\mathrm{i}_1\cdot \mathrm{i}_2) = x(\mathrm{i}_1, \mathrm{i}_2)
	\end{equation}
	defines the tensor $y$ of dimension $n_1n_2$ while the expression
	\begin{equation}
	    x(\mathrm{i}_1, \mathrm{i}_2) =  y(\mathrm{i}_1\cdot \mathrm{i}_2)
	\end{equation}
	defines $x\in\mathbb{R}^{n_1\times n_2}$ from $y\in\mathbb{R}^{n_1n_2}$.
\end{definition}

\begin{definition}\label{def:productlike}
    Consider the tensors $x\in\mathbb{R}^{\boldsymbol{n}_1\times a\times\boldsymbol{n}_2}$ and $y\in\mathbb{R}^{\boldsymbol{n}_3\times b\times\boldsymbol{n}_4}$.
    Then the expression
    \begin{equation} \label{eq:outer}
        z(\mathrm{i}^{\boldsymbol{n_1}},\mathrm{i}^{\boldsymbol{n_2}},\mathrm{j}_1,\mathrm{j}_2,\mathrm{i}^{\boldsymbol{n_3}},\mathrm{i}^{\boldsymbol{n_4}}) = x(\mathrm{i}^{\boldsymbol{n_1}},\mathrm{j}_1,\mathrm{i}^{\boldsymbol{n_2}})\cdot y(\mathrm{i}^{\boldsymbol{n_3}},\mathrm{j}_2,\mathrm{i}^{\boldsymbol{n_4}})
    \end{equation}
    defines the tensor $z\in\mathbb{R}^{\boldsymbol{n}_1\times \boldsymbol{n}_2\times a\times b\times\boldsymbol{n}_3\times\boldsymbol{n}_4}$ in the obvious way.
    Similarly, for $a=b$ the expression 
    \begin{equation} \label{eq:hadamard}
        z(\mathrm{i}^{\boldsymbol{n_1}},\mathrm{i}^{\boldsymbol{n_2}},\mathrm{j},\mathrm{i}^{\boldsymbol{n_3}},\mathrm{i}^{\boldsymbol{n_4}}) = x(\mathrm{i}^{\boldsymbol{n_1}},\mathrm{j},\mathrm{i}^{\boldsymbol{n_2}})\cdot y(\mathrm{i}^{\boldsymbol{n_3}},\mathrm{j},\mathrm{i}^{\boldsymbol{n_4}})
    \end{equation}
    defines the tensor $z\in\mathbb{R}^{\boldsymbol{n}_1\times \boldsymbol{n}_2\times a\times\boldsymbol{n}_3\times\boldsymbol{n}_4}$.
    Finally, also for $a=b$ the expression
    \begin{equation} \label{eq:inner}
        z(\mathrm{i}^{\boldsymbol{n_1}},\mathrm{i}^{\boldsymbol{n_2}},\mathrm{i}^{\boldsymbol{n_3}},\mathrm{i}^{\boldsymbol{n_4}}) = x(\mathrm{i}^{\boldsymbol{n_1}},\mathrm{j},\mathrm{i}^{\boldsymbol{n_2}})\cdot y(\mathrm{i}^{\boldsymbol{n_3}},\mathrm{j},\mathrm{i}^{\boldsymbol{n_4}})
    \end{equation}
    defines the tensor $z\in\mathbb{R}^{\boldsymbol{n}_1\times \boldsymbol{n}_2\times \boldsymbol{n}_3\times\boldsymbol{n}_4}$ as
    \begin{equation}
        z(\mathrm{i}^{\boldsymbol{n_1}},\mathrm{i}^{\boldsymbol{n_2}},\mathrm{i}^{\boldsymbol{n_3}},\mathrm{i}^{\boldsymbol{n_4}}) = \sum_{k=1}^a x(\mathrm{i}^{\boldsymbol{n_1}},k,\mathrm{i}^{\boldsymbol{n_2}})\cdot y(\mathrm{i}^{\boldsymbol{n_3}},k,\mathrm{i}^{\boldsymbol{n_4}}) .
    \end{equation}
\end{definition}
We choose this description mainly because of its simplicity and how it relates to the implementation of these operations in the numeric libraries \texttt{numpy}~\cite{numpy} and \texttt{xerus}~\cite{xerus}.

\subsection{Graphical notation and tensor networks}
This section will introduce the concept of \textit{tensor networks}~\cite{espig_2011_tensor_networks} and a graphical notation for certain operations which will simplify working with these structures.
To this end we reformulate the operations introduced in the last section in terms of nodes, edges and half edges.
\begin{definition}
    For a dimension tuple $\boldsymbol{n}$ of order $d$ and a tensor $x\in\mathbb{R}^{\boldsymbol{n}}$ the \textit{graphical representation} of $x$ is given by
    \begin{center}
    	\begin{tikzpicture}
    	\draw[black,fill=black] (0,0) circle (0.5ex);
    	\node[anchor=south east] at (0,0) {$x$};
    	\draw(-1,0)--(1,0);
    	\draw(0,1)--(0,-1);
    	\node[anchor=south] at (-1,0) {$\mathrm{i}_1$};
    	\node[anchor=west] at (0,1) {$\mathrm{i}_2$};	\node[anchor=north] at (1,0) {$\mathrm{i}_3$};
    	\node at (0.5,-0.5) {\reflectbox{$\ddots$}};
    	\node[anchor=east] at (0,-1) {$\mathrm{i}_d$};
    	\end{tikzpicture}
    \end{center}
    where the node represents the tensor and the half edges represent the $d$ different modes of the tensor illustrated by the symbolic indices $\mathrm{i}_1,\ldots,\mathrm{i}_d$.
\end{definition}

\begin{example}
    The presented graphical representation, allows us to write scalars, vectors, matrices and tensors of order $5$ in an easily understandable fashion:
    \begin{center}
    \begin{tabular}{ c c c c }
        \begin{tikzpicture}
            \node[core, label={[anchor=north, label distance=6]below:$s\in\mbb{R}$}] (s) at (-1.5,0) {};
            \Core{(s)}
            \node [above=1.25, align=flush center,text width=6em] at (s) {scalar};
        \end{tikzpicture}
        &
        \begin{tikzpicture}
            \node (x) at (0,0.75) {$d_1$};
            \node[core, label={[anchor=north, label distance=3]below:$v\in\mbb{R}^{d_1}$}] (v) at (0,0) {};
            \draw[contraction] (x) -- (v);
            \Core{(v)}
            \node [above=1.25, align=flush center,text width=6em] at (v) {vector};
        \end{tikzpicture}
        &
        \begin{tikzpicture}
            \node (x1) at (1.25,0.75) {$d_1$};
            \node (x2) at (2.75,0.75) {$d_2$};
            \node[core, label={[anchor=north, label distance=3]below:$M\in\mbb{R}^{d_1\times d_2}$}] (M)  at (2,0) {};
            \draw[contraction] (x1) -- (M) -- (x2);
            \Core{(M)}
            \node [above=1.25, align=flush center,text width=6em] at (M) {matrix};
        \end{tikzpicture}
        &
        \begin{tikzpicture}
            \node (x1) at (4.25,0)      {$d_1$};
            \node (x2) at (4.375,0.625) {$d_2$};
            \node (x3) at (5,0.75)      {$d_3$};
            \node (x4) at (5.625,0.625) {$d_4$};
            \node (x5) at (5.75,0)      {$d_5$};
            \node[core, label={[anchor=north, label distance=3]below:$T\in\mbb{R}^{d_1\times\cdots\times d_5}$}] (T) at (5,0) {};
            \draw[contraction] (x1) -- (T);
            \draw[contraction] (x2) -- (T);
            \draw[contraction] (x3) -- (T);
            \draw[contraction] (x4) -- (T);
            \draw[contraction] (x5) -- (T);
            \Core{(T)}
            \node [above=1.25, align=flush center,text width=6em] at (T) {tensor};
        \end{tikzpicture}
    \end{tabular}
    \end{center}
\end{example}

With this definition we can write the reshapings of Defintion~\ref{def:indexproduct} simply as
\begin{center}
	\begin{tikzpicture}
	\node[anchor=east] at (-1,0) {$x(\mathrm{i}_1, \mathrm{i}_2\cdot \mathrm{i}_3\cdots \mathrm{i}_d)\quad=\quad$};
	\draw[black,fill=black] (0,0) circle (0.5ex);
	\node[anchor=south, inner sep=5] at (0,0) {$x$};
	\draw(-1,0)--(1,0);
	\node[anchor=south] at (-1,0) {$\mathrm{i}_1$};
	\node[anchor=south west] at (1,0) {$\hspace{-0.5em}\mathrm{i}_2\cdot \mathrm{i}_3\cdots \mathrm{i}_d$};
	\end{tikzpicture}
\end{center}
and also simplify the binary operations of Definition~\ref{def:productlike}.
\begin{definition}
    Let $x\in\mathbb{R}^{\boldsymbol{n}_1\times a\times \boldsymbol{n}_2}$ and $y\in\mathbb{R}^{\boldsymbol{n}_3\times b\times \boldsymbol{n}_4}$ be two tensors.
    Then Operation~\eqref{eq:outer} is represented by
    \begin{center}
    	\begin{tikzpicture}
    	\draw[black,fill=black] (0,0) circle (0.5ex);
    	\node[anchor=south east] at (0,0) {$x$};
    	\draw(-1,0)--(0,0);
    	\draw(0,1)--(0,-1);
    	\node[anchor=south] at (-1,0) {$\mathrm{i}$};
    	\node[anchor=south] at (0,1) {{$\mathrm{i}^{\boldsymbol{n_1}}$}};	
    	\node[anchor=north] at (0,-1) {{$\mathrm{i}^{\boldsymbol{n_2}}$}};
    	
    	\draw[black,fill=black] (2,0) circle (0.5ex);
    	\node[anchor=south west] at (2,0) {$y$};
    	\draw(2,0)--(3,0);
    	\draw(2,1)--(2,-1);
    	\node[anchor=south] at (3,0) {$\mathrm{j}$};
    	\node[anchor=south] at (2,1) {$\mathrm{i}^{\boldsymbol{n_3}}$};	
    	\node[anchor=north] at (2,-1) {$\mathrm{i}^{\boldsymbol{n_4}}$};
    	
    	\node at (4,0) {$=$};
    	
    	\draw[black,fill=black] (6,0) circle (0.5ex);
    	\node[anchor=south east] at (6,0) {$z$};
    	\draw(5,0)--(7,0);
    	\draw(6,1)--(6,-1);
    	\node[anchor=south] at (5,0) {$\mathrm{i}$};
    	\node[anchor=south] at (7,0) {$\mathrm{j}$};
    	\node[anchor=south] at (6,1) {{$\mathrm{i}^{\boldsymbol{n_1}}\cdot \mathrm{i}^{\boldsymbol{n_3}}$}};
    	\node[anchor=north] at (6,-1) {{$\mathrm{i}^{\boldsymbol{n_2}}\cdot \mathrm{i}^{\boldsymbol{n_4}}$}};
    	
    	\node[anchor=north] at (8,0) {.};
    	
    	\end{tikzpicture}
    \end{center}
    and defines $z\in\mathbb{R}^{\cdots\times a \times b \times\cdots}$.
    For $a=b$ Operation~\eqref{eq:hadamard} is represented by 
    \begin{center}
    	\begin{tikzpicture}
    	\draw[black,fill=black] (0,0) circle (0.5ex);
    	\node[anchor=south east] at (0,0) {$x$};
    	\draw(1,0.5)--(0,0);
    	\draw(0,1)--(0,-1);
    
        \node[anchor=south] at (0,1) {{$\mathrm{i}^{\boldsymbol{n_1}}$}};	
    	\node[anchor=north] at (0,-1) {{$\mathrm{i}^{\boldsymbol{n_2}}$}};
    	\node at (0.5,0) {$\mathrm{i}$};
    	\draw[black,fill=black] (2,0) circle (0.5ex);
    	\node[anchor=south west] at (2,0) {$y$};
    	\draw(2,0)--(1,0.5);
    	\draw(2,1)--(2,-1);
    	\node[anchor=south] at (2,1) {$\mathrm{i}^{\boldsymbol{n_3}}$};	
    	\node[anchor=north] at (2,-1) {$\mathrm{i}^{\boldsymbol{n_4}}$};
    	
    	\node at (1.5,0) {$\mathrm{i}$};
    	\draw(1,0.5)--(1,1);
    	
    	\node[anchor=south east] at (1,0.5) {$\mathrm{i}$};
    	
    	\node at (4,0) {$=$};
    	
    	\draw[black,fill=black] (6,0) circle (0.5ex);
    	\node[anchor=south east] at (6,0) {$z$};
    	\draw(5,0)--(6,0);
    	\draw(6,1)--(6,-1);
    	\node[anchor=south] at (5,0) {$\mathrm{i}$};
    
    	\node[anchor=south] at (6,1) {{$\mathrm{i}^{\boldsymbol{n_1}}\cdot \mathrm{i}^{\boldsymbol{n_3}}$}};
    	\node[anchor=north] at (6,-1) {{$\mathrm{i}^{\boldsymbol{n_2}}\cdot \mathrm{i}^{\boldsymbol{n_4}}$}};

    	\node[anchor=north] at (8,0) {.};
    	
    	\end{tikzpicture}
    \end{center}		
    and defines $z\in\mathbb{R}^{\cdots\times a  \times \cdots}$ and Operation~\eqref{eq:inner} defines $z\in\mathbb{R}^{\cdots \times \cdots}$ by
    \begin{center}
    	\begin{tikzpicture}
    	\draw[black,fill=black] (0,0) circle (0.5ex);
    	\node[anchor=south east] at (0,0) {$x$};
    	\draw(1,0)--(0,0);
    	\draw(0,1)--(0,-1);
    	
    	\node[anchor=south] at (0,1) {{$\mathrm{i}^{\boldsymbol{n_1}}$}};	
    	\node[anchor=north] at (0,-1) {{$\mathrm{i}^{\boldsymbol{n_2}}$}};
    	\node[anchor=north] at (0.5,0) {$\mathrm{i}$};
    	\draw[black,fill=black] (2,0) circle (0.5ex);
    	\node[anchor=south west] at (2,0) {$y$};
    	\draw(2,0)--(1,0);
    	\draw(2,1)--(2,-1);
    	\node[anchor=south] at (2,1) {$\mathrm{i}^{\boldsymbol{n_3}}$};	
    	\node[anchor=north] at (2,-1) {$\mathrm{i}^{\boldsymbol{n_4}}$};
    	
    	\node[anchor=north] at (1.5,0) {$\mathrm{i}$};
    	
    	\node at (4,0) {$=$};
    	
    	\draw[black,fill=black] (6,0) circle (0.5ex);
    	\node[anchor=south east] at (6,0) {$z$};
    	\draw(6,1)--(6,-1);
    	
    	\node[anchor=south] at (6,1) {{$\mathrm{i}^{\boldsymbol{n_1}}\cdot \mathrm{i}^{\boldsymbol{n_3}}$}};
    	\node[anchor=north] at (6,-1) {{$\mathrm{i}^{\boldsymbol{n_2}}\cdot \mathrm{i}^{\boldsymbol{n_4}}$}};

    	\node[anchor=north] at (8,0) {.};
    	
    	\end{tikzpicture}
    \end{center}
\end{definition}
With these definitions we can compose entire networks of multiple tensors which are called tensor networks.

\subsection{The Tensor Train Format}\label{subsec:tensortrains}

A prominent example of a tensor network is the \textit{tensor train (TT)}~\cite{oseledets_2011_tensor_trains,holtz_alternating_2012}, which is the main tensor network used throughout this work.
This network is discussed in the following subsection.
\begin{definition}
    Let $\boldsymbol{n}$ be an dimensional tuple of order-$d$.
    The TT format decomposes an order $d$ tensor $x\in\mathbb{R}^{\boldsymbol{n}}$ into $d$ \emph{component tensors} $x_k\in\mathbb{R}^{r_{k-1}\times n_k\times r_k}$ for $k=1,\ldots,d$ with $r_0 = r_d = 1$.
    This can be written in tensor network formula notation as 
    \begin{equation}
        x(\mathrm{i}_1,\cdots,\mathrm{i}_d) = x_1(\mathrm{i}_1,\mathrm{\mathrm{j}}_1)\cdot x_2(\mathrm{\mathrm{j}}_1,\mathrm{i}_2,\mathrm{\mathrm{j}}_2)\cdots x_d(\mathrm{\mathrm{j}}_{d-1},\mathrm{i}_d).
    \end{equation}
    The tuple $(r_1,\ldots,r_{d-1})$ is called the \emph{representation rank} of this representation.
\end{definition}

In graphical notation it looks like this
\begin{center}
	\begin{tikzpicture}
	\draw[black,fill=black] (0,0) circle (0.5ex);
	\node[anchor=south east] at (0,0) {$x$};
	\draw(-1,0)--(1,0);
	\draw(0,1)--(0,-1);
	\node[anchor=south] at (-1,0) {$\mathrm{i}_1$};
	\node[anchor=west] at (0,1) {$\mathrm{i}_2$};	\node[anchor=north] at (1,0) {$\mathrm{i}_3$};
	\node at (0.5,-0.5) {\reflectbox{$\ddots$}};
	\node[anchor=east] at (0,-1) {$\mathrm{i}_d$};
	
	\node[anchor=east] at (2,0) {$=$};

	\draw[black,fill=black] (3,0) circle (0.5ex);	
	\draw[black,fill=black] (4,0) circle (0.5ex);	
	\draw[black,fill=black] (5,0) circle (0.5ex);
	\node at (6,0) {\reflectbox{$\cdots$}};	
	\draw[black,fill=black] (7,0) circle (0.5ex);
	
	\draw(3,0)--(5.5,0);
	\draw(6.5,0)--(7,0);
	\draw(3,0)--(3,-0.75);
	\draw(4,0)--(4,-0.75);
	\draw(5,0)--(5,-0.75);
	\draw(7,0)--(7,-0.75);	
	
	\node[anchor=south] at (3,0) {$x_1$};
	\node[anchor=south] at (4,0) {$x_2$};	\node[anchor=south] at (5,0) {$x_3$};
	\node[anchor=south] at (7,0) {$x_d$};
	\node[anchor=north] at (3,-0.75) {$\mathrm{i}_1$};
	\node[anchor=north] at (4,-0.75) {$\mathrm{i}_2$};	\node[anchor=north] at (5,-0.75) {$\mathrm{i}_3$};
	\node[anchor=north] at (7,-0.75) {$\mathrm{i}_d$};
	\node[anchor=north] at (3.5,0) {$\mathrm{j}_1$};
	\node[anchor=north] at (4.5,0) {$\mathrm{j}_2$};	\node[anchor=north] at (5.5,0) {$\mathrm{j}_3$};
	\node[anchor=north] at (6.5,0) {$\mathrm{j}_{d-1}$};
	\end{tikzpicture}
\end{center}

\begin{remark}\label{rmk:TT_representation}
    Note that this representation is not unique.
    For any pair of matrices $(A,B)$ that satisfies $AB=\operatorname{Id}$ we can replace $x_k$ by $x_k(\mathrm{i}_1,\mathrm{i}_2,\mathrm{j})\cdot A(\mathrm{j},\mathrm{i}_3)$ and $x_{k+1}$ by $B(\mathrm{i}_1,\mathrm{j})\cdot x(\mathrm{j},\mathrm{i}_2,\mathrm{i}_3)$ without changing the tensor $x$.
\end{remark}

The representation rank of $x$ is therefore dependent on the specific representation of $x$ as a TT, hence the name.
Analogous to the concept of matrix rank we can define a minimal necessary rank that is required to represent a tensor $x$ in the TT format.
\begin{definition} 
	The \textit{tensor train rank} of a tensor $x\in\mathbb{R}^{\boldsymbol{n}}$ with tensor train components $x_1\in\mathbb{R}^{n_1\times r_1}$, $x_k\in\mathbb{R}^{r_{k-1}\times n_k\times r_k}$ for $k=2,\ldots,d-1$ and $x_d\in\mathbb{R}^{r_{d-1}\times n_d}$ is the set  
	\[
    	 \text{TT-rank}(x) = (r_1,\cdots,r_d)
	\]
	of minimal $r_k$'s such that the $x_k$ compose $x$.
\end{definition}
In \cite[Theorem 1a]{holtz_2012_manifolds} it is shown that the TT-rank can be computed by simple matrix operations.
Namely, $r_k$ can be computed by joining the first $k$ indices and the remaining $d-k$ indices and computing the rank of the resulting matrix.

At last, we need to introduce the concept of left and right orthogonality for the tensor train format. 
\begin{definition}
    Let $x\in\mathbb{R}^{\boldsymbol{m}\times n}$ be a tensor of order $d+1$.
    We call $x$ \emph{left orthogonal} if
    \[
        x(\mathrm{i}^{\boldsymbol{m}},\mathrm{j}_1) \cdot x(\mathrm{i}^{\boldsymbol{m}},\mathrm{j}_2)
        = \operatorname{Id}(\mathrm{j}_1,\mathrm{j}_2) .
    \]
    Similarly, we call a tensor $x\in\mathbb{R}^{m\times\boldsymbol{n}}$ of order $d+1$ \emph{right orthogonal} if 
    \[
        x(\mathrm{i}_1, \mathrm{j}^{\boldsymbol{n}}) \cdot x(\mathrm{i}_2, \mathrm{j}^{\boldsymbol{n}})
        = \operatorname{Id}(\mathrm{i}_1,\mathrm{i}_2) .
    \]
    A tensor train is \emph{left orthogonal} if all component tensors $x_1,\ldots,x_{d-1}$ are left orthogonal.
    It is \emph{right orthogonal} if all component tensors $x_2,\ldots,x_d$ are right orthogonal.
\end{definition}

\begin{lemma}[\cite{oseledets_2011_tensor_trains}]
	For every tensor $x\in\mathbb{R}^{\boldsymbol{n}}$ of order $d$ we can find left and right orthogonal decompositions. 
\end{lemma}

\section{Proof of Theorem~\ref{thm:K_properties}} \label{proof:thm:K_properties}

\begin{enumerate}
    \item Follows directly from the definition.
    
    \item To see that $\mfrak{K}_{A} = \mfrak{K}_{\cl\pars{A}}$ let $a\in\cl\pars{A}\setminus\braces{0}$.
    Then there exists a sequence $\braces{a_k} \in A\setminus\braces{0}$ such that $a_k \to a$.
    Due to the continuity of $a \mapsto a\pars{y}^2/\norm{\bullet}^2$ on $A\setminus\braces{0}$ it follows that
    \begin{equation}
        \mfrak{K}_{\braces{a}}\pars{y}
        = \frac{\abs{a\pars{y}}^2}{\norm{a}^2}
        = \lim_{k\to\infty} \frac{\abs{a_k\pars{y}}^2}{\norm{a_k}^2}
        = \lim_{k\to\infty} \mfrak{K}_{\braces{a_k}}\pars{y}
    \end{equation}
    And since $\mfrak{K}_{\braces{a_k}} \le \mfrak{K}_{A}$ for all $k=1,\ldots,\infty$ and we can conclude $\mfrak{K}_{\braces{a}} \le \mfrak{K}_{A}$.
    The assertion follows with \ref{thm:K_properties:union} since $\mfrak{K}_{A} \le \mfrak{K}_{\cl\pars{A}} = \sup_{a\in\cl\pars{A}}\mfrak{K}_{\braces{a}} \le \mfrak{K}_{A}$.

    \item[3.-5.] In all three case we can write $\mfrak{K}_{\bullet} = \operatorname{sqr}\circ \sup\circ \operatorname{abs}\circ\, U$ with
    \begin{align}
        \operatorname{sqr} &: \mcal{V}_{w,\infty} \to \mcal{V}_{w^2,\infty},
        & \operatorname{sqr}\pars{v}\pars{y} &:= v\pars{y}^2, \label{eq:sqr} \\
        \operatorname{abs} &: \mcal{V}_{w,\infty} \to \mcal{V}_{w,\infty},
        & \operatorname{abs}\pars{v}\pars{y} &:= \abs{v\pars{y}}, \label{eq:abs} \\
        \sup &: \mfrak{P}\pars{\mcal{V}_{w,\infty}} \to \mcal{V}_{w,\infty},
        & \sup\pars{V}\pars{y} &:= \sup_{v\in V} v\pars{y}, \text{ and} \label{eq:sup} \\
        \inf &: \mfrak{P}\pars{\mcal{V}_{w,\infty}} \to \mcal{V}_{w,\infty},
        & \inf\pars{V}\pars{y} &:= \inf_{v\in V} v\pars{y} .
    \end{align}
    This allows us to prove the continuity of $\operatorname{sqr}\circ \sup\circ \operatorname{abs} : \mfrak{P}\pars{\mcal{V}_{w,\infty}} \to \mcal{V}_{w^2,\infty}$ and $U$ individually.
    The main difference between \crefrange{thm:K_properties:continuity:compact_sets}{thm:K_properties:continuity:cones} then comes from the domain of $U$.
    
    We proceed by showing that $\operatorname{sqr}\circ \sup\circ \operatorname{abs} : \mfrak{P}\pars{\mcal{V}_{w,\infty}} \to \mcal{V}_{w^2,\infty}$ is continuous with respect to the Hausdorff pseudometric, which also implies the continuity of $\operatorname{sqr}\circ \sup\circ \operatorname{abs} : \mfrak{C}\pars{\mcal{V}_{w,\infty}} \to \mcal{V}_{w^2,\infty}$ with respect to the Hausdorff metric.
    
    To do this we require the following four lemmata.
    \begin{lemma} \label{lem:continuity_implies_pseudo-Hausdorff_continuity}
        Let $\pars{M_1,d_1}$ and $\pars{M_2,d_2}$ be metric spaces, let $f:M_1\to M_2$ and define $f\pars{X} := \braces{ f\pars{x} : x\in X}$ for any $X\in M_1$.
        If $f$ is uniformly continuous, then $f : \mfrak{P}\pars{M_1}\to\mfrak{P}\pars{M_2}$ is uniformly continuous with respect to the Hausdorff pseudometric.
    \end{lemma}
    \begin{proof}
        Recall, that $d_{\mathrm{H}}\pars{X,Y} \le \varepsilon$ means that
        \begin{equation}
            \forall x\in X\,\exists y\in Y:\, d\pars{x,y} \le \varepsilon
            \qquad\text{and}\qquad
            \forall y\in Y\,\exists x\in X:\, d\pars{y,x} \le \varepsilon .
        \end{equation} \todo{Make this the definition!}
    
        Let $\varepsilon>0$.
        Since $f$ is uniformly continuous there exists $\delta>0$ such that $d_1\pars{x,y}<\delta$ implies $d_2\pars{f\pars{x},f\pars{y}}<\varepsilon$.
        We now show that $d_{\mathrm{H}}\pars{U,V}<\delta$ implies $d_{\mathrm{H}}\pars{f\pars{U}, f\pars{V}}<\varepsilon$.

        For this let $f_u\in f\pars{U}$ and choose $u\in U$ such that $f\pars{u} = f_u$.
        Since $d_{\mathrm{H}}\pars{U,V}<\delta$ there exists $v\in V$ such that $d_1\pars{u,v}<\delta$ and consequently $d_2\pars{f\pars{u},f\pars{v}}<\varepsilon$, by uniform continuity.
        This means that for every $f_u\in f\pars{U}$ there exists $f_v \in f\pars{V}$ such that $d_2\pars{f_u, f_v} < \varepsilon$.
        Since this argument remains valid if the roles of $U$ and $V$ are reversed we can conclude that $d_{\mathrm{H}}\pars{f\pars{U}, f\pars{V}} < \varepsilon$.
    \end{proof}
    
    \begin{lemma} \label{lem:abs_continuous}
        $\operatorname{abs} : \mcal{V}_{w,\infty} \to \mcal{V}_{w,\infty}$ is Lipschitz continuous with constant $1$.
        $\operatorname{abs} : \mfrak{P}\pars{\mcal{V}_{w,\infty}} \to \mfrak{P}\pars{\mcal{V}_{w,\infty}}$ is uniformly continuous with respect to the Hausdorff pseudometric.
    \end{lemma}
    \begin{proof}
        The first assertion follows by the reverse triangle inequality, $\abs{\abs{v\pars{y}}-\abs{w\pars{y}}} \le \abs{v\pars{y}-w\pars{y}}$.
        The second asserion follows by \cref{lem:continuity_implies_pseudo-Hausdorff_continuity}, since Lipschitz continuity implies uniform continuity.
    \end{proof}
    
    \begin{lemma} \label{lem:sup_continuous}
        $\sup : \mfrak{P}\pars{\mcal{V}_{w,\infty}} \to \mcal{V}_{w,\infty}$ is Lipschitz continuous with constant $1$.
    \end{lemma}
    \begin{proof}
        Let $U,V \in \mfrak{P}\pars{\mcal{V}_{w,\infty}}$ and assume w.l.o.g.\ that $\sup\pars{U}\pars{y} \ge\sup\pars{V}\pars{y}$.
        Then $\norm{\sup\pars{U} - \sup\pars{V}}_{w,\infty} \le d_{\mathrm{H}}\pars{U,V}$ follows via
        \begin{equation}
            \abs{\sup\pars{U}\pars{y} - \sup\pars{V}\pars{y}}
            = \sup_{u\in U}\inf_{v\in V} u\pars{y}-v\pars{y} 
            \le \sup_{u\in U}\inf_{v\in V} \norm{u-v}_{w,\infty}
            \le d_{\mathrm{H}}\pars{U, V}
        \end{equation}
        which proves the assertion.
    \end{proof}
    
    \begin{lemma} \label{lem:sqr_continuous}
        $\operatorname{sqr} : \mcal{V}_{w,\infty} \to \mcal{V}_{w^2,\infty}$ is continuous.
        \note{It is true that products of essentially bounded functions are essentially bounded. This however may not be the case for weighted boundedness! As an example, consider $w\pars{x}=\exp\pars{-x^2}$. Then $w^{-1}$ is $w$-bounded, but $w^{-2}$ is not! If we use a weight function $w$, then $\operatorname{sqr}$ maps from $\mcal{V}_{w,\infty}$ to $\mcal{V}_{w^2,\infty}$.}
    \end{lemma}
    \begin{proof}
        Fix $v\in\mcal{V}_{w,\infty}$ and let $w\in\mcal{V}_{w,\infty}$ be arbitrary.
        Then 
        \begin{align}
            \norm{v^2 - w^2}_{w^2,\infty}
            &= \norm{vv - vw + vw - ww}_{w^2,\infty} \\
            &\le \norm{v\pars{v-w}}_{w^2,\infty} + \norm{w\pars{v-w}}_{w^2,\infty} \\
            &\le \pars{\norm{v}_{w,\infty}+\norm{w}_{w,\infty}}\norm{v-w}_{w,\infty} .
        \end{align}
        Observe that, due the reverse triangle inequality, $\norm{v-w}_{w,\infty}\le\delta$ implies $\norm{w}_{w,\infty} \le \norm{v}_{w,\infty}+\delta$.
        This proves continuity, since for any $\varepsilon$ we can choose $\delta$ such that $\norm{v-w}_{w,\infty}\le \delta$ implies
        \begin{equation}
            \norm{v^2 - w^2}_{w^2,\infty} \le \pars{2\norm{v}_{w,\infty} + \delta} \delta \le \varepsilon .
        \end{equation}
    \end{proof}
    
    As a composition of continuous functions, the continuity of $\operatorname{sqr}\circ \sup\circ \operatorname{abs} : \mfrak{P}\pars{\mcal{V}_{w,\infty}} \to \mcal{V}_{w^2,\infty}$ is guaranteed by \crefrange{lem:abs_continuous}{lem:sqr_continuous}.
     
    \item To prove this we need the subsequent lemma.
    \begin{lemma} \label{lem:continuity_implies_Hausdorff_continuity}
        Let $\pars{M_1,d_1}$ and $\pars{M_2,d_2}$ be metric spaces, let $f:M_1\to M_2$ and define $f\pars{X} := \braces{ f\pars{x} : x\in X}$ for any $X\in M_1$.
        If $f$ is continuous, then $f : \mfrak{C}\pars{M_1}\to\mfrak{C}\pars{M_2}$ is continuous with respect to the Hausdorff metric.
    \end{lemma}
    \begin{proof}
        $f : \mfrak{C}\pars{M_1}\to\mfrak{C}\pars{M_2}$ is well-defined since the image of a compact set under a continuous function is compact.
        Now recall, that $d_{\mathrm{H}}\pars{X,Y} \le \varepsilon$ means that
        \begin{equation}
            \forall x\in X\,\exists y\in Y:\, d\pars{x,y} \le \varepsilon
            \qquad\text{and}\qquad
            \forall y\in Y\,\exists x\in X:\, d\pars{y,x} \le \varepsilon .
        \end{equation}
        
        Let $\varepsilon>0$ and $U\in\mfrak{C}\pars{M_1}$.
        Since $f$ is continuous in every $u\in U$ there exists a $\delta_u > 0$ that guarantees
        \begin{equation}
            d_1\pars{u,\tilde{u}} \le \delta_u \Rightarrow d_2\pars{f\pars{u}, f\pars{\tilde{u}}} \le \frac{\varepsilon}{2} .
        \end{equation}
        Now define the sets $N_u := \braces{\tilde{u}\in M_1 : d_1\pars{u,\tilde{u}} \le \frac{\delta_u}{2}}$.
        Since $u\in N_u$, the family $\braces{N_u}_{u\in U}$ defines a covering of $U$ and since $U$ is compact there exists a finite subcovering $\braces{N_{u_i}}_{i=1,\ldots,n}$.
        Choose $\delta := \min_{i=1,\ldots,n}\frac{\delta_{u_i}}{2}$ and note, that $\delta$ has to be positive, since it is the minimum of finitely many positive numbers.
        Now let $V\in\mfrak{C}\pars{M_1}$ such that $d_{\mathrm{H}}\pars{U,V}\le\delta$.
        
        First, we show that
        \begin{equation}
            \forall f_v\in f\pars{V}\exists f_u\in f\pars{U} : d_2\pars{f_v, f_u} \le \varepsilon .
        \end{equation}
        For this let $v\in V$ be any element that satisfies $f\pars{v} = f_v$.
        Since $d_{\mathrm{H}}\pars{U,V}\le\delta$ there exists $u\in U$ with $d_1\pars{u,v}\le\delta$.
        Moreover, by definition of the covering $\braces{N_{u_i}}_{i=1,\ldots,n}$, there exists $u_i$ such that $d_1\pars{u,u_i}\le\frac{\delta_{u_i}}{2}$.
        Using the triangle inequality, we thus obtain
        \begin{equation}
            d_1\pars{u_i,v} \le d_1\pars{u_i,u} + d_1\pars{u,v} \le \frac{\delta_{u_i}}{2} + \delta \le \delta_{u_i}
        \end{equation}
        and the definition of $\delta_{u_i}$ finally yields $d_2\pars{f\pars{v}, f\pars{u_i}} \le \frac{\varepsilon}{2} \le \varepsilon$.
        
        Now we show that
        \begin{equation}
            \forall f_u\in f\pars{U}\exists f_v\in f\pars{V} : d_2\pars{f_u, f_v} \le \varepsilon .
        \end{equation}
        Analogously to the argument from above let $u\in U$ be any element that satisfies $f\pars{u}=f_u$.
        Since $d_{\mathrm{H}}\pars{U,V}\le\delta$ there exists $v\in V$ with $d_1\pars{u,v}\le\delta$ and by the definition of the covering there exists also a $u_i$ with $d_1\pars{u,u_i}\le\frac{\delta_{u_i}}{2}$.
        We can now estimate
        \begin{equation}
            d_2\pars{f\pars{u},f\pars{v}} \le d_2\pars{f\pars{u},f\pars{u_i}} + d_2\pars{f\pars{u_i},f\pars{v}} \le \frac{\varepsilon}{2} + \frac{\varepsilon}{2} = \varepsilon
        \end{equation}
        which holds by the definition of $\delta_{u_i}$ and because
        \begin{equation}
            d_1\pars{u_i,v} \le d_1\pars{u_i,u} + d_1\pars{u,v} \le \frac{\delta_{u_i}}{2} + \delta \le \delta_{u_i} .
        \end{equation}
    \end{proof}
    
    Since the function $u \mapsto u/\norm{u}$ is continuous on $\mcal{V}_{w,\infty}\setminus \braces{0}$ the function $U : \mfrak{C}\pars{\mcal{V}_{w,\infty}\setminus\braces{0}} \to \mfrak{C}\pars{S\pars{0,1}\cap\mcal{V}_{w,\infty}}$ is continuous by \cref{lem:continuity_implies_Hausdorff_continuity}.
    \note{$U : \mfrak{P}\pars{\mcal{V}_{w,\infty}} \to \mfrak{P}\pars{S\pars{0,1}\cap\mcal{V}_{w,\infty}}$ ist nicht stetig. Gegenbsp ist die Menge $\pars{0,1} \subseteq \mbb{R} = \mcal{V}_{w,\infty}$, denn $d_{\mathrm{H}}\pars{\pars{-\delta/2,1}, \pars{0,1}} < \delta$, aber $d_{\mathrm{H}}\pars{U\pars{-\delta/2,1}, U\pars{0,1}} = d_{\mathrm{H}}\pars{\braces{-1,1}, \braces{1}} = 2$.}
    
    \item Let $r>0$.
    Since the function $u \mapsto u/\norm{u}$ is uniformly continuous on $\mcal{V}_{w,\infty}\setminus B\pars{0,r}$ the function $U : \mfrak{P}\pars{\mcal{V}_{w,\infty}\setminus B\pars{0,r}} \to \mfrak{P}\pars{S\pars{0,1}\cap\mcal{V}_{w,\infty}}$ is uniformly continuous by \cref{lem:continuity_implies_pseudo-Hausdorff_continuity}.
    
    \item By definition of the truncated Hausdorff distance, $U : \operatorname{Cone}\pars{\mfrak{P}\pars{\mcal{V}_{w,\infty}}} \to \mfrak{P}\pars{S\pars{0,1}\cap\mcal{V}_{w,\infty}}$ is Lipschitz continuous with constant $1$.
    
    \item[6.-7.] Every $v\in A + B$ can be written as $v = \vec{v}^\intercal\alpha$ for some $\alpha\in\mathbb{R}^2$ and $\vec{v}\in \pars{A\times B}\setminus\braces{0}$.
    Moreover, $A\perp B$ implies that $\norm{v}^2 = \alpha^\intercal D(\vec{v})^2 \alpha$ with $D(\vec{v}) := \operatorname{diag}\pars{\norm{\vec{v}_1}, \norm{\vec{v}_2}}$.
    Now define $C_{\vec{v},y} = D(\vec{v})^{-1}\vec{v}(y)$ and observe that
    \note{The second supremum is the reason why equality does not hold. If $A$ and $B$ are singletons, then only $\alpha = \pars{1\ 1}$ is admissible.}
    \begin{align}
        \mfrak{K}_{A + B}\pars{y} 
        &\le \sup_{\vec{v}\in \pars{A\times B}\setminus\braces{0}} \sup_{\alpha\in\mathbb{R}^2\setminus\braces{0}}\frac{\abs{\alpha^\intercal \vec{v}(y)\vec{v}(y)^\intercal \alpha}}{\alpha^\intercal D(\vec{v})^2 \alpha}
        = \sup_{\vec{v}\in \pars{A\times B}\setminus\braces{0}} \sup_{\beta\in\mathbb{R}^2\setminus\braces{0}} \frac{\abs{\beta^\intercal C_{\vec{v},y}C_{\vec{v},y}^\intercal \beta}}{\beta^\intercal \beta} \\
        &= \sup_{\vec{v}\in \pars{A\times B}\setminus\braces{0}} \norm{C_{\vec{v},y}}_2^2
        = \sup_{\vec{v}_1\in A\setminus\braces{0}} \sup_{\vec{v}_2\in B\setminus\braces{0}} \tfrac{\abs{\vec{v}_1(y)}^2}{\norm{\vec{v}_1}^2} + \tfrac{\abs{\vec{v}_2(y)}^2}{\norm{\vec{v}_2}^2}
        = \mfrak{K}_{A}\pars{y} + \mfrak{K}_{B}\pars{y} .
    \end{align}
    Note that the first inequality is indeed an equality, if $A$ and $B$ are linear spaces.
    \note{We can generalize this from $\abs{v}_y = \abs{v\pars{y}}$ to $\abs{v}_y = \norm{G_yv}_2$ (with e.g.\ $G_y v = \pars{v\pars{y}\ v'\pars{y}}$) and show $\mfrak{K}_{A+B}\le \mfrak{K}_A + \mfrak{K}_B$.
    Let $\abs{v}_y^2 := \norm{L_y v}_2^2 = \abs{\alpha^\intercal \operatorname{diag}\pars{\vec{v}}^\intercal \begin{pmatrix}L_y^\intercal\\L_y^\intercal\end{pmatrix} \begin{pmatrix}L_y&L_y\end{pmatrix} \operatorname{diag}\pars{\vec{v}}\alpha}$ and define $C_{\vec{v},y} := \begin{pmatrix}L_y&L_y\end{pmatrix} \operatorname{diag}\pars{\vec{v}} D\pars{\vec{v}}^{-1} = \begin{pmatrix}\frac{L_y\vec{v}_1}{\norm{v_1}}&\frac{L_y\vec{v}_2}{\norm{v_2}}\end{pmatrix}$.
    Then the above argument can be repeated with $\sup_{\beta} \frac{\abs{\beta^\intercal C_{\vec{v},y}^\intercal C_{\vec{v},y}\beta}}{\beta^\intercal\beta} \le \norm{C_{\vec{v},y}}_{\mathrm{Fro}}^2$ and $\sup_{\vec{v}_1}\sup_{\vec{v}_2}\norm{C_{\vec{v},y}}_{\mathrm{Fro}}^2 = \sup_{\vec{v}_1} \frac{\norm{L_y v_1}^2}{\norm{v_1}^2} + \sup_{\vec{v}_2} \frac{\norm{L_y v_2}^2}{\norm{v_2}^2}$ which concludes the proof.
    }
    \addtocounter{enumi}{2}

    \item Let $a\in A$ and $b\in B$.
    Since $a\indep b$ also $a^2\indep b^2$ and consequently $\norm{a\cdot b}^2 = \mbb{E}\bracs{a^2 b^2} = \mbb{E}\bracs{a^2} \mbb{E}\bracs{b^2} = \norm{a}^2\norm{b}^2$.
    Now recall that $\mathfrak{K}_A\pars{y} = \sup_{a\in U\pars{A}} a\pars{y}^2$. Thus
    \begin{equation}
        \mathfrak{K}_{A\cdot B}\pars{y}
        = \sup_{a\in A} \sup_{b\in B} \frac{\pars{a\cdot b}\pars{y}^2}{\norm{a\cdot b}^2}
        = \sup_{a\in A} \sup_{b\in B} \frac{a\pars{y}^2\cdot b\pars{y}^2}{\norm{a}^2\norm{b}^2}
        = \mathfrak{K}_{A}\pars{y} \cdot \mathfrak{K}_{B}\pars{y} .
    \end{equation}
    \todo{For this to hold for arbitrary $\abs{\bullet}_y$ it needs to hold that $\abs{a\cdot b}_y = \abs{a}_y\abs{b}_y$. This is the case for $L^2$ and $H^1_0$ but not for $H^1$.
    Otherwise it holds that $\mathfrak{K}_{A\cdot B}\pars{y} \le \mathfrak{K}_{A}\pars{y} \cdot \mathfrak{K}_{B}\pars{y}$.
    \begin{itemize}
        \item This also holds for GRKHS. For the sake of simplicity assume $A\subseteq L^2\pars{Y_1,\rho_1}$ and  $B\subseteq L^2\pars{Y_2,\rho_2}$.
        \item Note that for general GRKHS $\abs{\bullet}_y$ is different for $a$, $b$ and $a\cdot b$. On $Y_1$ and $Y_2$ e.g.\ it is $H^1\pars{Y_{1/2}}$ but on $Y_1\times Y_2$ it is $H^{1,\mathrm{mix}} = H^1\otimes H^1$!
        \item Then $\abs{a\cdot b}_y = \norm{\pars{L^{Y_1}_y a} \otimes \pars{L^{Y_2}_y b}}_{\mathrm{Fro}} = \norm{L^{Y_1}_y a}_2 \norm{L^{Y_2}_y b}_2$ where $\pars{L^{Y_1}_y a} \otimes \pars{L^{Y_2}_y b}$ is the outer product between the vectors $\pars{L^{Y_1}_y a}$ and $\pars{L^{Y_2}_y b}$.
        I.e.\ we need $L^{Y_1\times Y_2}_y = L^{Y_1}_{y_1}\otimes L^{Y_2}_{y_2}$.
        \item Finally, $\norm{a\cdot b} = \norm{a}\norm{b}$. (E.g.\ for $H^1$: $\norm{ab}_{H^1\otimes H^1}^2 = \inner{ab,ab}+\inner{a'b,a'b}+\inner{ab',ab'}+\inner{a'b',a'b'} = \inner{a,a}\inner{b,b}+\inner{a',a'}\inner{b,b}+\inner{a,a}\inner{b',b'}+\inner{a',a'}\inner{b',b'} = \pars{\inner{a,a}+\inner{a',a'}} + \pars{\inner{b,b}+\inner{b',b'}} = \norm{a}_{H^1}^2 \norm{b}_{H^1}^2$.)
        \item Zumindest für normale Sobolevräume gilt dass $v\in H^k$ und $\Phi$ ein $k$-diffeomorphismus, dann ist $v\circ\Phi\in H^k$. Etwas ähnliches gilt bestimmt auch für $H^{k,\mathrm{mix}}$.
        Deswegen können wir wieder verallgemeinern auf $A\indep B$ (i.e.\ es existiert eine Koordinaten transformation $Y_1\times Y_2 \to Z_1\times Z_2$ sodass $A\subseteq L^2\pars{Z_1,\rho_1}$ und $B\subseteq L^2\pars{Z_2,\rho_2}$).
    \end{itemize}
    }

    \item A direct consequence of \cref{thm:K_properties:direct_sum} is the following lemma.
    \begin{lemma} \label{lem:K_properties:span}
        Let $\braces{P_j}_{j\in J}$ be an orthonormal basis for $A$. Then $\mfrak{K}_{A}\pars{y} = \sum_{j\in J} P_j\pars{y}^2$. \hfill\qedsymbol
    \end{lemma}
    Now let $\braces{P_{A,j}}_{j\in J}$ be an orthonormal basis of $A$ and $\braces{P_{B,k}}_{k\in K}$ be an orthonormal basis of $B$.
    Then $\braces{P_{A,j}\otimes P_{B,k}}_{j\in J, k\in K}$ is an orthonormal basis for $A\otimes B$ and by \cref{lem:K_properties:span}
    \begin{equation}
        \mfrak{K}_{A\otimes B}\pars{y}
        = \sum_{j\in J}\sum_{k\in K} P_{A,j}\pars{y}^2\cdot P_{B,k}\pars{y}^2
        = \pars*{\sum_{j\in J} P_{A,j}\pars{y}^2}\cdot\pars*{\sum_{k\in K} P_{B,k}\pars{y}^2}
        = \mfrak{K}_A\pars{y}\cdot \mfrak{K}_B\pars{y} .
    \end{equation}
\end{enumerate}

\section{Proof of Theorem~\ref{thm:convergence_M}} \label{proof:thm:convergence_M}

Recall that $R = \operatorname{rch}\pars{\mcal{M}\cap B\pars{u,r_0}}$ and $r \le \min\braces{r_0, R}$ and define $C := \pars{2R}^{-1}$.
Also recall that $d_{\mathrm{H}}\pars{\mcal{M}\cap B\pars{u,r}, \pars{u+\mbb{T}_u\mcal{M}}\cap B\pars{u,r}} \le Cr^2$ is equivalent to the conjunction of the following two statements.
\begin{enumerate}
    \item For every $v\in \mcal{M}\cap B\pars{u,r}$ there exists a $w\in \pars{u+\mbb{T}_u\mcal{M}}\cap B\pars{u,r}$ such that $\norm{v-w}\le Cr^2$.
    \item For every $w\in \pars{u+\mbb{T}_u\mcal{M}}\cap B\pars{u,r}$ there exists a $v\in \mcal{M}\cap B\pars{u,r}$ such that $\norm{v-w}\le Cr^2$.
\end{enumerate}

\paragraph{Proof of 1.}
This statement characterizes the reach of a set.
An easily accessible proof that relies only on the definition of $R$ and fundamental geometric arguments is presented in \cite[Theorem~7.8~(2)]{boissonnat2018geometric}.
We reiterate it in the following  since the proof of the second statement relies on similar arguments.

Let $v\in\mcal{M}\cap B\pars{u,r}$. %
Then there exists a unique best approximation of $v$ in $u+\mbb{T}_u\mcal{M}$ which we denote by $w$.
To show that $\norm{v-w}\le C\norm{u-v}^2$ we consider the intersection of the sets 
$\mcal{M}$ and $u+\mbb{T}_u\mcal{M}$ with the plane $\inner{u, v, w}$.
Since all three points lie in this plane their relative distances are preserved and it suffices to consider this two-dimensional problem from here on.
Let $D$ be the disk of radius $R$ that is tangent to $\mbb{T}_u\mcal{M}$ at $u$ and whose center $c$ is on the same side of $u+\mbb{T}_u\mcal{M}$ as $v$.
This is illustrated in \cref{fig:projecting_M_to_TM}.
Since $D$ is tangent to $\mbb{T}_u\mcal{M}$ and has radius $R$, it follows that $D$ only intersects $\mcal{M}$ in $u$.
Hence, $v$ does not lie in the interior of $D$ and the line segment $\overline{uv} = \braces{\lambda u + \pars{1-\lambda}v : \lambda\in\bracs{0,1}}$ must intersect the boundary of $D$ in a point $x$.
Since $\norm{v-w} = \norm{u-v}\sin\pars{\alpha}$, it suffices to bound $\sin\pars{\alpha}$.

Note that $\Delta\pars{u,c,x}$ is an isosceles triangle which entails that $\beta = 2\alpha$ and $\norm{u-x} = 2R\sin\pars{\frac{\beta}{2}} = 2R\sin\pars{\alpha}$.
Using $\norm{u-x}\le\norm{u-v}$ yields
\begin{equation} \label{eq:v-w_bound}
    \norm{v-w} = \norm{u-v}\sin\pars{\alpha} = \norm{u-v}\frac{\norm{u-x}}{2R} \le C\norm{u-v}^2 \le Cr^2.
\end{equation}
Finally, note that, by the Pythagorean theorem, $\norm{u-w}^2 = \norm{u-v}^2 - \norm{v-w}^2 \le r^2$ and thus $w\in\pars{u+\mbb{T}_u\mcal{M}}\cap B\pars{u,r}$.

\begin{figure}[ht]
    \centering
    \newcommand{\degre}{\ensuremath{^\circ}}
    \definecolor{angleColor}{rgb}{0.5,0.5,0.5}
    \definecolor{denim}{rgb}{0.08, 0.38, 0.74}
    \tikzset{
    dot/.style = {circle, fill, minimum size=#1, inner sep=0pt, outer sep=0pt},
    dot/.default = 3.5pt  %
    }
    \newcommand*\curvature{0.1}
    \newcommand*\radius{3.0}  %
    \newcommand*\vx{5.5}
    \pgfmathdeclarefunction{Parabola}{1}{\pgfmathparse{-(#1)^2*\curvature/2}}
    \begin{tikzpicture}[line cap=round,line join=round]
    \newcommand*\xlim{7.5}
    \clip(-\xlim,-4) rectangle (\xlim,1);  %
    
    \coordinate (u) at (0,0);
    \coordinate (c) at (0,-\radius);
    \coordinate (v) at ($(\vx,{Parabola(\vx)})$);
    \coordinate (w) at (\vx,0);
    
    \draw [line width=1.2pt,domain=-6.5:6.5,samples=50,name path=parabola] plot (\x, {Parabola(\x)});
    {
        \def\Mx{-5.75}
        \draw[color=black] ($(\Mx,{Parabola(\Mx)})$) node[anchor=north] {$\mcal{M}$};
    }
    \draw [line width=1.2pt,domain=-\xlim:\xlim,name path=line] plot (\x, {0*\x});
    {
        \def\TuMx{-6.5}
        \draw[color=black] (\TuMx,0) node[anchor=south] {$u+\mbb{T}_u\mcal{M}$};
    }
    \draw [name path=circle] (c) circle (\radius);
    {
        \def\Dphi{170}
        \draw[color=black] ($(c) + (\Dphi:\radius)$) node[anchor=east] {$D$};
    }
    
    \node[dot, label=above:$u$] at (u) {};
    \node[dot, label=below:$v$] at (v) {};
    \node[dot, label=above:$w$] at (w) {};
    \draw (v) -- (w);
    \draw [dash pattern=on 3pt off 3pt,name path=segment] (u) -- (v);
    \path [name intersections={of = circle and segment}];
	\node[dot, label={[shift={(-0.4,-0.4)}]:$x$}] (x) at (intersection-2) {};
    \node[dot, label=below:$c$] at (c) {};
    \draw (c) -- (u) node [midway, left] {$R$}
          (c) -- (x); %
          
    \draw let
        \p1 = (x),
        \p2 = (u),
        \n1 = {veclen(\x1-\x2,\y1-\y2)},
    in
        \pgfextra{\pgfmathsetmacro{\sinAngle}{\n1/(2cm*\radius)}}
        \pgfextra{\pgfmathsetmacro{\startAngle}{90-2*asin(\sinAngle)}}
        ($(c)+(\startAngle:1)$) arc (\startAngle:90:1) node [midway, label={[shift={(-0.03,-0.65)}]$\beta$}] {};

    \draw let
        \p1 = (u),
        \p2 = (v),
        \p3 = (w),
        \n1 = {veclen(\x1-\x2,\y1-\y2)},  %
        \n2 = {veclen(\x2-\x3,\y2-\y3)},  %
    in
        \pgfextra{\pgfmathsetmacro{\sinAngle}{\n2/\n1}}
        \pgfextra{\pgfmathsetmacro{\startAngle}{-asin(\sinAngle)}}
        ($(u)+(\startAngle:2)$) arc (\startAngle:0:2) node[label={[shift={(0,-0.07)}]$\alpha$}] {};
    \end{tikzpicture}
    \caption{}
    \label{fig:projecting_M_to_TM}
\end{figure}

\paragraph{Proof of 2.}
Let $w\in \pars{u+\mbb{T}_u\mcal{M}}\cap B\pars{u,r}$.
By \cref{prop:reach_increases} we know that $\rch\pars{\mcal{M}\cap B\pars{u,r}} \ge R$
and since $r\le R$ there exists a best approximation of $w$ in $\mcal{M}\cap B\pars{u,r}$ which we denote by $v$.
To show that $\norm{v-w} \le C\norm{u-w}^2$ we consider again the intersection of the sets 
$\mcal{M}$ and $u+\mbb{T}_u\mcal{M}$ with the plane $\inner{u, v, w}$.
Again, the distance between the points is preserved and we can consider the resulting two-dimensional problem.
This is illustrated in \cref{fig:projecting_TM_to_M}.
Let $D$ be the disk of radius $R$ that is tangent to $\mcal{M}$ at $u$ and whose center $c$ is on the same side of $\mbb{T}_u\mcal{M}$ as $v$.
Note, that the best approximation of $w$ in $D$ is given by $x := R\frac{w-c}{\norm{w-c}} + c$ and denote the intersection of the line segment $\overline{wx} = \braces{\lambda w + \pars{1-\lambda}x : \lambda\in\bracs{0,1}}$ with $\mcal{M}$ by $\tilde{v}$.
By the best approximation property and the definition of $\tilde{v}$ and $x$ it follows that 
\begin{equation}
    \norm{w-v} \le \norm{w-\tilde{v}} \le \norm{w-x} .
\end{equation}
It thus suffices to bound $\norm{w-x}$ which is given by the Pythagorean theorem as $\norm{w-x} = \sqrt{R^2+\norm{w-u}^2} - R$.

Defining $\ell\pars{r} := \sqrt{R^2+r^2} - R$ and $\tilde{\ell}\pars{r} := \frac{r^2}{2R}$ we observe that $\ell\pars{r} \le \tilde{\ell}\pars{r}$ since $\ell\pars{0} = 0 = \tilde{\ell}\pars{0}$ and
\begin{equation}
    \ell'\pars{r} = \frac{r}{\sqrt{R^2+r^2}} \le \frac{r}{R} = \tilde{\ell}'\pars{r} .
\end{equation}
This yields $\norm{w-v}\le\norm{w-x} = \ell\pars{\norm{w-u}} \le \tilde\ell\pars{\norm{w-u}} = Cr^2$ and concludes the proof.
\qedappendix

\begin{figure}[ht]
    \centering
    \newcommand{\degre}{\ensuremath{^\circ}}
    \definecolor{angleColor}{rgb}{0.5,0.5,0.5}
    \definecolor{denim}{rgb}{0.08, 0.38, 0.74}
    \tikzset{
    dot/.style = {circle, fill, minimum size=#1, inner sep=0pt, outer sep=0pt},
    dot/.default = 3.5pt  %
    }
    \newcommand*\curvature{0.1}
    \newcommand*\radius{3.0}  %
    \newcommand*\wx{5.5}
    \pgfmathdeclarefunction{Parabola}{1}{\pgfmathparse{-(#1)^2*\curvature/2}}
    \begin{tikzpicture}[line cap=round,line join=round]
    \newcommand*\xlim{7.5}
    \clip(-\xlim,-4) rectangle (\xlim,1);  %
    
    \coordinate (u) at (0,0);
    \coordinate (c) at (0,-\radius);
    \coordinate (w) at (\wx,0);
    
    \draw [line width=1.2pt,domain=-6.5:6.5,samples=50,name path=parabola] plot (\x, {Parabola(\x)});
    {
        \def\Mx{-5.75}
        \draw[color=black] ($(\Mx,{Parabola(\Mx)})$) node[anchor=north] {$\mcal{M}$};
    }
    \draw [line width=1.2pt,domain=-\xlim:\xlim,name path=line] plot (\x, {0*\x});
    {
        \def\TuMx{-6.5}
        \draw[color=black] (\TuMx,0) node[anchor=south] {$u+\mbb{T}_u\mcal{M}$};
    }
    \draw [name path=circle] (c) circle (\radius);
    {
        \def\Dphi{170}
        \draw[color=black] ($(c) + (\Dphi:\radius)$) node[anchor=east] {$D$};
    }
    
    \node[dot, label=above:$u$] at (u) {};
    \node[dot, label=below:$c$] at (c) {};
    \node[dot, label=above:$w$] at (w) {};
    \draw [name path=segment] (c) -- (w);
    
    \path [name intersections={of = circle and segment}];
	\node[dot, label={[shift={(0.08,0.08)}]$x$}] (x) at (intersection-1) {};
    \draw (c) -- (u) node [midway, left] {$R$}
          (c) -- (x); %
	
    \path [name intersections={of = parabola and segment}];
	\node[dot, label=below:$\tilde{v}$] (tilde_v) at (intersection-1) {};
	
	\coordinate (v) at ($(\wx,{Parabola(\wx)})!0.4!(tilde_v)$);
    \node[dot, label=below:$v\vphantom{\tilde{v}}$] at (v) {};
    \draw (v) -- (w);
    \end{tikzpicture}
    \caption{}
    \label{fig:projecting_TM_to_M}
\end{figure}

\section{Proof of Theorem~\ref{thm:convergence_UM}} \label{proof:thm:convergence_UM}

Recall that $R = \operatorname{rch}\pars{\mcal{M}\cap B\pars{u,r_0}}$ and $r \le \min\braces{r_0, R}$ and define $C := \pars{2R}^{-1}$.
To prove $d_{\mathrm{H}}\pars{U\pars{\mcal{M}\cap B\pars{u,r}-u}, U\pars{\mbb{T}_u\mcal{M}}} \le 2Cr$, note that $d_{\mathrm{H}}$ is induced by a norm and is therefore absolutely homogeneous and translation invariant.
Therefore,
\begin{equation}
    d_{\mathrm{H}}\pars{U\pars{\mcal{M}\cap B\pars{u,r}-u}, U\pars{\mbb{T}_u\mcal{M}}} = \frac{1}{r} d_{\mathrm{H}}\pars{rU\pars{\mcal{M}\cap B\pars{u,r}-u}, rU\pars{\mbb{T}_u\mcal{M}}} .
\end{equation}
Now define the operator $U_r\pars{X} := rU\pars{X}$ that scales every element of a set to norm $r$.
The claim follows if $d_{\mathrm{H}}\pars{U_r\pars{\mcal{M}\cap B\pars{u,r}-u}, U_r\pars{\mbb{T}_u\mcal{M}}} \le 2Cr^2$.
To prove this we need to show that the following two statements hold.
\begin{enumerate}
    \item For every $\hat{v}\in U_r\pars{\mcal{M}\cap B\pars{u,r}-u}$ there exists a $\hat{w}\in U_r\pars{\mbb{T}_u\mcal{M}}$ such that $\norm{\hat{v}-\hat{w}}\le 2Cr^2$.
    \item For every $\hat{w}\in U_r\pars{\mbb{T}_u\mcal{M}}$ there exists a $\hat{v}\in U_r\pars{\mcal{M}\cap B\pars{u,r}-u}$ such that $\norm{\hat{v}-\hat{w}}\le 2Cr^2$.
\end{enumerate}

\paragraph{Proof of 1.}
Let $\hat{v}\in U_r\pars{\mcal{M}\cap B\pars{u,r}-u}$ and let $v\in\mcal{M}\cap B\pars{u,r}-u$ be any element that satisfies $U_r\pars{\braces{v}}=\braces{\hat{v}}$.
In the proof of \cref{thm:convergence_M} we have shown that there exists a $w\in\mbb{T}_u\mcal{M}$ that satisfies $\norm{v-w}\le C\norm{v}^2$ (cf.~\cref{eq:v-w_bound}).
We use this $w$ to define 
\begin{equation}
    \tilde{v} := \frac{r}{\norm{v}}v,
    \quad
    \tilde{w} := \frac{r}{\norm{v}}w,
    \quad\text{and}\quad
    \hat{w} = \frac{r}{\norm{w}}w
\end{equation}
and observe that $\tilde{v} = \hat{v} \in U_r\pars{\mbb{T}_u\mcal{M}}$ and that $\hat{w}\in U_r\pars{\mbb{T}_u\mcal{M}}$.
Moreover, $\norm{\hat{v}-\hat{w}} \le \norm{\hat{v}-\tilde{w}} + \norm{\tilde{w}-\hat{w}}$ and $\norm{\tilde{v}-\tilde{w}} = \frac{r}{\norm{v}}\norm{v-w} \le Cr\norm{v} \le Cr^2$.
It thus remains to show that $\norm{\tilde{w}-\hat{w}} \le Cr^2$.

To see this we consider the intersection of $\mcal{M}-u$ and $\mbb{T}_u\mcal{M}$ with the plane $\inner{0,v,w}$.
This is illustrated in \cref{fig:unitization}.
Since all the points that we have defined so far reside in this plane, the distances between them are preserved and we can henceforth consider only this two-dimensional problem.

To show $a := \norm{\tilde{w}-\hat{w}} \le \norm{\tilde{w}-\tilde{v}} =: b$, we consider the triangle $\Delta\pars{\tilde{v},\tilde{w},0}$ and employ the Pythagorean theorem
\begin{equation}
    r^2 = \pars{r-a}^2 + b^2 .
\end{equation}
Expanding the product and rearranging the terms results in the equation $b^2 = 2ra - a^2$.
Since $r\ge a$ also $2ra \ge 2a^2$.
Therefore, $b^2 \ge 2a^2 - a^2 = a^2$ which is what we wanted to prove.

\paragraph{Proof of 2.}
Let $\hat{w}\in U_r\pars{\mbb{T}_u\mcal{M}}$.
Since $r\le R$, \cref{thm:convergence_M} guarantees that there exists a $v\in\mcal{M}\cap B\pars{u,r}-u$ such that $\norm{\hat{w}-v}\le Cr^2$.
Let $\hat{v} := \frac{r}{\norm{v}}v$ and observe that, by the reverse triangle inequality,
\begin{equation}
    \norm{\hat{w}} - \norm{v} \le \abs{\norm{\hat{w}} - \norm{v}} \le \norm{\hat{w}-v} \le Cr^2 .
\end{equation}
Rearranging the terms and substituting $\norm{\hat{w}} = r$ then yields $\norm{v} \ge r-Cr^2$.
It is now easy to estimate
\begin{equation}
    \norm{v-\hat{v}} = \abs*{1-\frac{r}{\norm{v}}}\norm{v} 
    = r-\norm{v} \le Cr^2 .
\end{equation}
Finally, using the triangle inequality, we obtain $\norm{\hat{w}-\hat{v}}\le \norm{\hat{w}-v} + \norm{v-\hat{v}}\le 2Cr^2$.
This concludes the proof.
\qedappendix

\begin{figure}[ht]
    \centering
    \newcommand{\degre}{\ensuremath{^\circ}}
    \definecolor{angleColor}{rgb}{0.5,0.5,0.5}
    \definecolor{denim}{rgb}{0.08, 0.38, 0.74}
    \tikzset{
    dot/.style = {circle, fill, minimum size=#1, inner sep=0pt, outer sep=0pt},
    dot/.default = 3.5pt  %
    }
    \newcommand*\xlim{7.5}
    \newcommand*\curvature{0.2}
    \newcommand*\radius{5.5}
    \newcommand*\vx{4.0}
    \pgfmathdeclarefunction{Parabola}{1}{\pgfmathparse{-(#1)^2*\curvature/2}}
    \begin{tikzpicture}[line cap=round,line join=round]
    
    \clip(-\xlim,-4) rectangle (\xlim,1);  %
    
    \coordinate (u) at (0,0);
	\coordinate (v) at ($(\vx,{Parabola(\vx)})$);
	\path let \p1 = (v) in  coordinate (w) at (\x1,0);
    \coordinate (hat_w) at (\radius,0);
    
    \draw [line width=1.2pt,domain=-6.5:6.5,samples=50,name path=parabola] plot (\x, {Parabola(\x)});
    {
        \def\Mx{-5.75}
        \draw[color=black] ($(\Mx,{Parabola(\Mx)})$) node[anchor=south east] {$\mcal{M}-u$};
    }
    \draw [line width=1.2pt,domain=-\xlim:\xlim,name path=line] plot (\x, {0*\x});
    {
        \def\TuMx{-6.5}
        \draw[color=black] (\TuMx,0) node[anchor=south] {$\mbb{T}_u\mcal{M}$};
    }
    \draw [name path=circle] (u) circle (\radius);
    
	\path [name path=ray] (u) -- ($(u)!2*\radius!(v)$);  %
    \path [name intersections={of = ray and circle}];
	\coordinate (hat_v) at (intersection-1);
    \path let \p1 = (hat_v) in coordinate (tilde_w) at (\x1,0);
    
    \node[above] at ({-\radius/2},0) {$r$};
    \node[dot, label=above:$0$] at (u) {};
    \node[dot, label=below:$v$] at (v) {};
    \node[dot, label=above:$w$] at (w) {};
    \node[dot, label=above right:$\hat{w}$] at (hat_w) {};
    \node[dot, label={right:$\tilde{v}=\hat{v}$}] at (hat_v) {};
    \node[dot, label={above:$\tilde{w}$}] at (tilde_w) {};
    
    \draw (v) -- (w);
    \draw [dash pattern=on 3pt off 3pt] (u) -- (hat_v) -- (tilde_w);
    
    \end{tikzpicture}
    \caption{}
    \label{fig:unitization}
\end{figure}

\note{Let in the following $A_r := u-\mcal{M}\cap B\pars{u,r}$ and $B_r := \mbb{T}_u\mcal{M}\cap B\pars{0,r}$.
Observe that $U_r\pars{B_r} = S\pars{0,r}\cap\mbb{T}_u\mcal{M} = U_r\pars{B_1}$ and that $U\pars{A_r}$ is a decreasing sequence with $U\pars{B_1}\subseteq U\pars{A_r}$ for all $r$.
Now we have the further inclusion
\begin{equation} \label{eq:Ur_subset}
    U_r\pars{A_r} \subseteq \bigcup_{\tilde{w}\in U_r\pars{\mbb{T}_u\mcal{M}}} S\pars{0,r}\cap B\pars{\tilde{w}, 2Cr^2} =: rC_r .
\end{equation}
Since $d_{\mathrm{H}}\pars{rC_r, S\pars{0,r}\cap\mbb{T}_u\mcal{M}} \le 2Cr^2$ we can conclude that $d_{\mathrm{H}}\pars{C_r, S\pars{0,1}\cap\mbb{T}_u\mcal{M}} \le 2Cr$.
Since $U\pars{B_r} \subseteq U\pars{A_r} \subseteq C_r$ we can conclude that $d_{\mathrm{H}}\pars{U\pars{A_r}, U\pars{B_r}} \le \frac{1}{r}2Cr^2 = 2Cr$.}

\section{Algorithm for computing the variation function in Figure~\ref{fig:local_variation_constant}} \label{sec:algorithm_KUbbone}

Let in the following $K\pars{A} := \sup_{a\in A} \norm{a}_\infty^2$ and observe that $\norm{\mfrak{K}_A}_\infty = K\pars{U\pars{A}}$ for any set $A\subseteq\mcal{V}$.
Moreover, let $K^{\mathrm{loc}}_{u,r}\pars{\mcal{M}} := K\pars{U\pars{\braces{u_{\mcal{M}}} - \mcal{M}\cap B\pars{u_{\mcal{M}},r}}}$.
We present an algorithm, which computes the quantity %
\begin{equation}
    K^{\mathrm{loc},\infty}_{u,r}\pars{\mcal{M}}
    := K\pars{U\pars{\braces{u_{\mcal{M}}} - \mcal{M}^{\mathrm{loc},\infty}_{u,r}}}
    \quad\text{with}\quad
    \mcal{M}^{\mathrm{loc},\infty}_{u,r} := \braces{v\in\mcal{M} : \norm{u_{\mcal{M}}-v}_\infty\le r} .
\end{equation}
Since $\norm{\bullet}_\infty \le \norm{\bullet}_2 \le \sqrt{N}\norm{\bullet}_\infty$ on the finite dimensional Euclidean space $\mbb{R}^N$ we can conclude that
\begin{equation}
    \mcal{M}\cap B\pars{u,N^{-1/2}r}
    \subseteq \mcal{M}^{\mathrm{loc},\infty}_{u,r}
    \subseteq \mcal{M}\cap B\pars{u,r}
    \quad\text{and hence}\quad
    K^{\mathrm{loc}}_{u,N^{-1/2}r}\pars{\mcal{M}} \le K^{\mathrm{loc},\infty}_{u,r}\pars{\mcal{M}} \le K^{\mathrm{loc}}_{u,r}\pars{\mcal{M}} .
\end{equation}
This equivalence justifies the use of this modified variation constant, since the rate of convergence of $K^{\mathrm{loc},\infty}_{u,r}\pars{\mcal{M}_l}$ equals that of $K^{\mathrm{loc}}_{u,r}\pars{\mcal{M}_l}$ for any sequence of model classes $\mcal{M}_l\subseteq\mbb{R}^N$.
The following proposition now shows how this modification allows us to simplify the computation of the variation constant.

\begin{proposition} \label{prop:KU_bbone_bound}
    Let $\mcal{M}$ be the set of rank-$1$ matrices in $\mbb{R}^{d_1\times d_2}$ and define $M_{\alpha,\beta,\gamma} := \pars{\alpha\ \beta\ \cdots\ \beta}^\intercal \pars{1\ \gamma\ \cdots\ \gamma}\in\mbb{R}^{d_1\times d_2}$ and $\mathsf{Max}_{M,r} := \braces{M_{\alpha,\beta,\gamma} : \norm{\bbone - M_{\alpha,\beta,\gamma}}_\infty = \abs{1-\alpha} \le r}$.
    Then
    \begin{equation}
        K\pars{U\pars{\bbone-\mcal{M}^{\mathrm{loc},\infty}_{\bbone,r}}}
        = \sup_{M_{\alpha,\beta,\gamma}\in \mathsf{Max}_{M,r}} \frac{\pars{1-\alpha}^2}{\norm{\bbone - M_{\alpha,\beta,\gamma}}_{\mathrm{Fro}}^2} .
    \end{equation}
\end{proposition}

With this proposition we can compute $\mcal{K}^{\mathrm{loc}}_{u,r}$ numerically. %
For fixed $r$, the condition $M_{\alpha,\beta,\gamma}\in \mathsf{Max}_{M,r}$ implies $\abs{1-\alpha} \le r$ and $\abs{1-\beta} \le \abs{1-\alpha}$.
We can hence discretize $\alpha$ in the range $1-r = \alpha_1 < \ldots < \alpha_m = 1+r$ and $\beta$ in the range $1-\abs{1-\alpha} = \beta_1 < \ldots < \beta_m = 1+\abs{1-\alpha}$ for some $m\in\mbb{N}$.
The resulting estimate
\begin{equation}
    K^{\mathrm{loc}}_{u,r} \approx \norm{\boldsymbol{K}}_{\mathrm{max}}
    \quad\text{with}\quad
    \boldsymbol{K}_{jk} = \sup_{\gamma\in\Gamma} \frac{\pars{1-\alpha_j}^2}{\norm{\bbone - M_{\alpha_j,\beta_k,\gamma}}_{\mathrm{Fro}}^2}
\end{equation}
converges due to the continuity of $\mfrak{K}$ proven in \cref{thm:K_properties:continuity:cones}.
Note, that each value $\boldsymbol{K}_{jk}$ is the solution to a one-dimensional quadratic minimization problem with a set of linear constraints $\gamma\in\Gamma$ that are induced by the constraint $M_{\alpha_j,\beta_k,\gamma}\in\mathsf{Max}_{M,r}$.
Due to this simple structure, the values $\boldsymbol{K}_{jk}$ can be computed analytically.

This idea can be generalized to rank-$1$ tensors of order $M$ and the resulting $\boldsymbol{K}$ is of order $M$ as well.
A low-rank approximation of $\boldsymbol{K}$ can be computed by cross-approximation \citep[c.f.][]{oseledets2010cross} and $\norm{\boldsymbol{K}}_{\mathrm{max}}$ can be computed by a modified power iteration \citep[c.f.][]{grasedyck2019TT_maxnorm}.

To prove \cref{prop:KU_bbone_bound} we require the following lemma.
\begin{lemma}
    Define $\mathsf{Max}_{r} := \braces{v\in\mbb{R}^{d_1\times d_2} : \norm{v}_\infty = \abs{v_{11}} \le r}$.
    Then
    \begin{equation}
        K\pars{U\pars{\bbone-\mcal{M}^{\mathrm{loc},\infty}_{\bbone,r}}}
        = K\big(U\Big(\pars{\bbone-\mcal{M}^{\mathrm{loc},\infty}_{\bbone,r}}\cap \mathsf{Max}_{r}\Big)\big) .
    \end{equation}
\end{lemma}
\begin{proof}
    Observe that $\norm{\bullet}_\infty$ and $\norm{\bullet}_{\mathrm{Fro}}$ are invariant under permutation and that for all permutation matrices $P_1, P_2$ and for all $v\in\mcal{M}^{\mathrm{loc},\infty}_{\bbone,r}$ it holds that $P_1vP_2\in\mcal{M}^{\mathrm{loc},\infty}_{\bbone,r}$.
    Moreover, for all $v\in\mcal{M}^{\mathrm{loc},\infty}_{\bbone,r}$ there exist permutation matrices $P_1, P_2$ such that $P_1 \pars{\bbone-v} P_2 \in \mathsf{Max}_{r}$.
    Therefore
    \begin{equation}
        \sup_{v\in\mcal{M}^{\mathrm{loc},\infty}_{\bbone,r}\setminus\braces{\bbone}} \frac{\norm{\bbone - v}_\infty^2}{\norm{\bbone - v}_{\mathrm{Fro}}^2}
        = \sup_{v\in\mcal{M}^{\mathrm{loc},\infty}_{\bbone,r}\setminus\braces{\bbone}} \frac{\norm{P_1\pars{\bbone - v}P_2}_\infty^2}{\norm{P_1\pars{\bbone - v}P_2}_{\mathrm{Fro}}^2}
        = \sup_{\substack{v\in\mcal{M}^{\mathrm{loc},\infty}_{\bbone,r}\setminus\braces{\bbone}\\\bbone-v\in\mathsf{Max}_{r}}} \frac{\norm{\bbone - v}_\infty^2}{\norm{\bbone - v}_{\mathrm{Fro}}^2} .
    \end{equation}
\end{proof}

\begin{proof}[Proof of \cref{prop:KU_bbone_bound}]
    \renewcommand{\vec}{\boldsymbol}
    Let $vw^\intercal \in \mcal{M}^{\mathrm{loc},\infty}_{\bbone,r}\setminus\braces{\bbone}$.
    By the previous lemma we may assume that $\bbone-vw^\intercal \in \mathsf{Max}_{r}$ and define
    $\vec{1} := \pars{1\ \ldots\ 1}^\intercal$,
    $l^* := \argmin_{l>1} \braces{\norm{\vec{1}-vw_l}_2}$ and $\widetilde{w} := \pars{w_1\ w_{l^*}\ \ldots\ w_{l^*}}^\intercal$. Then
    \begin{align}
        \norm{\bbone - vw^\intercal}_\infty
        &= \max_{\substack{k=1,\ldots,d_1 \\ l=1,\ldots,d_2}}\braces{\abs{1 - v_kw_l}}
        = \abs{1 - v_1w_1}
        = \max_{\substack{k=1,\ldots,d_1 \\ l=1,l^*}k}\braces{\abs{1 - v_kw_l}}
        = \norm{\bbone - v\widetilde{w}^\intercal}_\infty \\
        \intertext{implies $\bbone-v\widetilde{w}^\intercal\in\mathsf{Max}_{r}$ and since $0 < \norm{\bbone-v\widetilde{w}^\intercal}_\infty = \norm{\bbone-vw^\intercal}_\infty \le r$ also $v\widetilde{w}^\intercal\in\mcal{M}^{\mathrm{loc},\infty}_{\bbone,r}\setminus\braces{\bbone}$. Moreover, together with}
        \norm{\bbone - vw^\intercal}_{\mathrm{Fro}}^2
        &= \norm{\vec{1} - vw_1}_2^2 + \sum_{l=2}^{d_2} \norm{\vec{1} - vw_l}_2^2
        \ge \norm{\vec{1} - vw_1}_2^2 + \sum_{l=2}^{d_2} \norm{\vec{1} - vw_{l^*}}_2^2
        = \norm{\bbone - v\widetilde{w}^\intercal}_{\mathrm{Fro}}^2
    \end{align}
    it implies $K\pars{U\pars{\bbone - v\widetilde{w}^\intercal}} \ge K\pars{U\pars{\bbone - vw^\intercal}}$.
    We can now apply a similar argument to obtain $\widetilde{v}$ from $v$.
    Finally, observe that $M_{\alpha,\beta,\gamma} = \widetilde{v}\widetilde{w}^\intercal$ for $\alpha = \widetilde{v}_1\widetilde{w}_1$, $\beta = \widetilde{v}_2\widetilde{w}_1$ and $\gamma = \widetilde{w}_2/\widetilde{w}_1$.
\end{proof}

\end{document}